\documentclass[final,onefignum,onetabnum]{siamonline190516}
\usepackage{braket,amsfonts}

\usepackage{array}

\usepackage{pgfplots}

\newsiamthm{claim}{Claim}
\newsiamremark{remark}{Remark}
\newsiamremark{hypothesis}{Hypothesis}
\crefname{hypothesis}{Hypothesis}{Hypotheses}

\usepackage{algorithmic}

\usepackage{graphicx,epstopdf}

\Crefname{ALC@unique}{Line}{Lines}

\usepackage{amsopn}

\usepackage{xspace}
\usepackage{bold-extra}
\usepackage[most]{tcolorbox}

\colorlet{texcscolor}{blue!50!black}
\colorlet{texemcolor}{red!70!black}
\colorlet{texpreamble}{red!70!black}
\colorlet{codebackground}{black!25!white!25}

\usepackage[letterpaper,top=1.5in, bottom=1in, left=1in, right=1in]{geometry}

\usepackage{epsfig,epsf,fancybox}
\usepackage{amsmath}
\usepackage{mathrsfs}
\usepackage{amssymb}
\usepackage{amsfonts}
\usepackage{graphicx}
\usepackage{color}
\usepackage{boxedminipage}
\usepackage{stmaryrd}
\usepackage{multirow}
\usepackage{booktabs}
\usepackage{cite}
\usepackage[ruled,vlined,linesnumbered,algo2e]{algorithm2e}

\usepackage[shortlabels]{enumitem}
\usepackage{wrapfig} 

\def\endproof{\hfill $\Box$ \vskip 0.4cm}

\usepackage{tikz}
\usetikzlibrary{fit, arrows, shapes, positioning, shadows, matrix, calc}

\tikzstyle fwd=[line width=0.7pt, ->]   
\tikzstyle fwddash=[line width=0.7pt, dashed, ->]   
\tikzstyle bwd=[double, line width=0.3pt, ->]  
\tikzstyle refl=[double,  dashed, line width=0.2pt, ->]    

\tikzstyle{every node}=[font=\small] 

\tikzset{
  >=stealth', 
  invisible/.style={opacity=0}, 
  alt/.code args={<#1>#2#3}{\alt<#1>{\pgfkeysalso{#2}}{\pgfkeysalso{#3}}}, 
  visible on/.style={alt=#1{}{invisible}}, 
  smallnode/.style={circle, fill=black, thick, inner sep=1pt, minimum size=1.5pt}, 
  punkt/.style={
           rectangle,
           rounded corners,
           draw=black, very thick,
           text width=5em,
           minimum height=2em,
           text centered},
}

\tikzset{%
        brace/.style = { decorate, decoration={brace, amplitude=5pt} },
       mbrace/.style = { decorate, decoration={brace, amplitude=5pt, mirror} },
        label/.style = { black, midway, scale=0.5, align=center },
     toplabel/.style = { label, above=.5em, anchor=south },
    leftlabel/.style = { label,rotate=-90,left=.5em,anchor=north },   
  bottomlabel/.style = { label, below=.5em, anchor=north },
        force/.style = { rotate=-90,scale=0.4 },
        round/.style = { rounded corners=2mm },
       legend/.style = { right,scale=0.4 },
        nosep/.style = { inner sep=0pt },
   generation/.style = { anchor=base }
}

\usetikzlibrary{positioning}
\usetikzlibrary{shadows}
\usetikzlibrary{calc,positioning,shadows.blur,decorations.pathreplacing}
\usepackage{etoolbox}

\usepackage{subcaption}
\usepackage{fontawesome}

\usepackage{mathtools}

\usepackage[normalem]{ulem} 

\newcommand{\va}{{\mathbf{a}}}
\newcommand{\vb}{{\mathbf{b}}}

\newcommand{\vg}{{\mathbf{g}}}

\newcommand{\vm}{{\mathbf{m}}}

\newcommand{\vu}{{\mathbf{u}}}
\newcommand{\vv}{{\mathbf{v}}}
\newcommand{\vw}{{\mathbf{w}}}
\newcommand{\vx}{{\mathbf{x}}}
\newcommand{\vy}{{\mathbf{y}}}
\newcommand{\vz}{{\mathbf{z}}}

\newcommand{\vI}{{\mathbf{I}}}

\newcommand{\cE}{{\mathcal{E}}}

\newcommand{\EE}{\mathbb{E}} 
\newcommand{\RR}{\mathbb{R}} 
\newcommand{\vzero}{\mathbf{0}} 

\newcommand{\dist}{\mathrm{dist}}    
\newcommand{\Prob}{{\mathrm{Prob}}} 
\newcommand{\prox}{{\mathbf{prox}}} 
\newcommand{\dom}{{\mathrm{dom}}} 
\newcommand{\tr}{{\mathrm{tr}}} 
\newcommand{\Proj}{{\mathrm{Proj}}} 

\DeclareMathOperator*{\argmin}{arg\,min} 

\DeclareMathOperator*{\Min}{minimize}

\newcommand{\bc}{\begin{center}}
\newcommand{\ec}{\end{center}}

\newcommand{\bdm}{\begin{displaymath}}
\newcommand{\edm}{\end{displaymath}}

\newcommand{\beq}{\begin{equation}}
\newcommand{\eeq}{\end{equation}}

\newcommand{\bfl}{\begin{flushleft}}
\newcommand{\efl}{\end{flushleft}}

\newcommand{\bt}{\begin{tabbing}}
\newcommand{\et}{\end{tabbing}}

\newcommand{\beqn}{\begin{eqnarray}}
\newcommand{\eeqn}{\end{eqnarray}}

\newcommand{\beqs}{\begin{align*}} 
\newcommand{\eeqs}{\end{align*}}  

\newtheorem{assumption}{Assumption}

\numberwithin{equation}{section}

\lstdefinestyle{siamlatex}{%
  style=tcblatex,
  texcsstyle=*\color{texcscolor},
  texcsstyle=[2]\color{texemcolor},
  keywordstyle=[2]\color{texemcolor},
  moretexcs={cref,Cref,maketitle,mathcal,text,headers,email,url},
}

\tcbset{%
  colframe=black!75!white!75,
  coltitle=white,
  colback=codebackground, 
  colbacklower=white, 
  fonttitle=\bfseries,
  arc=0pt,outer arc=0pt,
  top=1pt,bottom=1pt,left=1mm,right=1mm,middle=1mm,boxsep=1mm,
  leftrule=0.3mm,rightrule=0.3mm,toprule=0.3mm,bottomrule=0.3mm,
  listing options={style=siamlatex}
}

\patchcmd\newpage{\vfil}{}{}{}
\begin{tcbverbatimwrite}{tmp_\jobname_header.tex}
\title{Distributed stochastic inertial-accelerated methods with delayed derivatives for nonconvex problems}

\author{Yangyang Xu\thanks{Department of Mathematical Sciences, Rensselaer Polytechnic Institute, Troy, NY (\email{xuy21@rpi.edu})}
\and Yibo Xu\footnotemark[1]
\and Yonggui Yan\footnotemark[1]
\and Jie Chen\thanks{MIT-IBM Watson AI Lab, IBM Research}
}

\headers{Distributed stochastic inertial-accelerated methods with delayed derivatives}{Yangyang Xu, Yibo Xu, Yonggui Yan, and Jie Chen}
\end{tcbverbatimwrite}
\input{tmp_\jobname_header.tex}

\begin{document}
\maketitle

\begin{tcbverbatimwrite}{tmp_\jobname_abstract.tex}
\begin{abstract}
   
   Stochastic gradient methods (SGMs) are predominant approaches for solving stochastic optimization.  
On smooth nonconvex problems, a few acceleration techniques have been applied to improve the convergence rate of SGMs. However, little exploration has been made on applying a certain acceleration technique to a stochastic subgradient method (SsGM) for nonsmooth nonconvex problems. In addition, few efforts have been made to analyze an (accelerated) SsGM with delayed derivatives. The information delay naturally happens in a distributed system, where computing  workers do not coordinate with each other.  

In this paper, we propose an inertial proximal SsGM for solving nonsmooth nonconvex stochastic optimization problems. The proposed method can have guaranteed convergence even with delayed derivative information in a distributed environment. Convergence rate results are established to three classes of nonconvex problems: weakly-convex nonsmooth problems with a convex regularizer, composite nonconvex problems with a nonsmooth convex regularizer, and smooth nonconvex problems. For each problem class, the convergence rate is $O(1/K^{\frac{1}{2}})$ in the expected value of the gradient norm square, for $K$ iterations.  
In a distributed environment, the convergence rate of the proposed method will be slowed down by the information delay. Nevertheless, the slow-down effect will decay with the number of iterations for the latter two problem classes. 
We test the proposed method on three applications. The numerical results clearly demonstrate the advantages of using the inertial-based acceleration. Furthermore, we observe higher parallelization speed-up in asynchronous updates over the synchronous counterpart, though the former uses delayed derivatives. Our source code is released at \url{https://github.com/RPI-OPT/Inertial-SsGM}
\end{abstract}


\end{tcbverbatimwrite}

\input{tmp_\jobname_abstract.tex}

\section{Introduction}

The stochastic approximation method is one popular approach for solving stochastic problems. It can date back to \cite{robbins1951stochastic} for solving root-finding problems. Nowadays, its first-order versions, such as the stochastic gradient method (SGM), have been extensively used to solve stochastic problems or deterministic problems that involve a huge amount of data (e.g., see \cite{nemirovski2009robust, sra2012optimization}). A standard (or vanilla) SGM often converges slowly. Several acceleration techniques have been used to improve its theoretical and/or empirical convergence speed (e.g., \cite{allen2017katyusha, johnson2013accelerating, duchi2011adaptive, xu2020momentum, tran2021hybrid}) for solving convex or smooth nonconvex problems. However, \emph{for nonsmooth nonconvex problems, it appears that it is still unknown whether a proximal SGM or a stochastic subgradient method (SsGM) can still have guaranteed convergence if a certain acceleration technique is applied}. 
In this paper, we give a positive answer to this open question by using an inertial-type acceleration technique, even if the derivative information can be delayed in a distributed environment. 

Our study focuses on stochastic optimization problems in the form of 
\begin{equation}\label{eq:stoc-prob}
\phi^*=\Min_{\vx\in \RR^n} ~\phi(\vx):=F(\vx)+r(\vx), \text{ with } F(\vx):= \EE_\xi [f(\vx;\xi)].
\end{equation}
Here, $\xi$ is a random variable that can represent a stochastic scenario or a data point, $F$ is often called a loss function or a data-fitting term, and $r$ can include a hard constraint and/or a soft regularization term. We will study a few problem classes, where $F$ is nonconvex and can be smooth or nonsmooth but $r$ is convex and nondifferentiable if it exists. 
As a special case, when $\xi$ is distributed on a finite (but possibly very large-scale) dataset, $F$ will reduce to a finite-sum structured function that appears in any application involving a pre-collected dataset. 

Applications in the form of \eqref{eq:stoc-prob} include the robust phase retrieval that has been used in imaging and speech processing \cite{eldar2014phase,duchi2019solving}, the blind deconvolution in astronomy and computer vision \cite{chan1998total,levin2011understanding}, the robust principal component analysis in image deconvolution \cite{chandrasekaran2011rank,candes2011robust}, the online nonnegative matrix factorization in image processing and pattern recognition \cite{guan2012online}, and the sparsity-regularized deep learning \cite{scardapane2017group}. Specific formulations of some applications are given in section~\ref{sec:numerical}.

\subsection{Proposed algorithm}

%
\begin{wrapfigure}{r}{0.35\textwidth}
	\vspace{-0.7cm}
	\begin{center}
	\begin{tikzpicture}[fill=blue!20, scale = 1.0]
		\path (0,0) node (master) [circle, draw, fill=green!60] {master}
		(0, 2) node (agg1) [rounded corners, rectangle, draw, fill=red!25] {worker}
		(-2.2, .50) node (agg2) [rounded corners, rectangle, draw, fill=red!25] {worker}
		(-1.2, -1.6) node (agg3)  [rounded corners, rectangle, draw, fill=red!25] {worker} 
		(1.2, -1.6) node (agg4) [rounded corners, rectangle, draw, fill=red!25] {worker}
		(2.2, .50) node (agg5) [rounded corners, rectangle, draw, fill=red!25] {worker};	
		
		\draw[ ->] (.1,1.75) -- (.1,0.65)
		node[pos = .5, rotate = -0, right]	 {\footnotesize $\vg$};		
		\draw[ ->] (-1.6,.6) -- (-0.6,.1)
		node[pos = .5, rotate = -26, above]	 {\footnotesize $\vg$};		
		\draw[->] (-1.3, -1.35) -- (-.45,-0.5) 
		node[pos = .5, rotate = 45, above]	 {\footnotesize $\vg$};
		\draw[->] (1.1,-1.35) -- (.3,-0.55) 
		node[pos = .5, rotate = -45, below]	 {\footnotesize $\vg$};
		\draw[ ->] (1.6,0.4) -- (0.6,-0.1)
		node[pos = .5, rotate = 26, below]	 {\footnotesize $\vg$};
		
		\draw[ ->] (-.1,0.65) -- (-.1,1.75)
		node[pos = .5, rotate = 0, left]	 {\footnotesize $\vx$};		
		\draw[ ->] (-0.6,-0.1) -- (-1.6,0.4)
		node[pos = .5, rotate = -26, below]	 {\footnotesize $\vx$};		
		\draw[->] (-.3,-0.55) --  (-1.1, -1.35)
		node[pos = .5, rotate = 45, below]	 {\footnotesize $\vx$};	
		\draw[->] (.45,-0.5) --  (1.3,-1.35)
		node[pos = .5, rotate = -45, above]	 {\footnotesize $\vx$};	
		\draw[ ->] (0.6,.1) -- (1.6,.6) 
		node[pos = .5, rotate = 26, above]	 {\footnotesize $\vx$};	
	\end{tikzpicture}	
\end{center}
\vspace{-0.2cm}	
\caption{{\small A master-worker architecture. The \emph{master} performs update to $\vx$; \emph{worker}s compute sample (sub)gradients.}}\label{fig:par-in-sgd}	
\vspace{-0.6cm}
\end{wrapfigure}
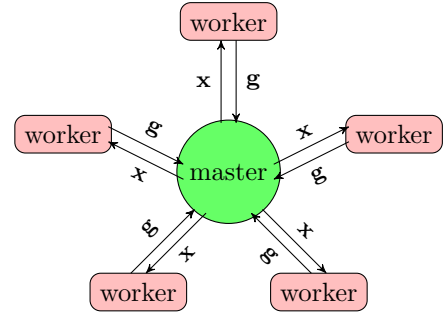

We propose to solve \eqref{eq:stoc-prob} in a distributed environment. Suppose there are multiple agents. One agent is designated as the \emph{master} and all the others as \emph{worker}s. The master performs update to $\vx$ while the workers compute sample (sub)gradients; see Fig.~\ref{fig:par-in-sgd} for an illustration. The master-worker architecture has been adopted in many works. It can naturally happen, either because data are collected from local devices and then sent to a central server for processing such as in a sensor network application \cite{muruganathan2005centralized}, or because the pre-collected dataset is too large to fit on a single machine and must be distributed over multiple machines. 

We assume that each worker can acquire samples of $\xi$ and compute the (sub)gradient of each sampled function $f(\,\cdot\,;\xi)$. Each worker sends its computed sample (sub)gradient $\vg$ to the master, and the latter updates $\vx$ by using its received sample (sub)gradients and then sends the updated $\vx$ to workers.  Our scheme is described in Alg.~\ref{alg:async-hvb-sgm}, which is from the master's point of view.

\begin{algorithm2e}
	\DontPrintSemicolon
	\caption{A distributed stochastic inertial subgradient method for \eqref{eq:stoc-prob}}
	\label{alg:async-hvb-sgm}
	\textbf{Initialization:} choose $\vx^{(0)}\in\dom(r)$ and set $\vx^{(1)}=\vx^{(0)}$\;
	\For{$k=1,2,\ldots$}{
		Let $\vg^{(k)} = \tilde\nabla f(\vx^{(k-\tau_k)}; \xi_k)$ computed by a worker, where $\xi_k$ is a sample of $\xi$ and $\tau_k$ measures the possible delay;\;
		Choose stepsize $\alpha_k>0$ and inertial parameter $\beta_k\ge0$;\;
		Update the variable $\vx$ by
		\begin{equation}\label{eq:update-x-2}
			\vx^{(k+1)}= \prox_{\alpha_{k} r}\left( \vx^{(k)} - \alpha_{k} \vg^{(k)} +\beta_k (\vx^{(k)}-\vx^{(k-1)})\right).
		\end{equation}
	}
\end{algorithm2e}

Here, $\tilde\nabla h(\vx)$ denotes a subgradient of a function $h$ at $\vx$, and it reduces to gradient if $h$ is differentiable at $\vx$.
In \eqref{eq:update-x-2}, the proximal mapping is defined as
\begin{equation}\label{eq:prox-def}
	\prox_{\alpha r} (\vx)= \argmin_{\vy\in \RR^n}\left\{ \textstyle r(\vy)+ \frac{1}{2\alpha} \|\vy-\vx\|^2 \right\}.
\end{equation}
We use $k$ to count the number of updates performed by the master. Notice that the master will update $\vx$ once it receives a sample (sub)gradient from one worker, and we do not enforce coordination between the workers. Hence, the $\vg^{(k)}$ used in \eqref{eq:update-x-2} may not be a sample (sub)gradient computed at $\vx^{(k)}$ but at an outdated iterate $\vx^{(k-\tau_k)}$. This setup with delayed information is the same as that in  \cite{agarwal2011distributed}. Also, instead of using a single sample, we can take multiple samples to compute $\vg^{(k)}$ as the average of the multiple sample (sub)gradients.

Consider a special case, where $f(\,\cdot\,;\xi)$ is differentiable for each $\xi$ and $r(\cdot)\equiv 0$. Then the update in \eqref{eq:update-x-2} becomes
$\vx^{(k+1)}=   \vx^{(k)} - \alpha_{k} \vg^{(k)} +\beta_k (\vx^{(k)}-\vx^{(k-1)})$. Let $\beta_k=\frac{\alpha_{k}}{\alpha_{k-1}}\beta$ for all $k\ge 1$ and for some $\beta\in(0,1)$. Define a recursive sequence by
\begin{equation}\label{eq:update-m} 
	\vm^{(k)}=\beta \vm^{(k-1)} + (1-\beta)\vg^{(k)}, \forall\, k\ge1, \text{ with }\vm^{(0)}=\vzero.	
\end{equation}
Then the $\vx$-update can be rewritten to 
\begin{equation}\label{eq:update-x}
\textstyle 	\vx^{(k+1)}=  \vx^{(k)} - \frac{\alpha_{k}}{1-\beta} \vm^{(k)},
\end{equation}
which is often referred as a momentum SGM in the literature (e.g., \cite{yan2016unified, gitman2019understanding})

\vspace{0.1cm}

\noindent\textbf{Why use inertial force or momentum?}~~Different from a standard proximal SsGM, we introduce an inertial force (or heavy-ball momentum term) $\beta_k (\vx^{(k)}-\vx^{(k-1)})$ in the update \eqref{eq:update-x-2}. If $\beta_k=0$, the update reduces to the standard proximal SsGM step. The heavy-ball momentum acceleration technique was first used in \cite{polyak1964some}. With the inertial force, a heavy-ball gradient method can mitigate the zigzagging behavior of a standard gradient descent method and potentially achieve faster convergence. For unconstrained strongly-convex quadratic optimization, it has been shown (cf. \cite{recht2010cs726}) that the heavy-ball gradient method can achieve an optimal convergence rate. The advantage of using inertia has also been studied for deterministic composite nonconvex problems and stochastic smooth nonconvex problems. For example, the work \cite{gitman2019understanding}  studies a more general momentum-based method, called Quasi-Hyperbolic Momentum (QHM), which includes the heavy-ball momentum as a special case. 
For unconstrained smooth problems,	 \cite{gitman2019understanding} gives a local linear convergence result that suggests the advantage of adding a heavy-ball momentum term in the update of a standard SGM. In addition, it provides supporting experiments to demonstrate that the optimal inertial parameter has a positive correlation with the condition number of the underlying problem. 
Although a heavy-ball momentum SGM has been extensively used in practice, 
a theoretical convergence guarantee is not yet achieved in the literature for nonconvex nonsmooth stochastic problems. We will provide a novel guideline of parameter setting for the inertial SGM or SsGM along with convergence guarantee, even if each $\vg^{(k)}$ is computed at an outdated iterate. It is worth mentioning that for unconstrained smooth problems, a heavy-ball momentum SGM and Nesterov's Accelerated Gradient (NAG)  are different special cases of QHM \cite{gitman2019understanding}. Though beyond the scope of this paper, our work may shed light on the acceleration effect of general momentum-based methods for nonsmooth nonconvex problems, such as QHM and NAG.

\subsection{Related works}


Our method has a few key ingredients, including ``stochastic subgradient'', ``inertia'', ``nonsmooth nonconvex'', and ``distributed delayed'', which differentiate our method from existing ones. Below we review prior methods that share some ingredients with ours. We list a few closely-related methods with corresponding ingredients in Table~\ref{table:related}. 

\begin{table}[h]
	\caption{ {\small A comparison of ingredients amongst several algorithms for solving problems in the form of \eqref{eq:stoc-prob}. In the second column, ``property of $F$'' is to reflect the underlying assumption of $F$:  ``w.c.'' for weak convexity, ``smooth'' for Lipschitz continuous gradient, and ``cvx'' for convexity.
	In the third column, ``inertia'' is to reflect whether the algorithm introduces inertia. In the fourth column, ``composite model'' is to reflect the existence of $r$ in \eqref{eq:stoc-prob}: ``proj.'' indicates a simple convex constraint, and ``prox.'' indicates a proximable regularizer.  
		In the fifth column, ``distributed delayed'' is to reflect whether the algorithm can handle a distributed setting with delayed (sub)gradient information. In the last column, convergence rate results for nonconvex models are listed: $\tau$ for the upper bound on the delay and $K$ for the total number of iterations. 
		} 
	} 
	
	\label{table:related}
	\begin{center}
		\resizebox{.98\textwidth}{!}{
			\begin{tabular}{|c|c|c|c|c|c|}
				\hline
				Method                 & property of $F$ & inertia & composite model & distributed delayed &  convergence rate  \\ \hline\hline
				Mirror Descent \cite{agarwal2011distributed}               & smooth \& cvx       & no                 &  no               &     yes    &  ---  \\
				AdaptiveRevision \cite{mcmahan2014delay} & smooth \& cvx & no & no & yes & --- \\
				Random Incremental Subgrad. \cite{nedic2001distributed} & cvx        & no                  &      proj.          &    yes    &  ---  \\
				AdaDelay \cite{sra2016adadelay} & smooth \& cvx & no & proj. & yes & --- \\
				AsySG-con \cite{lian2015asynchronous}       & smooth       & no                  &       no     & yes   & $(1+\tau /\sqrt{K})/\sqrt{K}$  \\\hline
				\multirow{ 2}{*}{APAM \cite{xu2020asynchronous} }      & smooth \& cvx      & yes                  &       proj.     & yes   & ---  \\
				       & smooth       & yes                  &       no     & yes   & $(1+\tau /K^{1/4}+\tau^2 /\sqrt{K})/\sqrt{K}$  \\\hline
				SHB \cite{mai2020convergence-ICML}       & w.c.        & yes                  &       proj.        &    no    & $1/\sqrt{K}$  \\\hline
				\multirow{ 3}{*}{This paper} & w.c.       & yes                 &       proj. \& prox.            &    yes   & $(1+\tau /\sqrt{K}+\tau)/\sqrt{K}$  \\
			 & smooth       & yes                 &       proj. \& prox.            &    yes   & $(1+\tau^2 /\sqrt{K})/\sqrt{K}$  \\
			 & smooth       & yes                 &       no           &    yes   & $(1+\tau /\sqrt{K})/\sqrt{K}$  \\[0.1cm]\hline       
			\end{tabular}
		}
	\end{center}
\end{table}



\vspace{0.1cm}

\noindent\textbf{Heavy-ball and inertial methods.}~~Early advances based on the heavy-ball or inertial momentum acceleration technique 
	can date back to \cite{polyak1964some, nesterov1983method}. 
For decades, researchers have been 
designing heavy-ball or inertial methods  for deterministic optimization \cite{zavriev1993heavy,ochs2014ipiano,ochs2015ipiasco,ghadimi2015global,liang2016multi,ochs2018local},  structured stochastic optimization \cite{loizou2017linearly,loizou2020momentum, polyak1987introduction,tseng1998incremental,sutskever2013importance,ghadimi2016accelerated}, and even in the framework of maximal monotone operators \cite{alvarez2001inertial,alvarez2004weak,moudafi2003convergence}.
Convergence analysis has been conducted to convex problems and also nonconvex problems. For a convex deterministic model, \cite{sun2019non, sun2020nonergodic} provide last-iterate convergence for inertial methods. 
For a convex stochastic model, \cite{nazin2018algorithms} proposes an inertial mirror descent method and establishes an $O(1/\sqrt{K})$ 
convergence rate result. 
%
Under a bounded-gradient assumption, \cite{yan2016unified} provides a unified convergence analysis of stochastic momentum methods for unconstrained smooth nonconvex stochastic optimization.  \cite{ghadimi2016accelerated} incorporates momentum acceleration in SGM and achieves an optimal oracle complexity result for \eqref{eq:stoc-prob} when $F$ is smooth.
The work \cite{sun2019heavy} studies how heavy-ball technique can help SGM escape saddle points. 


\vspace{0.1cm}

\noindent\textbf{Distributed/parallel stochastic methods with delayed (sub)gradient information.}~~There have been quite a few works about distributed delayed or asynchronous (async) parallel SGMs for convex or nonconvex problems and SsGMs for convex problems.
Similar to our method, \cite{agarwal2011distributed} also adopts a master-worker setup. It 
analyzes a distributed delayed SGM for convex problems and establishes a convergence rate of $O(\frac{1+\tau^2/\sqrt{K}}{\sqrt{K}})$, where $\tau$ denotes the maximum delay of stochastic gradient and $K$ is the total number of updates. 
Under a shared-memory setting, \cite{recht2011hogwild} proposes an async-parallel SGM 
 for strongly-convex problems with a special sparsity structure and establishes a convergence rate of $O(\frac{1+\tau^2/\sqrt{n}}{K}\ln K)$, where $n$ is the number of coordinates. 
\cite{mcmahan2014delay} gives delay-tolerant algorithms for async distributed convex online learning problems. Its algorithms can achieve 
a regret of $O(\sqrt{ (1+\tau) K })$  if a uniform upper bound $\tau$ on the delay is known and $O((1+\tau) \sqrt{ K })$ otherwise. 
For smooth convex stochastic problems, \cite{backstrom2019mindthestep, sra2016adadelay} adapt the stepsize of an async-parallel SGM to the staleness of stochastic gradient. 
More precisely, let  $\tau_k$ denote the actual delay at iteration $k$. 
The stepsize of the methods in \cite{sra2016adadelay,backstrom2019mindthestep} depends on 
$\tau_k$.
\cite{sra2016adadelay} analyzes its projected {stochastic} gradient scheme under the assumption that the delay has 
a bounded expectation $\EE[\tau_k]=\bar\tau < \infty$ and a bounded second moment $\EE[\tau_k^2]=\Omega(\bar\tau^2)$. The convergence rate is $O(\frac{\sqrt{1+\bar\tau} +\bar\tau^4/\sqrt{K}}{\sqrt{K}})$ if $\bar\tau$ is known and $O(\frac{1+\bar\tau +\bar\tau^4/\sqrt{K}}{\sqrt{K}})$ otherwise. 
{Under the assumption $\EE[\tau_k]=\bar\tau$,} \cite{backstrom2019mindthestep} achieves a rate of $O(\frac{1+\bar\tau^2/K}{K}\ln K)$ for {unconstrained} strongly convex problems. 

 Async-parallel SGMs have also been studied for smooth non-convex problems.
For example, \cite{lian2015asynchronous} analyzes an async-parallel SGM for unconstrained stochastic problems 
 and obtains a convergence rate of $O(\frac{1+\tau /\sqrt{K}}{\sqrt{K}})$ in terms of the expected value of gradient norm square;
\cite{huo2017asynchronous} analyzes an async-parallel variance-reduced SGM for a finite-sum structured problem and shows a sublinear convergence when $\tau=O(1)$;  \cite{xu2020asynchronous} focuses on async distributed and parallel adaptive (i.e., quasi-Newton-type) SGM for unconstrained stochastic problems and gives a convergence rate of $O(\frac{1+\tau /K^{1/4}+\tau^2/\sqrt{K}}{\sqrt{K}})$.
The studies on delayed SsGMs are still limited and only for convex problems. For example, \cite{nedic2001distributed} proposes an async projected SsGM and shows an almost-sure subsequence convergence result but with no convergence rate result.

The distributed/parallel methods mentioned above either adopt a master-worker setup (i.e., centralized) or assume a shared-memory setting. Many other works about SGMs or SsGMs are built on a decentralized setting, where multiple agents are distributed on a connected network and can only communicate with their neighbors but not a central master agent. Extending our discussions to the decentralized setting is beyond the scope of this paper. The interested readers can refer to \cite{wang2014cooperative, doan2018convergence, masubuchi2014distributed, lian2018asynchronous} and the references therein.

\vspace{0.1cm}

\noindent\textbf{Most closely-related works.}~~The methods in \cite{chen2020distributed, mai2020convergence-ICML} are perhaps the most closely related to ours. \cite{chen2020distributed} gives a decentralized projected deterministic subgradient method for weakly-convex optimization. It establishes a sublinear convergence result for the deterministic method. A stochastic variant is also given in \cite{chen2020distributed} with  subsequence convergence but no convergence rate.
In comparison to \cite{chen2020distributed}, we incorporate the inertial-force acceleration in a proximal SsGM to achieve empirically faster convergence, and in addition, we allow for delayed subgradient and can still achieve sublinear convergence. 
\cite{mai2020convergence-ICML} proposes a projected inertial SsGM for weakly-convex stochastic optimization. The method appears similar to Alg.~\ref{alg:async-hvb-sgm}. 
However, its analysis is completely different from ours, and it does not consider the delayed case. 
More importantly, its theoretical result is not established on the inertial-generated sequence. This is explained as follows. 
The update of the method in \cite{mai2020convergence-ICML} is 
\begin{equation}\label{eq:mai2020}
\vx^{(k+1)}= \Proj_{X}\left( \vx^{(k)} - \alpha\beta\vg^{(k)} +(1-\beta) (\vx^{(k)}-\vx^{(k-1)})\right),
\end{equation} 
where $\Proj_{X}$ denotes the projection onto a closed convex set $X$. Its analysis is only on the choice of 
$\alpha\beta=\Theta(\frac{1}{K})$ for a given maximum number $K$ of updates and 
$1-\beta=1-\frac{1}{\sqrt{K}}$. 
The sequence generated from \eqref{eq:mai2020} is similar to that we generate from \eqref{eq:update-x-2}, i.e., inertial-generated sequence. However, the theoretical result in \cite{mai2020convergence-ICML} is not about $\{\vx^{(k)}\}$ but on the extrapolated sequence 
$\{\bar{\vx}^{(k)}:=\vx^{(k)}+\frac{1-\beta}{\beta}(\vx^{(k)}-\vx^{(k-1)})\}$. There are two potential issues on analyzing the property of $\{\bar\vx^{(k)}\}$. First, if $X\neq \RR^n$, the sequence may not be in $X$. In fact, $\bar\vx^{(k)}$ can be far away from $X$ if $\vx^{(k)}-\vx^{(k-1)}\neq\vzero$ as $\frac{1-\beta}{\beta}=\sqrt{K}-1$ is big. 
Second, 
if $X=\RR^n$, it holds 
	$\bar{\vx}^{(k+1)}=  \bar{\vx}^{(k)} - \alpha\vg^{(k)}$, and in this case, $\{\bar{\vx}^{(k)}\}$ is more like a non-inertial sequence, as compared to the sequence generated by the momentum SGM in \eqref{eq:update-x}. In contrast, our analysis will be on the inertial-generated sequence. 

\subsection{Contributions}
%

\begin{itemize}
\item We propose 
a proximal inertial stochastic subgradient method in Alg.~\ref{alg:async-hvb-sgm} for solving non-convex stochastic Problem~\eqref{eq:stoc-prob}. The method can tolerate a delay of derivative information in a distributed environment. To the best of our knowledge, it is the first method that applies the inertial-acceleration technique in a proximal stochastic subgradient method for non-convex problems.  
	
\item We provide convergence rate analysis of the proposed method for three problem classes in the form of \eqref{eq:stoc-prob}. For each problem class, the method, with an appropriate setting of parameters, enjoys an $ O(\frac{1}{\sqrt{K}})$ convergence rate in terms of the expected value of a gradient norm square, where $K$ is the number of total iterations.
First, when $F$ is weakly-convex (see Def.~\ref{def:weak-cvx} below) and possibly nondifferentiable and $r$ is convex, 
we establish the $ O(\frac{1}{\sqrt{K}})$ convergence rate by choosing   
$\alpha_k= \Theta(\frac{1}{\sqrt{K}})$ and $\beta_k=\Theta(\frac{1}{K^{1/4}}), \forall\, k\le K$, provided that the delay $\tau_k$ follows a static distribution and is bounded by $\tau=O(1)$. 
Second, when $F$ is smooth but possibly non-convex and $r$ is convex, we obtain the $ O(\frac{1}{\sqrt{K}})$ convergence rate by the same choice of $\alpha_k$ and $\beta_k$ as in the first case, under a relaxed condition on $\tau_k$, i.e., $\tau_k=O(K^{1/4})$ for all $k$. 
Third, for the case of a smooth $F$ and $r\equiv 0$, we obtain the $O(\frac{1}{\sqrt{K}})$ convergence rate with the choice of $\alpha_k=\Theta(\frac{1}{\sqrt{K}})$ and $\beta_k=\beta\in (0,1),\forall\, k\le K$, provided that $\tau_k=O(\sqrt{K})$ for all $k$. Hence, the proposed method can tolerate a larger delay if the problem has a nicer structure.

\item We conduct numerical experiments of the proposed method on three applications to demonstrate the effect of the inertial acceleration and also to demonstrate the higher parallelization speed-up by the asynchronous implementation over a synchronous counterpart. 
\end{itemize}

\subsection{Notation and organization}

We use lower-case bold letters $\vx,\vy,\ldots$ for vectors. 
 A superscript $^{(k)}$ is used to specify the iterate, i.e., $\vx^{(k)}$ denotes the $k$-th iterate. 
We use $\|\cdot\|$ to denote the Euclidean norm of a vector and also the spectral norm of a matrix. 
We use the big-$O$ notation with the standard meaning to compare two quantities that can both approach to infinity or zero. 
The randomness of Alg.~\ref{alg:async-hvb-sgm} comes from the samples $\{\xi_k\}_{k\ge 1} $. In our analysis, we use $\EE_{k}$ for the conditional expectation with the history until the $k$-th iteration, i.e., $\EE_{k}[\,\cdot\,]=\EE\left[\,\cdot\,|\,\{\xi_j\}_{j=1}^{k-1}\right]$.

The rest of the paper is organized as follows. In section~\ref{sec:prelim}, we give some basic concepts and preliminary results. The detailed analysis and convergence rate results are shown in section~\ref{sec:prox-sub}-\ref{sec:smooth} for three different problem classes. Numerical results are given in section~\ref{sec:numerical}. Finally, section~\ref{sec:conclusion} concludes the paper.

\section{ Preliminaries}\label{sec:prelim}
In this section, we give some basic concepts and preliminary results that will be used in our analysis. 
For a function $\phi:\RR^n\rightarrow\RR\cup\{\infty\}$, we let $\partial\phi(\vx)$ denote its subdifferential at $\vx$, i.e., the set of subgradients, which consists of all vectors $\vv$ satisfying
\[\phi(\vy)\ge\phi(\vx)+\left\langle \vv,\vy-\vx\right\rangle + o\left(  \left\| \vy-\vx \right\|\right) \quad \textup{as } \vy\rightarrow\vx .\] 
The definition and results below can be found in \cite{drusvyatskiy2019efficiency, davis2019stochastic}.

\begin{definition}\label{def:weak-cvx}
	A function $\phi$ is $\rho$-weakly convex if $\phi(\cdot) + \frac{\rho}{2}\|\cdot\|^2$ is convex for some $\rho > 0$. 
\end{definition}

\begin{lemma}
	If $\phi$ is $\rho$-weakly convex, then 
\begin{equation}
		\label{def:weakly}
		\phi(\vy)\ge\phi(\vx)+\left\langle \vv,\vy-\vx\right\rangle -{\textstyle \frac{\rho}{2}} \left\| \vy-\vx \right\|  ^2, \forall\, \vx,\,\vy\in\dom(\phi), \, \forall\, \vv\in\partial\phi(\vx),
	\end{equation}
and		
	\begin{equation}
		\label{def:monotone}
		\left\langle \vv-\vw,\vx-\vy\right\rangle\ge - \rho  \left\| \vy-\vx \right\|  ^2, \forall\, \vx,\,\vy\in\dom(\phi), \, \forall\, \vv\in\partial\phi(\vx), \vw\in\partial\phi(\vy).
	\end{equation}
\end{lemma}

The class of weakly-convex functions is rather big. It includes all convex functions and all smooth functions. In addition, the composition function $h(c(\vx)) $ is also weakly-convex, if $h:\RR^m\rightarrow \RR$ is convex and Lipschitz continuous and $c:\RR^n\rightarrow\RR^m$ is smooth. 
Specific applications that have weakly-convex objectives include nonlinear least squares, phase retrieval, robust PCA, robust low rank matrix recovery, optimization of the Conditional Value-at-Risk, and graph synchronization. 
More examples can be found in \cite{drusvyatskiy2019efficiency}.


A key tool used in recent works (e.g., \cite{davis2019stochastic, alacaoglu2020convergence, nazari2020adaptive, mai2020convergence-ICML, chen2020distributed}) about stochastic weakly-convex minimization 
is the Moreau envelope \cite{MR201952}, which is defined as follows. 

\begin{definition}
For a $\rho$-weakly convex function $\phi$ and $\lambda \in(0, 1/\rho) $, the Moreau envelope $\phi_\lambda(\cdot)$ is defined as 
	\begin{equation}	\phi_\lambda(\vx)=\min_{\vy}\left\{\textstyle \phi(\vy)+\frac{1}{2\lambda}\left\| \vy-\vx\right\| ^2\right\}.\label{eq:moreau}
	\end{equation}

\end{definition}


The Moreau envelope is useful to characterize near-stationarity of a point $\vx$ because of the results in the following lemma. From \eqref{eq:rel-x-tildex}, we notice that if $\| \nabla\phi_\lambda(\vx)\|$ is small, then $\widetilde\vx:=\prox_{\lambda\phi}(\vx)$ will be a near-stationary point of $\phi$ and $\vx$ is close to $\widetilde\vx$.
\begin{lemma}\label{lem:prox-grad}
	Let $\phi$ be $\rho$-weakly convex, then for any $\lambda\in(0,1/\rho)$, the Moreau envelope $\phi_\lambda$ is smooth with gradient given by \[\nabla\phi_\lambda(\vx)=\lambda^{-1}\big(\vx-\widetilde\vx\big),\] 
where $\widetilde{\vx}:=\prox_{\lambda\phi}(\vx)$. Moreover,  
\begin{equation}\label{eq:rel-x-tildex}
\left\| \vx-\widetilde{\vx}\right\| =\lambda\left\| \nabla\phi_\lambda(\vx)\right\|,\ \phi(\widetilde{\vx})\le\phi(\vx),\ \text{ and }\ \dist(\vzero,\partial\phi(\widetilde{\vx}))\le\left\| \nabla\phi_\lambda(\vx)\right\|. 
\end{equation}
\end{lemma}

Besides the class of weakly-convex functions, we will also consider smooth functions in our analysis, for which we are able to obtain stronger theoretical results. By slightly abusing the notation, we also use $\rho$ to denote the Lipschitz constant of a smooth function, as a $\rho$-smooth function must be $\rho$-weakly convex.
\begin{definition}
	A function $\phi$ is $\rho$-smooth, if it is differentiable, and 
  \[\|\nabla \phi(\vx)-\nabla \phi(\vy)\|\leq \rho \|\vx-\vy\|, \forall\, \vx,\vy\in \RR^n.\]
\end{definition}
If $\phi$ is $\rho$-smooth, then   
\begin{equation}\label{eq:ineq-smooth}
	|\phi(\vx)-\phi(\vy)-\langle\nabla \phi(\vy),\vx-\vy \rangle|\leq{\textstyle\frac{\rho}{2}}\|\vx-\vy\|^2,\, \forall\, \vx,\vy\in \RR^n.
\end{equation}

\section{Convergence analysis for nonsmooth weakly-convex problems}\label{sec:prox-sub}
In this section, we analyze Alg.~\ref{alg:async-hvb-sgm} for problems in the form of \eqref{eq:stoc-prob}, where  
$F$ is possibly nondifferentiable. Throughout this section, we make the following assumptions. 

\begin{assumption}[weak convexity]\label{assump:weak-cvx}
	$F$ is $\rho$-weakly convex with $\rho>0$.
\end{assumption}

\begin{assumption}[unbiased subgradient]\label{assump:unbiased}
	$\vg^{(k)}$ is an unbiased stochastic subgradient of $F$ at $\vx^{(k-\tau_k)}$ for each $k$, i.e., $\EE_{\xi_k}[\vg^{(k)}]\in\partial F(\vx^{(k-\tau_k)})$.
\end{assumption}


\begin{assumption}[bounded subgradient]\label{assump:subgrad-bound}
	There is a real number $M \ge 0$ such that  $\EE_\xi \| \tilde\nabla f(\vx;\xi)\|^2\le M ^2$ for all $\vx\in\dom(r)$ and all subgradient $\tilde\nabla f(\vx;\xi)\in\partial f(\vx;\xi)$. 
\end{assumption}

\subsection{Preparatory lemmas}

For a fixed $\overline{\rho}>\rho$, we denote 
\begin{subequations}\label{eq:def-xvv}
	\begin{align}
		\vv^{(k)}=\EE_{\xi_k}\big[\vg^{(k)}\big] \in\partial F(\vx^{(k-\tau_k)}),\quad \widetilde\vx^{(k)}=\prox_{\phi/\overline{\rho}}(\vx^{(k)}),\label{eq:def-xtilde}\\ 
 \intertext{and choose}
		\widetilde\vv^{(k)}\in\partial F(\widetilde\vx^{(k)}) ~\text{such that}~ \overline{\rho}( \vx^{(k)}-\widetilde\vx^{(k)})\in\partial r(\widetilde\vx^{(k)})+ \widetilde\vv^{(k)}. \label{eq:def-vtilde}
	\end{align}
\end{subequations}
Note that the existence of $\widetilde\vv^{(k)}$ is guaranteed from the definition of $\widetilde\vx^{(k)}$. By Assumption~\ref{assump:subgrad-bound}, it holds that 
\begin{equation}\label{eq:subgrad-2ndmoment}
	\EE_{\xi_k} \|\vg^{(k)}\|^2\le M ^2,\quad \| \vv^{(k)}\|^2\le M ^2,  ~\textup{and}~   \|\widetilde\vv^{(k)}\|^2\le M ^2.
\end{equation}

The next result is from \cite[Lemma 3.2]{davis2019stochastic}. Its proof only relies on the definition of $\widetilde\vx^{(k)}$ and the choice of $\widetilde\vv^{(k)}$. Hence, the result still holds for our case, though the algorithm in \cite[Lemma 3.2]{davis2019stochastic} does not have an inertial term in its update.

\begin{lemma}\label{lem:xwidetilde}
	Let $\widetilde\vx^{(k)}$ and $\widetilde\vv^{(k)}$ be defined as in \eqref{eq:def-xtilde} and \eqref{eq:def-vtilde}. Then
	\begin{equation}\label{eq:relate-xtilde}
		\widetilde\vx^{(k)} =\prox_{\alpha_{k} r}\big(\alpha_{k}\overline{\rho}\vx^{(k)}-\alpha_{k}\widetilde\vv^{(k)}+(1-\alpha_{k}\overline{\rho})\widetilde\vx^{(k)} \big) .
	\end{equation}
\end{lemma}

The next lemma extends the hypomonotonicity property of a weakly-convex function, in order to deal with the case with delayed subgradients.

\begin{lemma}
	Let $\widetilde\vx^{(k)}$, $\vv^{(k)}$ and $\widetilde\vv^{(k)}$ be defined as in \eqref{eq:def-xvv}. Then under Assumption~\ref{assump:weak-cvx}, it holds
	\begin{equation}\label{eq:bd-crs-term}
		\begin{aligned}
			&~-\big\langle    \vx^{(k)}-\widetilde\vx^{(k)}, \vv^{(k)} -\widetilde\vv^{(k)} \big\rangle \\
			\le & ~\textstyle F(\vx^{(k)}) - F(\vx^{(k-\tau_k)}) + \frac{\rho}{2}\|\vx^{(k)}-\widetilde\vx^{(k)}\|^2+ \frac{\rho}{2}\|\vx^{(k-\tau_k)}-\widetilde\vx^{(k)}\|^2-\big\langle    \vx^{(k)}-\vx^{(k-\tau_k)} ,\vv^{(k)} \big\rangle.
		\end{aligned}
	\end{equation}
\end{lemma}

\begin{proof}
	From the $\rho$-weak convexity of $F$, it follows that
	\begin{equation}\label{eq:term-1-vtilde}
	\textstyle	\big\langle    \vx^{(k)}-\widetilde\vx^{(k)}, \widetilde\vv^{(k)} \big\rangle \le F(\vx^{(k)})-F(\widetilde\vx^{(k)}) + \frac{\rho}{2}\|\vx^{(k)}-\widetilde\vx^{(k)}\|^2,
	\end{equation}
	and
	\begin{equation}\label{eq:term-2-v}
		\textstyle-\big\langle    \vx^{(k-\tau_k)}-\widetilde\vx^{(k)}, \vv^{(k)} \big\rangle \le F(\widetilde\vx^{(k)}) - F(\vx^{(k-\tau_k)}) + \frac{\rho}{2}\|\vx^{(k-\tau_k)}-\widetilde\vx^{(k)}\|^2.
	\end{equation}
	Hence, we obtain the desired result by adding the two inequalities in \eqref{eq:term-1-vtilde} and \eqref{eq:term-2-v}, and also noticing 
	$$-\big\langle    \vx^{(k)}-\widetilde\vx^{(k)} ,\vv^{(k)} -\widetilde\vv^{(k)} \big\rangle = \big\langle    \vx^{(k)}-\widetilde\vx^{(k)}, \widetilde\vv^{(k)} \big\rangle -\big\langle    \vx^{(k-\tau_k)}-\widetilde\vx^{(k)}, \vv^{(k)} \big\rangle -\big\langle    \vx^{(k)}-\vx^{(k-\tau_k)} ,\vv^{(k)} \big\rangle.$$
	This completes the proof.
\end{proof}

The result in the next lemma establishes a descent property of the iterate sequence from Alg.~\ref{alg:async-hvb-sgm} by relating it to the virtual sequence $\{\widetilde\vx^{(k)}\}$. It extends the result in \cite[Lemma 3.3]{davis2019stochastic}.

\begin{lemma}\label{lem:keybound2}
	Let $\overline{\rho}\in (\rho,2\rho]$ and $\alpha_{k}\in (0, 1/\overline{\rho}]$ for all $k$. Under Assumptions~\ref{assump:weak-cvx}--\ref{assump:subgrad-bound}, the iterate sequence $\{\vx^{(k)}\}$ from Alg.~\ref{alg:async-hvb-sgm} with stepsize sequence $\{\alpha_k\}$ and inertial parameter $\{\beta_k\}$ satisfies 
	\begin{equation}\label{eq:keybound2}
		\begin{aligned}
			\EE_{\xi_k}  \|\vx^{(k+1)}-\widetilde\vx^{(k)} \|^2  
			\le &~  \big(  1- 2\alpha_{k}(\overline{\rho}-\rho)   +c_k \big)  \| \vx^{(k)}-\widetilde\vx^{(k)}   \|^2  +(\textstyle 2+\frac{1}{c_k})\beta_k^2 \|\vx^{(k)}-\vx^{(k-1)}\|^2\\
			&~ +8\alpha_{k}^2 M ^2+ 2\alpha_k(1-\alpha_{k}\overline{\rho})
			\widehat\cE_k, 
		\end{aligned}
	\end{equation}
	where $\widetilde\vx^{(k)}$ is defined in \eqref{eq:def-xtilde}, $c_k$ is any positive number, and
	\begin{equation}\label{eq:delay-error-term-hat}
		\widehat \cE_k :=  \textstyle F(\vx^{(k)}) - F(\vx^{(k-\tau_k)}) - \frac{\rho}{2}\|\vx^{(k)}-\widetilde\vx^{(k)}\|^2+ \frac{\rho}{2}\|\vx^{(k-\tau_k)}-\widetilde\vx^{(k)}\|^2 -\big\langle    \vx^{(k)}-\vx^{(k-\tau_k)} ,\vv^{(k)} \big\rangle. 
	\end{equation}
\end{lemma}

The next lemma will be used to bound $\sum_{k=1}^K \|\vx^{(k+1)} - \vx^{(k)}\|^2$ for any given integer $K$.
\begin{lemma}
	Let $\{\vx^{(k)}\}$ be generated from Alg.~\ref{alg:async-hvb-sgm}. Under Assumptions~\ref{assump:weak-cvx} and  \ref{assump:subgrad-bound}, it holds for any $\gamma > 0$ that
	\begin{equation}\label{eq:bd-x-diff0}
		\big(\textstyle 1 -\gamma -\frac{\alpha_k\rho}{2} - \frac{\beta_k}{2}\big)\EE\|\vx^{(k+1)} - \vx^{(k)}\|^2 \le \alpha_{k} \EE\big(\phi(\vx^{(k)}) - \phi(\vx^{(k+1)}) \big) + \textstyle \frac{\beta_k}{2}  \EE\|\vx^{(k)}-\vx^{(k-1)}\|^2
		+\frac{\alpha_{k}^2M ^2}{\gamma}.
	\end{equation}
	
\end{lemma}

\begin{proof}
	By the convexity of $r$, we have $\big\langle \vx^{(k)} - \vx^{(k+1)}, \tilde\nabla r(\vx^{(k+1)}) \big\rangle \le r(\vx^{(k)}) - r(\vx^{(k+1)})$. In addition, it follows from \eqref{eq:update-x-2} that $\vzero\in \alpha_{k} \partial r(\vx^{(k+1)})+ \vx^{(k+1)} -\vx^{(k)} + \alpha_{k} \vg^{(k)} -\beta_k (\vx^{(k)}-\vx^{(k-1)})$. Hence, 
	\begin{equation}\label{eq:bd-x-diff}
		\big\langle \vx^{(k+1)} - \vx^{(k)},  \vx^{(k+1)} -\vx^{(k)} + \alpha_{k} \vg^{(k)} -\beta_k (\vx^{(k)}-\vx^{(k-1)})\big\rangle \le \alpha_{k} \big(r(\vx^{(k)}) - r(\vx^{(k+1)}) \big).
	\end{equation}
	By the $\rho$-weak convexity of $F$, it holds
	\begin{equation*}
	\textstyle	\big\langle \vx^{(k+1)} - \vx^{(k)}, \tilde\nabla F(\vx^{(k+1)}) \big\rangle\ge F(\vx^{(k+1)}) - F(\vx^{(k)}) - \frac{\rho}{2}\|\vx^{(k+1)} -\vx^{(k)}\|^2,
	\end{equation*}
	and thus
	\begin{equation}\label{eq:bd-x-g}
\begin{aligned}
\big\langle \vx^{(k+1)} - \vx^{(k)}, \alpha_{k} \vg^{(k)}\big\rangle \ge &~ \alpha_{k} \big\langle \vx^{(k+1)} - \vx^{(k)}, \vg^{(k)}- \tilde\nabla F(\vx^{(k+1)})\big\rangle\\
&~+ \alpha_{k}\big(F(\vx^{(k+1)}) - F(\vx^{(k)}) - \frac{\rho}{2}\|\vx^{(k+1)} -\vx^{(k)}\|^2\big).
\end{aligned}
	\end{equation}
	Plugging \eqref{eq:bd-x-g} into \eqref{eq:bd-x-diff} and rearranging terms give
	\begin{equation}\label{eq:bd-x-diff2}
		\begin{aligned}
			\big(\textstyle 1 -\frac{\alpha_k\rho}{2}\big)\|\vx^{(k+1)} - \vx^{(k)}\|^2 \le &~\alpha_{k} \big(\phi(\vx^{(k)}) - \phi(\vx^{(k+1)}) \big) + \beta_k \big\langle \vx^{(k+1)} - \vx^{(k)}, \vx^{(k)}-\vx^{(k-1)} \big\rangle\\
			&~ - \alpha_{k} \big\langle \vx^{(k+1)} - \vx^{(k)}, \vg^{(k)}- \tilde\nabla F(\vx^{(k+1)})\big\rangle.
		\end{aligned}
	\end{equation}
	Now using Assumption~\ref{assump:subgrad-bound} and the Young's inequality, we have
	\begin{equation*}
		\begin{aligned}
			\big(\textstyle 1 -\frac{\alpha_k\rho}{2}\big)\EE\|\vx^{(k+1)} - \vx^{(k)}\|^2 \le &~\alpha_{k} \EE\big(\phi(\vx^{(k)}) - \phi(\vx^{(k+1)}) \big) + \textstyle \frac{\beta_k}{2} \EE\big( \|\vx^{(k+1)} - \vx^{(k)}\|^2 + \|\vx^{(k)}-\vx^{(k-1)}\|^2 \big)\\
			&~ +\textstyle \gamma\EE\|\vx^{(k+1)} - \vx^{(k)}\|^2 + \frac{\alpha_k^2M ^2}{\gamma}.
		\end{aligned}
	\end{equation*}
	Rearranging terms in the above inequality gives the desired result.
\end{proof}

\subsection{Convergence rate results}
In this subsection, we establish the convergence rate results of Alg.~\ref{alg:async-hvb-sgm} for nonsmooth weakly-convex problems by using the lemmas in the previous subsection. We first give a generic result as follows. 

\begin{theorem}\label{thm:prox-sync}
	Given a positive integer $K$, let $\{\vx^{(k)}\}_{k=1}^K$ 
	be generated from Alg.~\ref{alg:async-hvb-sgm} with a stepsize sequence $\{\alpha_{k}\}$ and inertial parameter sequence $\{\beta_k\}$. Under Assumptions~\ref{assump:weak-cvx}--\ref{assump:subgrad-bound}, let $\overline \rho \in (\rho, 2\rho]$ and assume $\alpha_k\in (0, 1/\overline\rho]$ for all $k$.
	Then  
	\begin{equation}\label{eq:nsm-general}
		\begin{aligned}
			\EE\big\|\nabla\phi_{1/\overline{\rho}} (\vx^{(T)} )  \big\|^2  \le & \textstyle\frac{2\overline{\rho}}{(\overline{\rho}-\rho)\sum_{k=k_0}^{K} \alpha_{k}} 
			\Big[ \EE\big[\phi_{1/\overline{\rho}}(\vx^{(k_0)})  -\phi^*\big]  + \frac{\overline{\rho}}{2}\sum_{k=k_0}^{K}( 2+\frac{2}{\alpha_{k}(\overline{\rho}-\rho)}) \beta_k^2 \EE\|\vx^{(k)}-\vx^{(k-1)}\|^2  \\
			& \hspace{2.5cm}\textstyle+ 4\overline{\rho} M ^2\sum_{k=k_0}^{K}\alpha_{k}^2+\sum_{k=k_0}^{K} \alpha_k \overline{\rho} (1-\alpha_{k}\overline{\rho})\EE[\cE_k]
			\Big],
		\end{aligned}
	\end{equation}
	where $k_0\ge1$ is an integer, 
	$T$ is randomly selected from $\{k_0,\ldots, K\}$ by the distribution	
	\begin{equation}\label{eq:def-T}
	\textstyle	\Prob(T=k) = \frac{\alpha_{k}}{\sum_{j=k_0}^{K}\alpha_{j}},\,\forall\, k=k_0,\ldots, K,
	\end{equation}
and
\begin{equation}\label{eq:delay-error-term}
		\cE_k :=  \textstyle F(\vx^{(k)}) - F(\vx^{(k-\tau_k)}) +  \big(\frac{\rho}{2} + \frac{\rho^2}{\overline\rho - \rho}\big) \|\vx^{(k-\tau_k)}-\vx^{(k)}\|^2 -\big\langle    \vx^{(k)}-\vx^{(k-\tau_k)} ,\vv^{(k)} \big\rangle. 
	\end{equation}	
\end{theorem}
\begin{proof}
	By the definition of $\phi_\lambda$ in \eqref{eq:moreau} and Lemma~\ref{lem:keybound2}, we have
	\begin{align}\label{eq:phi-ineq1}
		&~\EE_{\xi_k} \big[ \phi_{1/\overline{\rho}}(\vx^{(k+1)})\big] \nonumber\\
		\le&~\textstyle\EE_{\xi_k} \big[  \phi (\widetilde\vx^{(k)}) + \frac{\overline{\rho}}{2}\|\vx^{(k+1)}-\widetilde\vx^{(k)}\|^2\big] \nonumber\\
		\le&~\textstyle \phi (\widetilde\vx^{(k)}) +\frac{\overline{\rho}}{2} \big[ \big(1- 2\alpha_{k}(\overline{\rho}-\rho)  + c_k \big)  \| \vx^{(k)}-\widetilde\vx^{(k)}   \|^2  +( 2+\frac{1}{c_k})\beta_k^2 \|\vx^{(k)}-\vx^{(k-1)}\|^2  
		+8\alpha_{k}^2 M ^2 \big] \nonumber\\
		&+\alpha_k \overline{\rho} (1-\alpha_{k}\overline{\rho})\widehat\cE_k \nonumber \\
		=&~\textstyle \phi_{1/\overline{\rho}}(\vx^{(k)})  -\frac{\overline{\rho}}{2}\left(2\alpha_{k}(\overline{\rho}-\rho) -c_k\right)\| \vx^{(k)}-\widetilde\vx^{(k)}  \|^2+ \frac{\overline{\rho}}{2}(2+\frac{1}{c_k}) \beta_k^2 \|\vx^{(k)}-\vx^{(k-1)}\|^2 +4\overline{\rho}\alpha_{k}^2 M ^2 \nonumber\\
		&+\alpha_k \overline{\rho} (1-\alpha_{k}\overline{\rho})\widehat\cE_k.
	\end{align}
where $\widehat\cE_k$ is defined in \eqref{eq:delay-error-term-hat}. 
By the Young's inequality, we have
\begin{align*}
\textstyle - \frac{\rho}{2}\|\vx^{(k)}-\widetilde\vx^{(k)}\|^2+ \frac{\rho}{2}\|\vx^{(k-\tau_k)}-\widetilde\vx^{(k)}\|^2 = &~ \textstyle \frac{\rho}{2}\|\vx^{(k-\tau_k)}-\vx^{(k)}\|^2 + \rho \langle \vx^{(k-\tau_k)}-\vx^{(k)}, \vx^{(k)}-\widetilde\vx^{(k)}\rangle \\
\le & ~ \textstyle \big(\frac{\rho}{2} + \frac{\rho^2}{\bar\rho - \rho}\big) \|\vx^{(k-\tau_k)}-\vx^{(k)}\|^2 + \frac{\bar\rho - \rho}{4}\|\vx^{(k)}-\widetilde\vx^{(k)}\|^2.	
\end{align*}
Using the definition of $\widehat\cE_k$ in \eqref{eq:delay-error-term-hat} and substituting the inequality above into \eqref{eq:phi-ineq1}, we have from $1-\alpha_{k}\overline{\rho} \le 1$ and the definition of $\cE_k$ in \eqref{eq:delay-error-term} that
\begin{align}\label{eq:phi-ineq1-2}
&~\EE_{\xi_k} \big[ \phi_{1/\overline{\rho}}(\vx^{(k+1)})\big] \nonumber\\
\le &~ ~\textstyle \phi_{1/\overline{\rho}}(\vx^{(k)})  -\frac{\overline{\rho}}{2}\left(\frac{3}{2}\alpha_{k}(\overline{\rho}-\rho) -c_k\right)\| \vx^{(k)}-\widetilde\vx^{(k)}  \|^2+ \frac{\overline{\rho}}{2}(2+\frac{1}{c_k}) \beta_k^2 \|\vx^{(k)}-\vx^{(k-1)}\|^2 +4\overline{\rho}\alpha_{k}^2 M ^2 \nonumber\\
		&+\alpha_k \overline{\rho} (1-\alpha_{k}\overline{\rho})\cE_k.
\end{align}
	Taking full expectation and summing the inequality in \eqref{eq:phi-ineq1-2} over $k=k_0,\ldots,K$, we have 
	\begin{align}
		&~\EE \big[ \phi_{1/\overline{\rho}}(\vx^{(K+1)})\big] \nonumber\\
		\le&~\textstyle \EE\big[ \phi_{1/\overline{\rho}}(\vx^{(k_0)}) \big]   -\frac{\overline{\rho}}{2}\sum_{k=k_0}^{K}\big( \frac{3}{2} \alpha_{k}(\overline{\rho}-\rho) -c_k\big)   \EE\| \vx^{(k)}-\widetilde\vx^{(k)}  \|^2
		\nonumber \\
		&~ \textstyle + \frac{\overline{\rho}}{2}\sum_{k=k_0}^{K}( 2+\frac{1}{c_k}) \beta_k^2 \EE\|\vx^{(k)}-\vx^{(k-1)}\|^2 + 4\overline{\rho} M ^2\sum_{k=k_0}^{K}\alpha_{k}^2+\sum_{k=k_0}^{K} \alpha_k \overline{\rho} (1-\alpha_{k}\overline{\rho})\cE_k. \nonumber
	\end{align}
	Choose $c_k= \frac{1}{2} \alpha_{k}(\overline{\rho}-\rho)$	for all $k\ge 1$ and rearrange the above inequality. 
	We obtain 
	\begin{equation} \label{eq:phi-ineq2}
		\begin{aligned}
		\textstyle	\frac{\overline{\rho}(\overline{\rho}-\rho)}{2} \sum_{k=k_0}^{K} \alpha_{k} \EE\| \vx^{(k)}-\widetilde\vx^{(k)}  \|^2  
			\le& ~ \textstyle \EE\big[\phi_{1/\overline{\rho}}(\vx^{(k_0)})  -\phi^*\big]  + 4\overline{\rho} M ^2\sum_{k=k_0}^{K}\alpha_{k}^2  \\ 
			&~\textstyle \hspace{-4cm}+ \frac{\overline{\rho}}{2}\sum_{k=k_0}^{K}( 2+\frac{2}{\alpha_{k}(\overline{\rho}-\rho)}) \beta_k^2 \EE\|\vx^{(k)}-\vx^{(k-1)}\|^2+\sum_{k=k_0}^{K} \alpha_k \overline{\rho} (1-\alpha_{k}\overline{\rho})\EE[\cE_k],
		\end{aligned}
	\end{equation}
	where we have used the fact $\phi_{1/\overline{\rho}}(\vx) \ge \phi^*, \forall\, \vx\in \dom(r)$.	
	From Lemma~\ref{lem:prox-grad}, we have $\| \vx^{(k)}-\widetilde\vx^{(k)}  \|^2 = \|\nabla\phi_{1/\overline{\rho}} (\vx^{(k)})  \|^2/\overline\rho^2$. Hence, plugging this equation into the left-hand side of \eqref{eq:phi-ineq2} and using the choice of $T$ in \eqref{eq:def-T}, we obtain the desired result. 
\end{proof}

To show the convergence rate in \eqref{eq:nsm-general}, it suffices to bound the summation terms on $\EE\|\vx^{(k)}-\vx^{(k-1)}\|^2$ and the delay term $\EE[\cE_k]$. 
If the delay is arbitrary, it is impossible to have convergence, and thus a certain condition on $\tau_k$ is needed. For nonsmooth problems, we make the following assumption. 
\begin{assumption}[stochastic delay]\label{assump:rand-delay}
	There is an integer $\tau$ such that the staleness $\tau_k$ follows the distribution
	$$\Prob(\tau_k = j)=p_j,\, \text{ for }j=0,1,\ldots, \tau, \ \forall\, k.$$
\end{assumption}
If the computing environment does not change during all the iterations, the assumption will hold. In addition, one can track the delay at the master node and thus estimate the probability. However, we do not need to know the values of $\{p_j\}$ or $\tau$ in the computation and analysis, but we only require their existence. A similar assumption has been made in \cite{sra2016adadelay, hannah2018unbounded,peng2019convergence}.

In the rest of this section, we show convergence rate results separately for the case with a fixed stepsize sequence and the one with a varying stepsize sequence.
\subsubsection{Convergence rate with a fixed stepsize}
In this subsubsection, we consider the case where $\alpha_k=\alpha_1$ and $\beta_k=\beta_1$ for all $k\ge1$. In this case, it is easy to bound the summation term about $\EE\|\vx^{(k)}-\vx^{(k-1)}\|^2$. 

\begin{lemma}\label{lem:bdx-fix-alpha}
	Given a positive integer $K$, let $\alpha_{k}=\frac{\alpha}{\sqrt{K}}, \forall\, k=1,\ldots,K$ for some $\alpha>0.$ Also, let $\beta_k=\frac{\beta}{K^{1/4}},\forall \, k$ for some nonnegative $\beta$ such that $\frac{\beta}{K^{1/4}} < 1 - \frac{\alpha\rho}{2\sqrt{K}}$. Then under Assumptions~\ref{assump:weak-cvx} and  \ref{assump:subgrad-bound}, it holds
	\begin{align}\label{eq:sum-xdiff-sq}
		\textstyle \sum_{k=1}^K\EE\|\vx^{(k+1)} - \vx^{(k)}\|^2 \le \frac{\alpha}{\gamma\sqrt K} \big(\phi(\vx^{(1)}) - \phi^* \big) 
		+\frac{\alpha^2M ^2}{\gamma^2}, \text{ where } \gamma = \frac{1}{2}\big(1 -\frac{\alpha\rho}{2\sqrt{K}} - \frac{\beta}{K^{1/4}} \big).
	\end{align}
\end{lemma}

\begin{proof}
	Let $\gamma = \frac{1}{2}\big(1 -\frac{\alpha\rho}{2\sqrt{K}} - \frac{\beta}{K^{1/4}} \big)$ in \eqref{eq:bd-x-diff0} and 
	sum it up  over $k$. We have
	\begin{align*}
		&\textstyle\big( 1 -\gamma -\frac{\alpha\rho}{2\sqrt{K}} - \frac{\beta}{2K^{1/4}} \big)\sum_{k=1}^K\EE\|\vx^{(k+1)} - \vx^{(k)}\|^2 \\
\le &~\textstyle\frac{\alpha}{\sqrt K}  \EE\big(\phi(\vx^{(1)}) - \phi(\vx^{(K+1)})\big) +  \frac{\beta}{2K^{1/4}}\sum_{k=1}^K \EE\|\vx^{(k)}-\vx^{(k-1)}\|^2
		+\frac{\alpha^2M ^2}{\gamma}.
	\end{align*}
	Since $\vx^{(0)}=\vx^{(1)}$ and $\phi(\vx^{(K+1)})\ge \phi^*$, the above inequality together with the choice of $\gamma$ implies the desired result. We complete the proof.
\end{proof}

When a fixed stepsize sequence is used, we can bound $\sum_{k=1}^K\EE[\cE_k]$ as in the next lemma.

\begin{lemma}\label{lem:nsm-error-term}
	Let $\cE_k$ be defined in \eqref{eq:delay-error-term}. Given a positive integer $K$, let $\alpha_{k}=\frac{\alpha}{\sqrt{K}}, \forall\, k=1,\ldots,K$ for some $\alpha>0.$ Also, let $\beta_k=\frac{\beta}{K^{1/4}},\forall \, k$ for some nonnegative $\beta$ such that $\frac{\beta}{K^{1/4}} < 1 - \frac{\alpha\rho}{2\sqrt{K}}$. Suppose that $F(\vx)$ is upper bounded by $C_F$ for all $\vx\in\dom(r)$. Then under Assumptions~\ref{assump:weak-cvx}, \ref{assump:subgrad-bound}, and \ref{assump:rand-delay}, we have
	\begin{equation}\label{eq:sum-error-nsm}
	\begin{aligned}
	\textstyle	\sum_{k=1}^K\EE[\cE_k] \le &~ \textstyle \tau\max\big\{0, -F(\vx^{(1)})\big\} + \tau C_F + \tau^2 \big(\frac{\rho}{2} + \frac{\rho^2}{\bar\rho - \rho}\big) \left(\frac{\alpha}{\gamma\sqrt K} \big(\phi(\vx^{(1)}) - \phi^* \big) +\frac{\alpha^2M ^2}{\gamma^2}\right) \\
	&~ \textstyle + M \tau\sqrt K\sqrt{\frac{\alpha}{\gamma\sqrt K} \big(\phi(\vx^{(1)}) - \phi^* \big) +\frac{\alpha^2M ^2}{\gamma^2}},
	\end{aligned}
	\end{equation}
	where $\gamma = \frac{1}{2}\big(1 -\frac{\alpha\rho}{2\sqrt{K}} - \frac{\beta}{K^{1/4}} \big)$.
\end{lemma}

Now from Theorem~\ref{thm:prox-sync} and Lemmas~\ref{lem:bdx-fix-alpha} and \ref{lem:nsm-error-term}, we can easily show the following convergence rate result. 

\begin{theorem}[convergence rate with fixed stepsize]\label{thm:nsm-rate-fixed-delay} Under Assumptions~\ref{assump:weak-cvx}--\ref{assump:subgrad-bound} and \ref{assump:rand-delay}, let $\overline\rho\in (\rho, 2\rho]$ and $K$ be the maximum number of iterations. Let $\{\vx^{(k)}\}$ be the sequence from Alg.~\ref{alg:async-hvb-sgm} with $\alpha_{k}=\frac{\alpha}{\sqrt{K}}$ and $\beta_k=\frac{\beta}{K^{1/4}},\forall \, k=1,\ldots,K$ for some $\alpha>0$ and nonnegative $\beta$ such that $\frac{\alpha}{\sqrt{K}}\in (0,1/\overline\rho]$ and $\frac{\beta}{K^{1/4}} < 1 - \frac{\alpha\rho}{2\sqrt{K}}$. Suppose that $F(\vx)$ is upper bounded by $C_F$ for all $\vx\in\dom(r)$. Then
	\begin{equation}\label{eq:nsm-rate-delay}
		\begin{aligned}
			\EE\big\|\nabla\phi_{1/\overline{\rho}} (\vx^{(T)} )  \big\|^2  \le &~ {\textstyle\frac{2\overline{\rho}}{(\overline{\rho}-\rho) \alpha\sqrt{K}} }
			\bigg[ \textstyle \frac{\overline{\rho}}{2}(2+\frac{2\sqrt K}{\alpha(\overline{\rho}-\rho)}) \frac{\beta^2}{\sqrt K} \Big(\frac{\alpha}{\gamma\sqrt K} \big(\phi(\vx^{(1)}) - \phi^* \big) 
			+\frac{\alpha^2M ^2}{\gamma^2}\Big) + 4\overline{\rho} M ^2\alpha^2  \\
			&  \hspace{-1.5cm}+ \textstyle \frac{\alpha\overline\rho \tau}{\sqrt K} \left(\max\big\{0, -F(\vx^{(1)})\big\} + C_F + \tau \big(\frac{\rho}{2} + \frac{\rho^2}{\bar\rho - \rho}\big) \big(\frac{\alpha}{\gamma\sqrt K} \big(\phi(\vx^{(1)}) - \phi^* \big) +\frac{\alpha^2M ^2}{\gamma^2}\big) \right)\\
			&~ \textstyle  + \phi_{1/\overline{\rho}}(\vx^{(1)})  -\phi^*  + M \alpha\overline\rho\tau \sqrt{\frac{\alpha}{\gamma\sqrt K} \big(\phi(\vx^{(1)}) - \phi^* \big) +\frac{\alpha^2M ^2}{\gamma^2}}
			\bigg],
		\end{aligned}
	\end{equation}
	where $\gamma = \frac{1}{2}\big(1 -\frac{\alpha\rho}{2\sqrt{K}} - \frac{\beta}{K^{1/4}} \big)$ and $T$ is randomly selected from $\{1,\ldots, K\}$ by \eqref{eq:def-T}.
\end{theorem}

\begin{proof}
	Notice $\sum_{k=1}^K\alpha_k = \alpha\sqrt K$ and $\sum_{k=1}^K\alpha_k^2 = \alpha^2$. Then the inequality in \eqref{eq:nsm-rate-delay} directly follows by substituting \eqref{eq:sum-xdiff-sq} and \eqref{eq:sum-error-nsm} into \eqref{eq:nsm-general} with $k_0=1$, and also noticing $1-\frac{\alpha\overline\rho}{\sqrt K} \le 1$. 
\end{proof}

\begin{remark}
	The result in \eqref{eq:nsm-rate-delay} indicates a convergence rate of $O(1/\sqrt K)$. For the no-delay case (i.e., $\tau=0$), the assumption $F(\vx) \le C_F, \forall\, \vx\in\dom(r)$ is not needed. The delay case has the same-order convergence as 
	the no-delay case. However, their constants are different. Compared to the no-delay case, the delay one has a few additional terms dependent on $\tau$. The term dependent on $\tau$ in the second line on the right-hand side of \eqref{eq:nsm-rate-delay} is negligible if $K$ is a large number, but the term in the third line will not vanish as $K\to\infty$. In other words, the delay always has a non-negligible effect on the convergence rate. To take a clearer look at the effect, let $\overline\rho=2\rho$, $\beta=0$, and $K\to\infty$. Then $\gamma\to \frac{1}{2}$, and the terms enclosed in the big square brackets of \eqref{eq:nsm-rate-delay} roughly equal $\phi_{1/\overline{\rho}}(\vx^{(1)})  -\phi^* + 8\rho \alpha^2M ^2+4\rho\alpha^2M^2\tau$. Hence, the delay can slow down the convergence rate by $\frac{\tau}{\tau + 2 + (\phi_{1/\overline{\rho}}(\vx^{(1)})  -\phi^*)/(4\rho\alpha^2M^2)}$. This indicates that the delay will have a smaller effect if $\rho$ is smaller (i.e., $F$ is closer to convexity) or if $\alpha$ is smaller (i.e., a smaller learning rate is used). 
\end{remark}

\subsubsection{Convergence rate with varying stepsizes}

When $\alpha_k$ varies with $k$, $\sum_k\alpha_{k} \big(\phi(\vx^{(k)}) - \phi(\vx^{(k+1)}) \big)$ may not be a telescoping series any more, so we cannot directly obtain a bound as in \eqref{eq:sum-xdiff-sq} by summing up \eqref{eq:bd-x-diff0}.  Below we make an additional assumption and show a bound on  $\sum_{k=1}^K\|\vx^{(k+1)} - \vx^{(k)}\|^2$ when $\alpha_k = \alpha/\sqrt{k}$ for all $k\ge1$.

\begin{assumption} \label{assump:subgrad-bound2}
	At least one of the following conditions holds.
	\begin{enumerate}[leftmargin=1.2cm]
		\item\label{subgrad-bnd1} 
		$\phi$ is bounded on $\dom(r)$, i.e., there is $C_\phi$ such that $|\phi(\vx)| \le C_\phi,\forall\, \vx\in \dom(r)$.
		\item\label{subgrad-bnd2} 
		The function $r$ has the form of $r=r_1+r_2$,  
		where $r_1$ is the indicator function of a closed convex set $X\subseteq \RR^n$, and  $r_2$ is convex. In addition, there is $M_r\ge 0$ such that   $  \| \vv \| \le M_r$ for all $\vx\in X$ and all $\vv\in\partial r_2(\vx)$. 
	\end{enumerate}	
\end{assumption}
In condition~\ref{subgrad-bnd1} of Assumption~\ref{assump:subgrad-bound2}, the boundedness of $\phi$ can be guaranteed if $\phi$ is continuous and $\dom(r)$ is compact. The second condition trivially holds if $r_2\equiv 0$, and it also holds if $X=\RR^n$ and $r_2$ is a Lipschitz continuous function such as a certain norm. 

\begin{lemma}
	Under Assumptions~\ref{assump:weak-cvx} and  \ref{assump:subgrad-bound}, let $\{\alpha_{k}\}$ be a positive nonincreasing sequence and $\alpha_1 < \frac {2} {\rho}$. Also, let $\beta_k\le \tilde\beta, \forall \,k\ge1$ for some $\tilde\beta$ such that $0\le\tilde\beta < 1 - \frac{\alpha_1\rho}{2}$. Then if the first condition in Assumption~\ref{assump:subgrad-bound2} holds, we have for any positive integer $K$,
	\begin{equation}\label{eq:sum-xdiff-sq-vary}
	\textstyle	\sum_{k=1}^K \|\vx^{(k+1)} - \vx^{(k)}\|^2 \le \frac{2\alpha_1C_\phi}{\gamma} + \sum_{k=1}^K\frac{\alpha_k^2 M^2}{\gamma^2},~\text{where}~\gamma = \frac{1}{2}\big({\textstyle 1 -\frac{\alpha_1\rho}{2} - \tilde\beta} \big).
	\end{equation}
\end{lemma}

\begin{proof}
	When condition~\ref{subgrad-bnd1} of Assumption~\ref{assump:subgrad-bound2} holds, i.e., $|\phi(\vx)| \le C_\phi,\forall\, \vx\in \dom(r)$, we have from the nonincreasing monotonicity of $\alpha_k$ that
	\begin{equation}\label{eq:bd-sum-phi}
	\begin{aligned}
	\textstyle	\sum_{k=1}^K\alpha_{k} \big(\phi(\vx^{(k)}) - \phi(\vx^{(k+1)}) \big) = &~\textstyle\alpha_1\phi(\vx^{(1)}) + \sum_{k=2}^K(\alpha_k-\alpha_{k-1}) \phi(\vx^{(k)}) - \alpha_K\phi(\vx^{(K+1)})\\
		\le &~\textstyle \alpha_1C_\phi - \sum_{k=2}^K(\alpha_k-\alpha_{k-1}) C_\phi + \alpha_K C_\phi = 2\alpha_1C_\phi.
	\end{aligned}
	\end{equation}
	Hence, let $\gamma = \frac{1}{2}\big(1 -\frac{\alpha_1\rho}{2} - \tilde\beta \big)$ in \eqref{eq:bd-x-diff0} and 
	sum it up  over $k$. We have by $\gamma \le 1 -\gamma -\frac{\alpha_k\rho}{2} - \beta_k,\forall\, k\ge1$ that
	$$\textstyle\gamma\sum_{k=1}^K \|\vx^{(k+1)} - \vx^{(k)}\|^2 \le 2\alpha_1C_\phi + \sum_{k=1}^K\frac{\alpha_k^2 M^2}{\gamma},$$
	which apparently implies the desired result.
\end{proof}

\begin{lemma}\label{lem:neighbor}	
	Suppose that Assumption~\ref{assump:subgrad-bound} and condition~\ref{subgrad-bnd2} of Assumption~\ref{assump:subgrad-bound2} hold. Let $\{\vx^{(k)}\}$ be the sequence from Alg.~\ref{alg:async-hvb-sgm} with a stepsize sequence $\{\alpha_k\}$ and inertial parameter $\{\beta_k\}$ such that $\beta_k\le \tilde\beta < 1$. Then for any positive integer $K$,
	\begin{equation}\label{eq:ex-grad-bound2} 
	\textstyle	\sum_{k=1}^K\EE\|\vx^{(k+1)} -\vx^{(k)} \|^2 
		\le ( M_r^2+M ^2)\frac{4(1+\tilde\beta^2)}{(1-\tilde\beta^2)^2}\sum_{k=1}^{K}\alpha_{k}^2.
	\end{equation}	
\end{lemma}

We still need to bound $\sum_{k=k_0}^{K} \alpha_k \overline{\rho} (1-\alpha_{k}\overline{\rho})\EE[\cE_k]$ in \eqref{eq:nsm-general}.

\begin{lemma}\label{lem:nsm-error-term-vary}
	Under Assumptions~\ref{assump:weak-cvx}--\ref{assump:subgrad-bound2}, let $\overline \rho \in (\rho, 2\rho]$ and $\alpha_{k}=\frac{\alpha}{\sqrt{k}}, \forall\, k\ge1$ for some $0<\alpha \le 1/\overline\rho$. Also, let $\beta_k=\min\big\{\tilde\beta, \frac{\beta}{k^{1/4}}\big\},\forall \, k,$ for some $\tilde\beta$ such that $0\le\tilde\beta < 1 - \frac{\alpha\rho}{2}$. Furthermore, assume $|F(\vx)| \le C_F, \forall\, \vx\in\dom(r)$. Then for any integer $K$ and $1 \le k_0 \le K$, it holds
\begin{align}\label{eq:bound-nsm-error-term-vary}
\textstyle \sum_{k=k_0}^{K} \alpha_k \overline{\rho} (1-\alpha_{k}\overline{\rho})\EE[\cE_k] \le & ~ \textstyle 2\alpha_{k_0} \overline{\rho} \tau C_F + \alpha_{k_0} \tau^2 \overline{\rho} \big(\frac{\rho}{2} + \frac{\rho^2}{\overline\rho - \rho}\big)\big(C_1 + C_2\alpha^2(1+\ln K)\big)	\\
& ~ \textstyle + M \tau \overline{\rho} \sqrt{\sum_{k=k_0}^K \alpha_k^2} \sqrt{C_1 + C_2\alpha^2(1+\ln K)}, \nonumber
\end{align}
where $\cE_k$ is defined in \eqref{eq:delay-error-term}.	
\end{lemma}

Now we are ready to show the convergence rate result for the case with varying stepsize. 

\begin{theorem}[convergence rate with varying stepsize]\label{thm:rate-nsm-vary-delay}
Under the same assumptions of Lemma~\ref{lem:nsm-error-term-vary}, let $\{\vx^{(k)}\}$ be the sequence from Alg.~\ref{alg:async-hvb-sgm}. We have	
	\begin{equation} \label{eq:nsm-vary-alpha-delay}
\begin{aligned}
    \textstyle \EE\big\|\nabla\phi_{1/\overline{\rho}} (\vx^{(T)} )  \big\|^2 \le & ~\textstyle\frac{\overline{\rho}}{(\overline{\rho}-\rho)\alpha(\sqrt{K+1}-\sqrt{k_0})} 
	\Big[\EE\big[\phi_{1/\overline{\rho}}(\vx^{(k_0)})  -\phi^*\big]  + 4\overline{\rho} M ^2\alpha^2(1+\ln K - \ln k_0) \\
	& ~~~~+ \textstyle \frac{\overline\rho}{2}\big(2\tilde\beta^2 +  \frac{2\beta^2}{\alpha(\overline{\rho}-\rho)} \big)\big(C_1 + C_2\alpha^2(1+\ln K)\big) \\
	& ~~~~+ \textstyle 2\frac{\alpha}{\sqrt{k_0}} \overline{\rho} \tau C_F + \frac{\alpha}{\sqrt{k_0}} \tau^2 \overline{\rho} \big(\frac{\rho}{2} + \frac{\rho^2}{\overline\rho - \rho}\big)\big(C_1 + C_2\alpha^2(1+\ln K)\big)	\\
& ~~~~ \textstyle + \alpha M \tau \overline{\rho} \sqrt{1+\ln K-\ln k_0} \sqrt{C_1 + C_2\alpha^2(1+\ln K)}	\Big],
\end{aligned}	    
	\end{equation}
	where $T$ is randomly selected from $\{k_0,\ldots,K\}$ by \eqref{eq:def-T} and 		
	\begin{subequations}\label{eq:def_C12}
		\begin{align}
			&\textstyle C_1 = \frac{4\alpha C_\phi}{1-\frac{\alpha\rho}{2}-\tilde\beta}, \quad C_2 = \frac{4M^2}{(1-\frac{\alpha\rho}{2}-\tilde\beta)^2}, \text{ if condition~\ref{subgrad-bnd1} of Assumption~\ref{assump:subgrad-bound2} holds; or},\\
			&\textstyle C_1 = 0, \quad C_2 = ( M_r^2+M ^2)\frac{4(1+\tilde\beta^2)}{(1-\tilde\beta^2)^2}, \text{ if condition~\ref{subgrad-bnd2} of Assumption~\ref{assump:subgrad-bound2} holds}.
		\end{align}
	\end{subequations}
\end{theorem}
\begin{proof}
	By the choice of $\{\alpha_k\}$  
	and $\{\beta_k\}$, we have
	$$\textstyle\sum_{k=k_0}^{K}( 2+{  \frac{2}{\alpha_{k}(\overline{\rho}-\rho)} }) \beta_k^2 \EE\|\vx^{(k)}-\vx^{(k-1)}\|^2 \le (2\tilde\beta^2 + {  \frac{2\beta^2}{\alpha(\overline{\rho}-\rho)} }) \sum_{k=k_0}^{K}\EE\|\vx^{(k)}-\vx^{(k-1)}\|^2,$$
	which, together with \eqref{eq:sum-xdiff-sq-vary} and \eqref{eq:ex-grad-bound2} and also Lemma~\ref{lem:sum-alpha-bnd} with $a=1$, gives
	\begin{align}\label{eq:bd-sum-xdiff-C12-k0}
	\textstyle	\frac{\overline\rho}{2}\sum_{k=k_0}^{K}( 2+{  \frac{2}{\alpha_{k}(\overline{\rho}-\rho)} }) \beta_k^2 \EE\|\vx^{(k)}-\vx^{(k-1)}\|^2 \le & ~ \textstyle \frac{\overline\rho}{2}(2\tilde\beta^2 + {  \frac{2\beta^2}{\alpha(\overline{\rho}-\rho)} })\big(C_1 + C_2\sum_{k=1}^K\alpha_k^2\big)\cr
	\le & ~ \textstyle \frac{\overline\rho}{2}(2\tilde\beta^2 + {  \frac{2\beta^2}{\alpha(\overline{\rho}-\rho)} })\big(C_1 + C_2\alpha^2(1+\ln K)\big),
	\end{align}
	with $C_1$ and $C_2$ defined in \eqref{eq:def_C12}.
In addition, $\sum_{k=k_0}^K\alpha_k \ge \alpha \int_{k_0}^{K+1}\frac{1}{\sqrt{x}}dx = 2\alpha(\sqrt{K+1}-\sqrt{k_0})$ and  $\sum_{k=k_0}^K\alpha_k^2 \le \alpha^2 + \alpha^2 \int_{k_0}^{K}\frac{1}{x}dx = \alpha^2(1+\ln K-\ln k_0)$. Hence, substituting \eqref{eq:bound-nsm-error-term-vary} and \eqref{eq:bd-sum-xdiff-C12-k0} into \eqref{eq:nsm-general} 
gives the desired result.  
\end{proof}

\begin{remark}
For the no-delay case (i.e., $\tau=0$), we can set $k_0=1$ in Theorem~\ref{thm:rate-nsm-vary-delay}; then the assumption $|F(\vx)| \le C_F, \forall\, \vx\in\dom(r)$ is not needed anymore. When $\tau>0$, the negative effect by the delay will not vanish as $K\to \infty$, similar to what we observe for the result in Theorem~\ref{thm:nsm-rate-fixed-delay}. Suppose that we have an estimate on $\tau$ and $K\gg \tau^4$. We can set $k_0=\Omega(\tau^4)$. Then the terms caused by the delay will near-linearly depend on $\tau$.
\end{remark}

\section{Convergence analysis for nonconvex composite problems}\label{sec:prox-async}

In this section, we analyze Alg.~\ref{alg:async-hvb-sgm} for problems in the form of \eqref{eq:stoc-prob}, where  
 $F$ is smooth and $r$ is a possibly nonsmooth convex function. 
 Instead of the $\rho$-weak convexity, we assume the $\rho$-smoothness condition on $F$. Here, we abuse the notation of $\rho$, which is used as the weak-convexity constant in the previous section. Nevertheless, if $F$ is $\rho$-smooth, it is also $\rho$-weakly convex. The stronger assumption will enable us to obtain better convergence result in terms of the effect caused by the staleness of the gradient.
\begin{assumption}[$\rho$-smoothness]\label{smoothness}
	$F(\vx)$ is $\rho$-smooth in $\dom(r)$, i.e., $$\|\nabla F(\vx)-\nabla F(\vy)\|\leq \rho \|\vx-\vy\|, \forall\, \vx,\vy\in \dom(r).$$
\end{assumption}
 
 When $F$ is smooth, it is standard to replace Assumption~\ref{assump:subgrad-bound} by the one below. 
 \begin{assumption}[bounded variance]\label{assump:variance}
There is $\sigma\ge 0$ such that  $\EE_\xi \| \nabla f(\vx;\xi)-\nabla F(\vx)\|^2\le \sigma^2$ for all $\vx\in\dom\,r$.
\end{assumption}

In addition, when $F$ is smooth, we only need a boundedness condition on the staleness but not a static distribution anymore.
\begin{assumption}[bounded staleness]\label{assump:bound-tau}
	There is a finite integer $\tau$ such that $\tau_k\le \tau$ for all $k\ge1$.
\end{assumption}
We can track the delay and ensure the boundedness of delay by discarding too outdated sample gradients.

\begin{lemma}\label{lem:g-deviate}	
	Under Assumptions~\ref{assump:unbiased}, \ref{smoothness}, and \ref{assump:variance}, the iterates $\{\vx^{(k)}\}$ from Algorithm~\ref{alg:async-hvb-sgm} satisfy 
	$$\EE_{\xi_k} \|\vg^{(k)}-\nabla F (\vx^{(k)})\|^2   \le \sigma^2+\rho^2\|\vx^{(k-\tau_k)} - \vx^{(k)}\|^2.$$
\end{lemma}
\begin{proof}
When $F$ is differentiable, the condition in Assumption~\ref{assump:unbiased}	becomes $\EE_{\xi_k} [\vg^{(k)}]=\nabla F (\vx^{(k-\tau_k)})$. Hence,   
	\begin{align}
		\EE_{\xi_k} \|\vg^{(k)}-\nabla F (\vx^{(k)})\|^2  
		=& ~\EE_{\xi_k} \|\vg^{(k)}-\nabla F (\vx^{(k-\tau_k)})+ \nabla F (\vx^{(k-\tau_k)})-\nabla F (\vx^{(k)})\|^2   \nonumber\\
		=&~\EE_{\xi_k} \|\vg^{(k)}-\nabla F (\vx^{(k-\tau_k)})\|^2 +  \|\nabla F (\vx^{(k-\tau_k)})-\nabla F ( \vx^{(k)})\|^2 \nonumber\\
		\le&~\sigma^2 +\rho^2\|\vx^{(k-\tau_k)} - \vx^{(k)}\|^2,\nonumber
	\end{align}
where the second inequality follows from $\EE_{\xi_k}\big\langle\vg^{(k)}-\nabla F (\vx^{(k-\tau_k)}), \nabla F (\vx^{(k-\tau_k)})-\nabla F (\vx^{(k)})\big\rangle = 0$, and the inequality holds by using Assumptions~\ref{smoothness} and \ref{assump:variance}. This completes the proof.	
\end{proof}

\begin{lemma}\label{lem:keybound}
Under Assumptions~\ref{assump:unbiased},  \ref{smoothness} and \ref{assump:variance}, let $\overline{\rho}>\rho$ and $\alpha_{k}\in (0, 1/\overline{\rho}]$ for all $k$. Then the iterates $\{\vx^{(k)}\}$ from Algorithm~\ref{alg:async-hvb-sgm} with a stepsize sequence $\{\alpha_k\}$ satisfies
	\begin{equation}\label{eq:keybound}
	\begin{aligned}
		\EE_{\xi_k}  \|\vx^{(k+1)}-\widetilde\vx^{(k)} \|^2  
		\le &~  \| \vx^{(k)}-\widetilde\vx^{(k)}  \|^2 -\left(\textstyle \frac{1}{2}\alpha_{k}(\overline{\rho}-\rho) -c_k\right)   \| \vx^{(k)}-\widetilde\vx^{(k)}  \|^2 \\ 
		&~ \hspace{-1.cm}+ \textstyle (2+\frac{1}{c_k}) \beta_k^2 \|\vx^{(k)}-\vx^{(k-1)}\|^2  
		+   \alpha_{k}^2\sigma^2   +2\big( \alpha_{k}^2+\frac{ \alpha_{k}} {\overline{\rho}-\rho}\big) \rho^2\|\vx^{(k-\tau_k)} - \vx^{(k)}\|^2, 
	\end{aligned}	
	\end{equation}
	where $c_k$ is any positive number, $\gamma_1\in(0,1)$, and $\widetilde\vx^{(k)}$ is defined in \eqref{eq:def-xtilde}.
\end{lemma}


Using the previous two lemmas, we show a convergence result below for generic parameters.

\begin{theorem}\label{thm:prox-async}
Under Assumptions~\ref{assump:unbiased},  \ref{smoothness} and \ref{assump:variance}, let $\overline{\rho}>\rho$ and $\alpha_{k}\in (0, 1/\overline{\rho}]$ for all $k\ge1$. Given a positive integer $K$,  let $\{\vx^{(k)}\}_{k=1}^K$ 
	be the sequence generated from Algorithm~\ref{alg:async-hvb-sgm} with a stepsize sequence $\{\alpha_{k}\}_{k=1}^{K}$ and inertial parameter $\{\beta_k\}$. Then   
	\begin{equation}\label{eq:prox-async}
	\begin{aligned}
		\EE\| \nabla\phi_{1/\overline{\rho}} (\vx^{(T)})  \|^2 
		\le&~\textstyle\frac{8\overline{\rho}}{(\overline{\rho}-\rho)\sum_{k=1}^{K} \alpha_{k}} \Big[ \overline{\rho}\sum_{k=1}^{K}\big( 2+\frac{4}{\alpha_{k}(\overline{\rho}-\rho)} \big) \beta_k^2 \EE\|\vx^{(k)}-\vx^{(k-1)}\|^2 \\
		&\hspace{-1.5cm}\textstyle \phi_{1/\overline{\rho}}(\vx^{(1)})- \phi^*   +    \frac{\sigma^2\overline{\rho}}{2}\sum_{k=1}^{K}\alpha_{k}^2   + \frac{\overline{\rho}\rho^2}{2}\sum_{k=1}^{K}\big(  \alpha_{k}^2+\frac{ \alpha_{k} }{\overline{\rho}-\rho} \big) \EE\|\vx^{(k-\tau_k)} - \vx^{(k)}\|^2\Big],
	\end{aligned}
	\end{equation}
	where $T$ is randomly selected from $\{1,\ldots, K\}$ by \eqref{eq:def-T}.
\end{theorem}
\begin{proof}
	By the definition of $\phi_\lambda$ in \eqref{eq:moreau} and Lemma~\ref{lem:keybound}, we have
	\begin{align}\label{eq:phi-ineq-prox1}
		&~\EE_{\xi_k} \big[ \phi_{1/\overline{\rho}}(\vx^{(k+1)})\big] \nonumber\\
		\le& ~\EE_{\xi_k} \big[\textstyle \phi (\widetilde\vx^{(k)}) + \frac{\overline{\rho}}{2}\|\vx^{(k+1)}-\widetilde\vx^{(k)}\|^2\big] \nonumber\\
		\le&~\textstyle\phi (\widetilde\vx^{(k)}) +\frac{\overline{\rho}}{2} \Big[ \| \vx^{(k)}-\widetilde\vx^{(k)}  \|^2 -( \frac{1}{2}\alpha_{k}(\overline{\rho}-\rho) -c_k)   \| \vx^{(k)}-\widetilde\vx^{(k)}  \|^2 \nonumber\\
		&~  \textstyle + \big(2+\frac{1}{c_k}\big) \beta_k^2 \|\vx^{(k)}-\vx^{(k-1)}\|^2 +  \alpha_{k}^2\sigma^2   +2\big( \alpha_{k}^2+\frac{ \alpha_{k} }{\overline{\rho}-\rho}\big) \rho^2\|\vx^{(k-\tau_k)} - \vx^{(k)}\|^2 \Big]\nonumber\\
		=&~ \textstyle \phi_{1/\overline{\rho}}(\vx^{(k)})  -\frac{\overline{\rho}}{2}( \frac{1}{2}\alpha_{k}(\overline{\rho}-\rho) -c_k)   \| \vx^{(k)}-\widetilde\vx^{(k)}  \|^2+ \frac{\overline{\rho}}{2}\big(2+\frac{1}{c_k}\big) \beta_k^2 \|\vx^{(k)}-\vx^{(k-1)}\|^2\\
		&~+  \textstyle  \frac{\overline{\rho}\alpha_{k}^2\sigma^2}{2}   +2\big( \alpha_{k}^2+\frac{ \alpha_{k} }{\overline{\rho}-\rho}\big) \frac{\overline{\rho}\rho^2}{2}\|\vx^{(k-\tau_k)} - \vx^{(k)}\|^2.  \nonumber
	\end{align}
	Take full expectation on both sides of \eqref{eq:phi-ineq-prox1} and sum up it over $k=1,\ldots,K$. Then we have  
	{\small	
	\begin{align}\label{eq:phi-ineq-prox2}
		&~\EE\big[ \phi_{1/\overline{\rho}}(\vx^{(K+1)})\big] \nonumber\\
		\le&~\textstyle\EE\big[ \phi_{1/\overline{\rho}}(\vx^{(1)})\big]   -\frac{\overline{\rho}}{2}\sum_{k=1}^{K}\big( \frac{1}{2} \alpha_{k}(\overline{\rho}-\rho) -c_k\big)   \EE\| \vx^{(k)}-\widetilde\vx^{(k)}  \|^2+ \overline{\rho}\sum_{k=1}^{K}\big( 2+\frac{1}{c_k} \big) \beta_k^2 \EE\|\vx^{(k)}-\vx^{(k-1)}\|^2 \nonumber\\
		&\textstyle+    \frac{\sigma^2\overline{\rho}}{2}\sum_{k=1}^{K}\alpha_{k}^2   + \frac{\overline{\rho}\rho^2}{2}\sum_{k=1}^{K}\big(  \alpha_{k}^2+\frac{ \alpha_{k} }{\overline{\rho}-\rho} \big) \EE\|\vx^{(k-\tau_k)} - \vx^{(k)}\|^2. 
	\end{align}
}Choose $c_k = 	\frac{1}{4} \alpha_{k}(\overline{\rho}-\rho)$ for all $k$ and replace $\| \vx^{(k)}-\widetilde\vx^{(k)}  \|^2$ by $\frac{1}{\overline\rho^2} \|\nabla\phi_{1/\overline{\rho}} (\vx^{(k)})  \|^2 $ from Lemma~\ref{lem:prox-grad}. We have from \eqref{eq:phi-ineq-prox2} that
{\small\begin{align}\label{eq:phi-ineq-prox3}
		&~\EE\big[ \phi_{1/\overline{\rho}}(\vx^{(K+1)})\big] \nonumber\\
		\le&~\textstyle\EE\big[ \phi_{1/\overline{\rho}}(\vx^{(1)})\big]   -\frac{1}{8\overline{\rho}}\sum_{k=1}^{K}\alpha_{k}(\overline{\rho}-\rho)    \EE\| \nabla\phi_{1/\overline{\rho}} (\vx^{(k)})  \|^2+ \overline{\rho}\sum_{k=1}^{K}\big( 2+\frac{4}{\alpha_{k}(\overline{\rho}-\rho)} \big) \beta_k^2 \EE\|\vx^{(k)}-\vx^{(k-1)}\|^2 \nonumber\\
		&\textstyle+    \frac{\sigma^2\overline{\rho}}{2}\sum_{k=1}^{K}\alpha_{k}^2   + \frac{\overline{\rho}\rho^2}{2}\sum_{k=1}^{K}\big(  \alpha_{k}^2+\frac{ \alpha_{k} }{\overline{\rho}-\rho} \big) \EE\|\vx^{(k-\tau_k)} - \vx^{(k)}\|^2. 
	\end{align}
}Rearrange terms in \eqref{eq:phi-ineq-prox3} and notice $\phi_{1/\overline{\rho}}(\vx^{(K+1)})\ge\phi^*$, we obtain the desired result by the definition of $T$.	
\end{proof}

To show the convergence rate, we still need the following result to bound $\sum_{k\ge1}\EE\|\vx^{(k+1)} - \vx^{(k)}\|^2$.
\begin{lemma}
Let $\{\vx^{(k)}\}$ be generated from Alg.~\ref{alg:async-hvb-sgm}. Under Assumptions~~\ref{smoothness} and \ref{assump:variance}, it holds for any $\gamma>0$,
\begin{equation}\label{eq:bd-x-diff0-prox}
\begin{aligned}
\textstyle (1-\gamma-\frac{\alpha_k\rho}{2} - \frac{\beta_k}{2})\EE_{\xi_k}\|\vx^{(k+1)} - \vx^{(k)}\|^2 \le & ~\textstyle \alpha_k\EE_{\xi_k}\big(\phi(\vx^{(k)}) - \phi(\vx^{(k+1)}) \big) + \frac{\beta_k}{2} \|\vx^{(k)}-\vx^{(k-1)}\|^2 \\
 & ~+ \textstyle \frac{\alpha_k^2}{2\gamma}\left(\rho^2\|\vx^{(k)} - \vx^{(k-\tau_k)}\|^2 + \sigma^2\right).
\end{aligned}
\end{equation}

\end{lemma}

\begin{proof}
By the $\rho$-smoothness of $F$ and $\alpha_k>0$, it holds
\begin{equation}\label{eq:F-smooth-ineq}
\alpha_k \big(F(\vx^{(k+1)}) - F(\vx^{(k)})\big) \le \alpha_k \left( \textstyle \langle\vx^{(k+1)} - \vx^{(k)}, \nabla F(\vx^{(k)}) \rangle + \frac{\rho}{2}\|\vx^{(k+1)} - \vx^{(k)}\|^2\right).
\end{equation}
Also notice that \eqref{eq:bd-x-diff} still holds. Hence, we obtain, by adding \eqref{eq:bd-x-diff} and \eqref{eq:F-smooth-ineq} and rearranging terms, that
\begin{equation}\label{eq:bd-x-diff-prox}
\begin{aligned}
\textstyle (1-\frac{\alpha_k\rho}{2})\|\vx^{(k+1)} - \vx^{(k)}\|^2 \le & ~\alpha_k\big(\phi(\vx^{(k)}) - \phi(\vx^{(k+1)}) \big) + \alpha_k\big\langle\vx^{(k+1)} - \vx^{(k)}, \nabla F(\vx^{(k)}) - \vg^{(k)} \big\rangle \\
 & ~+ \beta_k\big\langle \vx^{(k+1)} - \vx^{(k)}, \vx^{(k)}-\vx^{(k-1)}\big\rangle.
\end{aligned}
\end{equation}

By the Young's inequality, we have for any $\gamma>0$,
\begin{equation}\label{eq:young-x-g}
\alpha_k\big\langle\vx^{(k+1)} - \vx^{(k)}, \nabla F(\vx^{(k)}) - \vg^{(k)} \big\rangle \le \textstyle \gamma\|\vx^{(k+1)} - \vx^{(k)}\|^2 + \frac{\alpha_k^2}{4\gamma}\|\nabla F(\vx^{(k)}) - \vg^{(k)}\|^2,
\end{equation}
and 
\begin{equation}\label{eq:young-x-diff}
\beta_k\big\langle \vx^{(k+1)} - \vx^{(k)}, \vx^{(k)}-\vx^{(k-1)}\big\rangle \le \textstyle\frac{\beta_k}{2}\left(\|\vx^{(k+1)} - \vx^{(k)}\|^2 + \|\vx^{(k)}-\vx^{(k-1)}\|^2\right).
\end{equation}
Plugging \eqref{eq:young-x-g}  and \eqref{eq:young-x-diff} into \eqref{eq:bd-x-diff-prox} gives
\begin{equation}\label{eq:bd-x-diff2-prox}
\begin{aligned}
\textstyle (1-\frac{\alpha_k\rho}{2})\|\vx^{(k+1)} - \vx^{(k)}\|^2 \le & ~\alpha_k\big(\phi(\vx^{(k)}) - \phi(\vx^{(k+1)}) \big) + \textstyle \gamma\|\vx^{(k+1)} - \vx^{(k)}\|^2  \\
 & ~ \hspace{-2cm} \textstyle + \frac{\alpha_k^2}{4\gamma}\|\nabla F(\vx^{(k)}) - \vg^{(k)}\|^2 + \frac{\beta_k}{2}\left(\|\vx^{(k+1)} - \vx^{(k)}\|^2 + \|\vx^{(k)}-\vx^{(k-1)}\|^2\right).
\end{aligned}
\end{equation}
Now notice $\EE_{\xi_k}\|\nabla F(\vx^{(k)}) - \vg^{(k)}\|^2 \le 2\|\nabla F(\vx^{(k)}) - \nabla F(\vx^{(k-\tau_k)})\|^2 + 2\EE_{\xi_k}\|\nabla F(\vx^{(k-\tau_k)}) - \vg^{(k)}\|^2$ and use Assumptions~\ref{smoothness} and \ref{assump:variance}. We obtain the desired result by taking a conditional expectation about $\xi_k$ over both sides of \eqref{eq:bd-x-diff2-prox} and rearranging terms.
\end{proof}

Now we are ready to show the convergence rate result.
\begin{theorem}[convergence rate with fixed stepsize]\label{thm:rate-prox-delay}
Under Assumptions~\ref{assump:unbiased},  \ref{smoothness}, \ref{assump:variance} and \ref{assump:bound-tau}, let $\overline{\rho}>\rho$ and $K$ be the maximum number of iterations. Choose $\alpha_k = \frac{\alpha}{\sqrt K}$ and $\beta_k= \frac{\beta}{K^{1/4}}$ for some $\alpha>0$ and $\beta\ge0$ such that $\tilde\gamma:=\frac{1}{2}-\frac{\alpha\rho}{2\sqrt K} - \frac{\tau\alpha^2\rho^2}{K}- \frac{\beta}{K^{1/4}}>0$. Let $\{\vx^{(k)}\}$ be the sequence from Alg.~\ref{alg:async-hvb-sgm}. Then
\begin{equation}\label{eq:prox-async-fix}
\begin{aligned}
		\textstyle\EE\|\nabla\phi_{1/\overline{\rho}} (\vx^{(T)})\|^2  
		\le&~\textstyle\frac{8\overline{\rho}}{(\overline{\rho}-\rho)\alpha\sqrt{K}}\Bigg[\phi_{1/\overline{\rho}}(\vx^{(1)})- \phi^*+ \frac{\sigma^2\overline{\rho}\alpha^2}{2} \\
&~\textstyle \hspace{-2cm}		+ \frac{1}{\tilde\gamma}\left(\overline{\rho}\big(2+\frac{4\sqrt K}{\alpha(\overline{\rho}-\rho)}\big)\frac{\beta}{\sqrt K} + \frac{\tau^2\overline{\rho}\rho^2}{2}\big(\frac{\alpha^2}{K} + \frac{ \alpha }{(\overline{\rho}-\rho)\sqrt K}\big)\right)\big( \frac{\alpha}{\sqrt K}\EE(\phi(\vx^{(1)}) - \phi^*) + \alpha^2\sigma^2\big) \Bigg],
	\end{aligned}    
\end{equation}
where $T$ is randomly selected from $\{1,\ldots, K\}$ by \eqref{eq:def-T} with $k_0=1$.
\end{theorem}

\begin{proof}
With $\alpha_k = \frac{\alpha}{\sqrt K}$ and $\beta_k = \frac{\beta}{K^{1/4}}$, 
we take full expectation over \eqref{eq:bd-x-diff0-prox} with $\gamma=\frac{1}{2}$ and sum it up over $k=1$ through $K$ to have 
\begin{equation} \label{eq:bd-x-diff0-prox-sum1}
\begin{aligned}
 \textstyle (\frac{1}{2}-\frac{\alpha\rho}{2\sqrt K} - \frac{\beta}{2K^{1/4}}) \sum_{k=1}^K\EE\|\vx^{(k+1)} - \vx^{(k)}\|^2 
\le & ~ \textstyle \frac{\alpha}{\sqrt K}\EE\big(\phi(\vx^{(1)}) - \phi(\vx^{(K+1)}) \big) \\
 & ~ \textstyle \hspace{-5cm}+ \frac{\beta}{2K^{1/4}} \sum_{k=1}^K \EE\|\vx^{(k)}-\vx^{(k-1)}\|^2 +    \frac{\alpha^2}{K} \sum_{k=1}^K\left(\rho^2\EE\|\vx^{(k)} - \vx^{(k-\tau_k)}\|^2 + \sigma^2\right).
\end{aligned}
\end{equation}
Notice that $\vx^{(0)}=\vx^{(1)}$ and by Assumption~\ref{assump:variance}, it holds
\begin{equation}\label{eq:delay-x-tau-bd}
\textstyle \|\vx^{(k)} - \vx^{(k-\tau_k)}\|^2\le \tau \sum_{j=1}^\tau \|\vx^{(k+1-j)} - \vx^{(k-j)}\|^2.
\end{equation} 
Hence, we have from \eqref{eq:bd-x-diff0-prox-sum1} by rearranging terms and using $\phi(\vx)\ge\phi^*, \forall\, \vx\in\dom(r)$ that
\begin{align}\label{eq:bd-x-diff0-prox-sum2}
\textstyle \big(\frac{1}{2}-\frac{\alpha\rho}{2\sqrt K} - \frac{\tau^2\alpha^2\rho^2}{K}- \frac{\beta}{K^{1/4}}\big) \sum_{k=1}^K\EE\|\vx^{(k+1)} - \vx^{(k)}\|^2 
\le  \frac{\alpha}{\sqrt K}\EE\big(\phi(\vx^{(1)}) - \phi^* \big) + \alpha^2\sigma^2.
\end{align}

Therefore, 
\begin{align}\label{eq:bd-x-diff0-prox-sum3}
&~\textstyle\overline{\rho}\sum_{k=1}^{K}\big(  2+\frac{4}{\alpha_{k}(\overline{\rho}-\rho)}  \big) \beta_k^2 \EE\|\vx^{(k)}-\vx^{(k-1)}\|^2 +\frac{\overline{\rho}\rho^2}{2}\sum_{k=1}^{K}\big(   \alpha_{k}^2+\frac{ \alpha_{k} }{\overline{\rho}-\rho} \big) \EE\|\vx^{(k-\tau_k)} - \vx^{(k)}\|^2\cr
\le &~\textstyle \Big( \overline{\rho}\big( 2+\frac{4\sqrt K}{\alpha(\overline{\rho}-\rho)} \big)\frac{\beta}{\sqrt K} + \frac{\tau^2\overline{\rho}\rho^2}{2} \big(\frac{\alpha^2}{K} + \frac{ \alpha }{(\overline{\rho}-\rho)\sqrt K}\big)  \Big) \sum_{k=1}^K\EE\|\vx^{(k)}-\vx^{(k-1)}\|^2\cr
\le &~\textstyle \frac{1}{\gamma}\left(\overline{\rho}\big( 2+\frac{4\sqrt K}{\alpha(\overline{\rho}-\rho)} \big)\frac{\beta}{\sqrt K} + \frac{\tau^2\overline{\rho}\rho^2}{2} \big(\frac{\alpha^2}{K} + \frac{ \alpha }{(\overline{\rho}-\rho)\sqrt K}\big)\right) \big(  \frac{\alpha}{\sqrt K}\EE\big(\phi(\vx^{(1)}) - \phi^* \big)  + \alpha^2\sigma^2\big),
\end{align}
where the first inequality follows from \eqref{eq:delay-x-tau-bd}, and the second inequality is from \eqref{eq:bd-x-diff0-prox-sum2} and the definition of $\tilde\gamma$. Now plug \eqref{eq:bd-x-diff0-prox-sum3} and the choice of $\{\alpha_k\}$ into \eqref{eq:prox-async} to obtain the desired result.
\end{proof}

\begin{remark}\label{rm:delay-prox}
We make a few remarks here about Theorem~\ref{thm:rate-prox-delay}. First, in the proof, we take $\gamma=\frac{1}{2}$ for simplicity while using \eqref{eq:bd-x-diff0-prox}. The analysis goes through for any $\gamma>0$ such that $1-\gamma-\frac{\alpha\rho}{2\sqrt K} - \frac{\tau^2\alpha^2\rho^2}{2\gamma K}- \frac{\beta}{K^{1/4}}>0$. Second, we see from \eqref{eq:prox-async-fix} that a positive $\tau$ will slow down the convergence but its effect will be reduced in an order of ${K^{-\frac{1}{4}}}$. Hence, if $K$ is big enough such that $K^{1/4} \gg \tau$, then the effect caused by the staleness is negligible. 
\end{remark}

The $O(\frac{1}{\sqrt{K}})$ convergence above is established by using a fixed stepsize sequence.  We can show a similar result for the choice of $\alpha_k=\Theta(\frac{1}{\sqrt k})$ by assuming condition~\ref{subgrad-bnd1} of Assumption~\ref{assump:subgrad-bound2}. The proof is given in Appendix~\ref{varying2}.
\begin{theorem}[convergence rate with varying stepsize]\label{thm:rate-prox-delay-vary}
Suppose Assumptions~\ref{assump:unbiased},  \ref{smoothness}, \ref{assump:variance} and \ref{assump:bound-tau}, and also  condition~\ref{subgrad-bnd1} of Assumption~\ref{assump:subgrad-bound2} hold. Let $\overline{\rho}>\rho$, 
 $\alpha_k = \frac{\alpha}{\sqrt{k+a-1}}$ and $\beta_k= \min\big\{\tilde\beta, \frac{\beta}{(k+a-1)^{1/4}}\big\}$, for all $k\ge1$, for some $\alpha>0$, $\beta\ge0$, $\tilde\beta\ge0$, and $a\ge 1$ such that 
\begin{equation}\label{eq:require-para-comp}
\textstyle \tilde\gamma := \frac{1}{2}\big( 1- \frac{\alpha \rho}{\sqrt{a}} - \tilde\beta^2 - \frac{2\tau^2\rho^2\alpha^2}{a}\big)>0.
\end{equation}
Let $\{\vx^{(k)}\}$ be the sequence from Alg.~\ref{alg:async-hvb-sgm}. Then,
\begin{equation}\label{eq:prox-async-vary}
\begin{aligned}
		\textstyle\EE\|\nabla\phi_{1/\overline{\rho}} (\vx^{(T)})\|^2  
		\le&~\textstyle\frac{8\overline{\rho}}{(\overline{\rho}-\rho)\alpha(\sqrt{K+a} - \sqrt{a})}\Bigg[\phi_{1/\overline{\rho}}(\vx^{(1)})- \phi^*+ \frac{\sigma^2\overline{\rho}\alpha^2}{2}  (1+\ln\frac{a+K-1}{a}) \\
&~\textstyle \hspace{-1.5cm}		+ \Big( \overline{\rho}\big( 2\tilde\beta^2 + \frac{4\beta^2}{\alpha(\overline{\rho}-\rho)} \big) + \frac{\tau^2\overline{\rho}\rho^2}{2} \big(\frac{\alpha^2}{a}+ \frac{ \alpha }{\sqrt{a}(\overline{\rho}-\rho)}\big)  \Big) \frac{2}{\tilde\gamma}\left( \alpha_1C_\phi + \sigma^2 \alpha^2  (1+\ln\frac{a+K-1}{a}) \right) \Bigg],
	\end{aligned}    
\end{equation}
where $T$ is randomly selected from $\{1,\ldots, K\}$ by \eqref{eq:def-T} with $k_0=1$.
\end{theorem}

\begin{remark}
When there is no delay, i.e., $\tau=0$, we can choose $a=1$ and obtain a convergence rate of $\tilde \Theta(\frac{1}{\sqrt{K}})$. When there is delay, i.e., $\tau\ge 1$, \eqref{eq:prox-async-vary} with $a=\Theta(\tau^4)$, which can ensure \eqref{eq:require-para-comp}, gives a rate of $\widetilde{\Theta}(\frac{1}{\sqrt{K+\tau^4}-\sqrt{\tau^4}})=\widetilde{\Theta}\Big( \frac{1}{\sqrt{K}}(\sqrt{1+\frac{\tau^4}{K}} +\sqrt{\frac{\tau^4}{K}})\Big) $. In this case, the delay will have a negligible effect on the convergence speed if $\tau=o(K^{\frac{1}{4}})$.
\end{remark}

\section{Convergence analysis for smooth nonconvex problems}\label{sec:smooth}
\newcommand{\vGam}{{\mathbf{\Gamma}}}
\newcommand{\vPhi}{{\mathbf{\Phi}}}
\newtheorem{fact}{Fact}

In this section, we consider the case where $r=0$, i.e., a non-regularized smooth problem. 
For this special case, we are able to show a stronger result under the same assumptions as we used in section~\ref{sec:prox-async}, in the sense that the delay has a weaker effect on the convergence speed. However, the analysis is significantly different from those in the previous two sections. 
Throughout this section, we let 
\begin{equation}\label{eq:special-set}
\textstyle \beta_k=\frac{\alpha_{k}}{\alpha_{k-1}}\beta,\text{ for all }k\ge 1\text{ and for some }\beta\in(0,1).
\end{equation} Then the update in \eqref{eq:update-x-2} reduces to \eqref{eq:update-x} with $\vm$-vectors defined in \eqref{eq:update-m}. We declare the following notation, as they appear extensively in this section: 
\begin{equation}\label{def:uk}
	\vu^{(k)}=\nabla F(\vx^{(k-\tau_k)} ) 
	\text{ and } u_k=\EE\|\vu^{(k)}\|^2 \text{ for all } k\ge 1.
\end{equation}


With the setting in \eqref{eq:special-set}, we define the following quantities that are critical for bounding the staleness:
\begin{equation}\label{def:theta}
\textstyle	\theta_{k,j}=\sum_{l=0}^{ \min\{\tau_k-1,\\k-j-1\}}\alpha_{k-l-1} \beta^{k-j-l-1}, \text{ and }  \pi_{k,j}(t)=\sum_{l=0}^{\min\{\tau_k-1,\\k-j-1\}} t^{k-j-l-1}.
\end{equation}

\begin{lemma}\label{lem:pi-bnd}
	Let $t\in (0,1)$, we have the following results: 
	\begin{equation}\label{eq:pi-sum1}
	\textstyle	 \pi_{k,j}(t)
	=\begin{cases}
		\frac{1-t^{k-j}}{1-t}  &\text{ if  } j\ge k-\tau_k+1 ,\\[0.1cm]
		\frac{1-t^{\tau_k}}{1-t}t^{k-\tau_k-j}   &\text{ if  } j\le k-\tau_k ;
	\end{cases} 
\quad \sum_{j=1}^{k-1} \pi_{k,j}(t)\le \frac{ \tau}{ 1-t };\quad \sum_{j=1}^{k-1} \pi_{k,j}^2(t)\le\frac{ \tau}{ (1-t)^2}.
	\end{equation}
\end{lemma}

\begin{lemma}\label{lem:hatgap-dist}
Let $\{\vx^{(k)}\}_{k\ge 1} $ and $\{\vm^{(k)}\}_{k\ge 1} $ be generated from \eqref{eq:update-x} and \eqref{eq:update-m}. Under Assumptions~\ref{assump:unbiased} and \ref{assump:variance},  it holds for $k\ge 1,$
\begin{equation}\label{vm-2mom}
	\textstyle\EE \|\vm^{(k)}\|^2\le  (1-\beta) \sum_{j=1}^k\beta^{k-j} u_j +(1-\beta)^2\sum_{j=1}^k\beta^{2(k-j)} u_j  + \frac{(1-\beta)^2}{1-\beta^2}    \sigma^2,
\end{equation}
	\begin{equation}\label{hatgap-2mom}
	\textstyle	\EE \|\vx^{(k-\tau_k)}- \vx^{(k)}\|^2\le    \sum_{l=1}^{k-1}\theta_{k,l}\sum_{j=1}^{k-1}\theta_{k,j} u_j  +\sum_{j=1}^{k-1}\theta_{k,j}^{2} u_j +       \sigma^2\sum_{j=1}^{k-1}\theta_{k,j}^{2}.
	\end{equation}
\end{lemma} 

In the remaining analysis, we follow the analytical framework of \cite{yan2016unified}. We define an auxiliary sequence $\vz^{(k)}$ as follows: 
\begin{equation}\label{defzk}
\textstyle	\vz^{(k)}=\vx^{(k)}+\frac{\beta}{1-\beta}(\vx^{(k)}-\vx^{(k-1)})=\frac{1}{1-\beta}\vx^{(k)}-\frac{\beta}{1-\beta}\vx^{(k-1)},\, \forall\, k\ge1.
\end{equation}
Recall $\vx^{(0)}=\vx^{(1)}$, so clearly, $\vz^{(1)}=\vx^{(1)}.$

\begin{lemma}\label{lem:zz}
	Let $\vz^{(k)}$ be defined as in \eqref{defzk} and $\alpha_0=\alpha_1.$ We have for $k\ge 1$,
	\begin{equation}\label{zzandxx}
	\textstyle	\vz^{(k+1)}-\vz^{(k)}=\frac{\beta}{1-\beta}(1 - \alpha_{k}  / \alpha_{k-1}   )(\vx^{(k-1)}-\vx^{(k)})-\frac{\alpha_{k}}{1-\beta} \vg^{(k)},
	\end{equation}
	and 
	\begin{equation}\label{eq:zxlip}
	\textstyle	\|\nabla F(\vz^{(k)})-\nabla F(\vx^{(k)})\|\leq \frac{\rho\beta }{1-\beta} \|\vx^{(k-1)}-\vx^{(k)}\|.
	\end{equation}
\end{lemma}

Now we are ready to show the main result. We first show the convergence by imposing general conditions on $\{\alpha_k\}$ and then specify the choice of the parameters that satisfies the imposed conditions.

\begin{theorem}\label{thm:ncv}
	Given a maximum number $K $ of iterations, let $\{\vx^{(k)}\}_{k=1}^K$ 
	be generated from Alg.~\ref{alg:async-hvb-sgm} with a non-increasing positive sequence $\{\alpha_{k}\}_{k=1}^{K}$. 
	Let $\bar\vx^{(K)} $ be drawn from $\{\vx^{(k)}\}_{k=1}^{K}$ with probability 
	\begin{equation}\label{eq:def-xK}
	\textstyle \Prob(\bar\vx^{(K)}= \vx^{(k)}) = \frac{\alpha_{k}}{\sum_{j=1}^{K}\alpha_{j}},\,\forall k=1,\ldots, K.
	\end{equation}
	Under Assumptions \ref{assump:unbiased} and \ref{smoothness}--\ref{assump:bound-tau}, if for all  $k\ge 2,$
	\begin{equation}\label{eq:vanish}
	\textstyle	(1-\alpha_{k}/\alpha_{k-1})^2\leq\frac{\alpha_{k}}{2(1-\beta)},
	\end{equation}
	and for all $j\ge1,$
	\begin{equation}\label{eq:main-cond}
	\textstyle	    \frac{3\rho\alpha_j}{1-\beta}    +    \rho^2    \left[ \frac{\tau(\tau-1)\alpha_1\alpha_{j}}{(1-\beta)^2} +\frac{ (\tau-1) \alpha_{j}^2}{(1-\beta)^2} 
		    + \frac{\tau\alpha_j^2}{(1-\beta)^3} + \frac{ \alpha_j^2}{(1-\beta)^2(1-\beta^2)} 
		    \right] 		
		+   \frac{2( 1+5\rho)  \beta^2}{ (1-\beta)^2(1-\beta^2) }  \alpha_{j} \le  1,
	\end{equation}
	then it holds 
	\begin{equation}\label{eq:main-ncv}
	\begin{aligned}
		\textstyle\EE\|\nabla F(\bar\vx^{(K)})\|^2  
		\le  \frac{4 \sigma^2}{ (1-\beta) \sum_{k=1}^K\alpha_{k}}\Big[ & \textstyle \frac{  \rho^2\tau}{2(1-\beta)}\sum_{k=1}^K\alpha_{k}\alpha_{\max\{k-\tau_k,1\}}^2
		+     \frac{(1+5\rho)\beta^2}{2(1-\beta^2 )} \sum_{k=1}^K \alpha_{k-1}^2  
		\\
		  & \textstyle + \rho\sum_{k=1}^K\alpha_{k}^2 \Big] + \frac{4(1-\beta)\left[ F(\vx^{(1)})-\inf_{\vx} F(\vx)\right] }{ \sum_{k=1}^K\alpha_{k}}. 
		\end{aligned}
	\end{equation}
\end{theorem}

\begin{proof}
	By the $\rho$-smoothness of $F,$ it follows from \eqref{eq:ineq-smooth} that
	\begin{align}\label{eq:break}
		0\leq &~\textstyle F(\vz^{(k)})-F(\vz^{(k+1)})  + \nabla F(\vz^{(k)})^{\top} (\vz^{(k+1)}-\vz^{(k)}) + \frac{\rho}{2}\|\vz^{(k+1)}-\vz^{(k)}\|^2\nonumber\\
		= &~\textstyle F(\vz^{(k)})-F(\vz^{(k+1)}) +\nabla F(\vx^{(k)})^{\top} (\vz^{(k+1)}-\vz^{(k)}) \nonumber\\
		&~\textstyle+ (\nabla F(\vz^{(k)})-\nabla F(\vx^{(k)}))^{\top} (\vz^{(k+1)}-\vz^{(k)}) + \frac{\rho}{2}\|\vz^{(k+1)}-\vz^{(k)}\|^2. 
	\end{align}
Taking the conditional expectation and using \eqref{zzandxx} and Assumption~\ref{assump:unbiased}, we have from \eqref{eq:break} that 
	\begin{align}\label{eq:break2}
		0\leq &~\textstyle\EE_k[ F(\vz^{(k)})-F(\vz^{(k+1)})]  + \nabla F(\vx^{(k)})^{\top} \big(\frac{\beta}{1-\beta} (1 - \alpha_{k}  / \alpha_{k-1}   )(\vx^{(k-1)}-\vx^{(k)})-\frac{\alpha_{k}}{1-\beta} \vu^{(k)}\big)\nonumber\\
		&~\textstyle+ (\nabla F(\vz^{(k)})-\nabla F(\vx^{(k)}))^{\top} \big(\frac{\beta}{1-\beta}(1 - \alpha_{k}  / \alpha_{k-1}   )(\vx^{(k-1)}-\vx^{(k)})-\frac{\alpha_{k}}{1-\beta} \vu^{(k)}\big) \nonumber\\ 
		&~\textstyle+\frac{\rho}{2}\EE_k\big\|\frac{\beta}{1-\beta}(1 - \alpha_{k}  / \alpha_{k-1}   )(\vx^{(k-1)}-\vx^{(k)})-\frac{\alpha_{k}}{1-\beta}  \vg^{(k)}\big\|^2.
	\end{align}
	We bound the right-hand side of \eqref{eq:break2} as follows:
	\begin{itemize}
		\item in the first line of \eqref{eq:break2}, applying the Cauchy-Schwarz inequality gives
		\[\textstyle	\nabla F(\vx^{(k)})^{\top} \frac{\beta}{1-\beta} (1 - \alpha_{k}  / \alpha_{k-1}   )(\vx^{(k-1)}-\vx^{(k)}) \le \frac{1}{2}(1 -\frac{\alpha_{k}}{\alpha_{k-1}})^2 \|\nabla F(\vx^{(k)})\|^{2} +  \frac{\beta^2}{2(1-\beta)^2}\|\vx^{(k-1)}-\vx^{(k)}\|^2; \]
		\item in the second line of \eqref{eq:break2}, it follows  from \eqref{eq:zxlip} and $ 0\le 1-\frac{\alpha_{k} }{\alpha_{k-1}}\le 1$ that
		\[\textstyle (\nabla F(\vz^{(k)})-\nabla F(\vx^{(k)}))^{\top}  \frac{\beta}{1-\beta}(1 - \alpha_{k}  / \alpha_{k-1}   )(\vx^{(k-1)}-\vx^{(k)}) \le
		\frac{\rho\beta^2 }{(1-\beta)^2}
		\|\vx^{(k-1)}-\vx^{(k)}\|^2		\]
		and in addition, by the Cauchy-Schwarz inequality,
		\begin{align*}
			\textstyle (\nabla F(\vz^{(k)})-\nabla F(\vx^{(k)}))^{\top} (-\frac{\alpha_{k}}{1-\beta} \vu^{(k)})
			&\le \textstyle
			\rho\cdot\frac{\beta }{1-\beta}\|\vx^{(k-1)}-\vx^{(k)}\|\cdot\frac{\alpha_{k}}{1-\beta}  \|\vu^{(k)}\| \\
			&\le \textstyle
			\frac{\rho\beta^2 }{2(1-\beta)^2}
			\|\vx^{(k-1)}-\vx^{(k)}\|^2	 + \frac{\rho\alpha_{k}^2}{2(1-\beta)^2}  \|\vu^{(k)}\|^2;
		\end{align*}		
		\item in the last line of \eqref{eq:break2},  using the Young's inequality gives
		\[\textstyle \frac{\rho}{2} \EE_k \|\frac{\beta}{1-\beta}(1 - \alpha_{k}  / \alpha_{k-1}   )(\vx^{(k-1)}-\vx^{(k)})-\frac{\alpha_{k}}{1-\beta}  \vg^{(k)}  \|^2
		\le \frac{\rho\beta^2}{(1-\beta)^2}
		\|\vx^{(k-1)}-\vx^{(k)}\|^2+\frac{\rho\alpha_{k}^2}{(1-\beta)^2} \EE_k\|  \vg^{(k)}\|^2;\]
		\item  furthermore, by Assumption~\ref{assump:variance}, 
		 we have 
		\begin{equation*}
			\EE_k\|  \vg^{(k)}\|^2
			= \EE_k \|   \nabla f(\vx^{(k-\tau_k)}; \xi_k)
			-\vu^{(k)}\|^2  + \|\vu^{(k)}\|^2 
			\le 
			\sigma^2 + \|\vu^{(k)}\|^2.   
		\end{equation*}	
	\end{itemize}
	Substitute the above four items into \eqref{eq:break2}, combine like terms, and take total expectation. We have  
	\begin{align}\label{eq:break3}
		0\le &~\textstyle\EE[ F(\vz^{(k)})-F(\vz^{(k+1)})]  + \frac{1}{2}(1 -\frac{\alpha_{k}}{\alpha_{k-1}})^2\EE \|\nabla F(\vx^{(k)})\|^{2} -   \frac{\alpha_{k}}{1-\beta}   \EE [\nabla F(\vx^{(k)})^{\top} \vu^{(k)}]  \nonumber\\
		&~\textstyle+   \frac{(1+5\rho)\beta^2}{2(1-\beta)^2}  \EE\|\vx^{(k-1)}-\vx^{(k)}\|^2 
		+ \frac{\rho\alpha_{k}^2\sigma^2}{(1-\beta)^2 } 
		 +\frac{3\rho\alpha_{k}^2}{2(1-\beta)^2}    u_k \nonumber \\
		= &~\EE[ F(\vz^{(k)})-F(\vz^{(k+1)})]  + \frac{1}{2}(1 -\frac{\alpha_{k}}{\alpha_{k-1}})^2\EE \|\nabla F(\vx^{(k)})\|^{2}   \\
		&~\textstyle-   \frac{\alpha_{k}  }{2(1-\beta)} \left[\EE\|\nabla F(\vx^{(k)})\|^2 + u_k - \EE\| \nabla F(\vx^{(k-\tau_k)} )-\nabla F(\vx^{(k)}) 
		\|^2\right]  \nonumber\\
		&~\textstyle+   \frac{(1+5\rho)\beta^2}{2(1-\beta)^2}  \EE\|\vx^{(k-1)}-\vx^{(k)}\|^2 
		+ \frac{\rho\alpha_{k}^2\sigma^2}{(1-\beta)^2 } +\frac{3\rho\alpha_{k}^2}{2(1-\beta)^2}    u_k,  \nonumber
	\end{align}
	where the equality is due to 
	$\va^\top\vb=\frac{1}{2}[\|\va\|^2+\|\vb\|^2-\|\va-\vb\|^2],$ for any two vectors $\va$ and $\vb$.

	 Using \eqref{eq:update-x} and the smoothness of $F$ and then substituting 
	\eqref{vm-2mom} and \eqref{hatgap-2mom}  to \eqref{eq:break3}, we have 
	\begin{align}\label{eq:break4}
		0\le &~\textstyle\EE[ F(\vz^{(k)})-F(\vz^{(k+1)})]  + \frac{1}{2}\left[(1 -\frac{\alpha_{k}}{\alpha_{k-1}})^2 -\frac{\alpha_{k}}{1-\beta}   \right]\EE \|\nabla F(\vx^{(k)})\|^{2}  \nonumber \\
		&~+\textstyle   \frac{\alpha_{k}  \rho^2}{2(1-\beta )}  \Big[\sum_{l=1}^{k-1}\theta_{k,l}\sum_{j=1}^{k-1}\theta_{k,j} u_j
		 +\sum_{j=1}^{k-1}\theta_{k,j}^{2} u_j +    \sigma^2\sum_{j=1}^{k-1}\theta_{k,j}^{2}\Big] +\frac{1}{2}\big( \frac{3\rho\alpha_{k}^2}{(1-\beta)^2}    -\frac{\alpha_{k}}{1-\beta} \big)  u_k
		\nonumber\\
		&~+\textstyle    \frac{(1+5\rho)\beta^2\alpha_{k-1}^2}{2(1-\beta)^2}  
		\Big[ \sum_{j=1}^{k-1}\big(\frac{\beta^{k-j-1}}{1-\beta} +\beta^{2(k-j-1)}\big) u_j  +  
		\frac{\sigma^2}{ 1-\beta^2}\Big]    
		+ \frac{\rho\alpha_{k}^2\sigma^2}{(1-\beta)^2 } . 
	\end{align}
	Summing the above inequality over $k=1,\ldots,K$ and utilizing \eqref{eq:vanish} lead to
	\begin{equation}
	\begin{aligned}\label{eq:tele1}
		0\le &~\textstyle F(\vx^{(1)})-\EE[F(\vz^{(k+1)})] - \frac{1}{4(1-\beta)}\sum_{k=1}^K \alpha_{k}  \EE \|\nabla F(\vx^{(k)})\|^{2}  +\frac{1}{2}\sum_{k=1}^K\big( \frac{3\rho\alpha_{k}^2}{(1-\beta)^2}    -\frac{\alpha_{k}}{1-\beta} \big) u_k  \\
		&~\textstyle+   \frac{   \rho^2}{2(1-\beta )} \sum_{k=1}^K\alpha_{k}\Big[\sum_{l=1}^{k-1}\theta_{k,l}\sum_{j=1}^{k-1}\theta_{k,j} u_j +\sum_{j=1}^{k-1}\theta_{k,j}^{2} u_j \Big] 	\\
		&~\textstyle+     \frac{(1+5\rho)\beta^2}{2( 1-\beta )^2}   \sum_{k=1}^K\alpha_{k-1}^2 
		  \sum_{j=1}^{k-1}\big(\frac{\beta^{k-j-1}}{1-\beta}+\beta^{2(k-j-1)}\big) u_j \\
		&~+\textstyle   \Big[ \frac{   \rho^2}{2(1-\beta )} \sum_{k=1}^K\alpha_{k}\sum_{j=1}^{k-1}\theta_{k,j}^{2}
		+     \frac{(1+5\rho)\beta^2}{2(1-\beta)^2(1-\beta^2)}  \sum_{k=1}^K   \alpha_{k-1}^2 + \frac{\rho}{(1-\beta)^2}\sum_{k=1}^K\alpha_{k}^2  	\Big]  
		\sigma^2. 
	\end{aligned}
	\end{equation}
Since $\{\alpha_{k}\}$ is non-increasing, it holds from \eqref{def:theta} that 
\begin{equation}\label{eq:theta-pi}
	\theta_{k,j}  \le \alpha_{\max\{k-\tau_k,j\}} \pi_{k,j}(\beta),
\end{equation} 
which together with {the two inequalities in} \eqref{eq:pi-sum1} gives
\begin{equation}\label{eq:theta-pi2}
\textstyle	\sum_{j=1}^{k-1} \theta_{k,j} \le \alpha_{\max\{k-\tau_k,1\}} \frac{ \tau}{ 1-\beta},
	\textup{ and }
\sum_{j=1}^{k-1} \theta_{k,j}^2 \le \alpha_{\max\{k-\tau_k,1\}}^2 \frac{ \tau}{ (1-\beta)^2}.
\end{equation}
Plugging the latter inequality of \eqref{eq:theta-pi2} into the fourth line of \eqref{eq:tele1}, and also interchanging the summations in the second and third lines of \eqref{eq:tele1} yield 
	\begin{align}\label{eq:tele2}
		0\le &~\textstyle F(\vx^{(1)})-\inf_{\vx} F(\vx) - \frac{1}{4(1-\beta)}\sum_{k=1}^K\alpha_{k}  \EE \|\nabla F(\vx^{(k)})\|^{2}  +\frac{1}{2}\sum_{k=1}^K\big( \frac{3\rho\alpha_{k}^2}{(1-\beta)^2}    -\frac{\alpha_{k}}{1-\beta} \big) u_k \nonumber \\
		&~\textstyle+   \frac{   \rho^2}{2(1-\beta )}  \sum_{j=1}^{K-1}u_j\sum_{k=j+1}^K\alpha_{k} \theta_{k,j}\big( \theta_{k,j} + \sum_{l=1}^{k-1}\theta_{k,l}\big)
		+     \frac{(1+5\rho)\beta^2}{(1-\beta)^3(1-\beta^2)}  \sum_{j=1}^{K-1} u_j \alpha_{j}^2 \nonumber\\
		&~\textstyle+   \Big[ \frac{  \rho^2\tau}{2(1-\beta) }\sum_{k=1}^K\alpha_{k}\alpha_{\max\{k-\tau_k,1\}}^2
		+     \frac{(1+5\rho)\beta^2}{2(1-\beta^2) }  \sum_{k=1}^K  \alpha_{k-1}^2 + \rho\sum_{k=1}^K\alpha_{k}^2  \Big] 
		\frac{\sigma^2}{(1-\beta)^2 },
	\end{align}
where the last summation in the second line is simplified by utilizing the following summation bound, $$\textstyle \sum_{k=1}^K\alpha_{k-1}^2\sum_{j=1}^{k-1} t^{k-j-1} u_j = \sum_{j=1}^{K-1}  u_j\sum_{k=j+1}^{K}\alpha_{k-1}^2 t^{k-j-1} \le \sum_{j=1}^{K-1}  u_j\alpha_{j}^2/(1-t).$$
Furthermore,
	\begin{align}
	&~\textstyle \sum_{k=j+1}^K \alpha_{k} \theta_{k,j} \big( \theta_{k,j}+ \sum_{l=1}^{k-1}\theta_{k,l}\big)\nonumber\\
	\le&~\alpha_{j} \textstyle \sum_{k=j+1}^K \alpha_{\max\{k-\tau_k,j\}} \pi_{k,j}(\beta) \big(\alpha_{\max\{k-\tau_k,j\}} \pi_{k,j}(\beta)+ \frac{\alpha_{\max\{k-\tau_k,1\}}\tau}{1-\beta}\big) \nonumber\\
	 \le&~ \textstyle \alpha_{j}   
	\sum_{k\colon j+1\le k \le K,\,   k\le j+\tau-1 } \alpha_j \frac{1}{ 1-\beta } \big(\alpha_j \frac{1}{ 1-\beta }+ \frac{\alpha_1\tau}{1-\beta}\big) \cr
	 &~ \textstyle + \alpha_{j} \sum_{k\colon j+1\le k \le K,\,   k\ge j+\tau} \alpha_{j} \frac{\beta^{k-\tau-j}}{1-\beta} \big(\alpha_{j} \frac{\beta^{k-\tau-j}}{1-\beta}+ \frac{\alpha_{j}\tau}{1-\beta}\big) \nonumber \\
	 \le&~ \textstyle \frac{\tau(\tau-1)\alpha_1\alpha_{j}^2}{(1-\beta)^2}  +\frac{ (\tau-1) \alpha_{j}^3}{(1-\beta)^2} 
	+ \frac{\tau\alpha_j^3}{(1-\beta)^3}  + \frac{ \alpha_j^3}{(1-\beta)^2(1-\beta^2)} .\label{eq:theta-sum-3}
\end{align}
In the above, the first inequality follows from $\alpha_{k}\le\alpha_{j}$ for all $k\ge j$, \eqref{eq:theta-pi}  and \eqref{eq:theta-pi2}; 
the second inequality breaks the summation on $k$ into two parts: in the first part $k\le j+\tau-1$, we used $\pi_{k,j}(\beta)\le \frac{1}{ 1-\beta }$ by  \eqref{def:theta} and also $ \alpha_{\max\{k-\tau_k,j\}}\le\alpha_{j}$ and $\alpha_{\max\{k-\tau_k,1\}}\le\alpha_1 $; and in the second part $k\ge j+\tau$, since $k\ge j+\tau_k$, we have  $\pi_{k,j}(\beta)\le \frac{\beta^{k-\tau-j}}{ 1-\beta }$ from the second case in the equality of \eqref{eq:pi-sum1} 
and also, $ \alpha_{\max\{k-\tau_k,j\}}=\alpha_{\max\{k-\tau_k,1\}}\le\alpha_{j} $.

	Now substitute \eqref{eq:theta-sum-3} into \eqref{eq:tele2}, use the assumption in \eqref{eq:main-cond} to drop the non-positive terms about $u_j$, also use the definition of $\bar\vx^{(K)}$ in \eqref{eq:def-xK}, and then rearrange terms to obtain the desired result in \eqref{eq:main-ncv}.
\end{proof}



Below we specify the setting of $\{\alpha_{k}\}$ and show the sublinear   convergence.

\begin{corollary}\label{cor:const}
	Given a maximum number $K $ of iterations, let $\alpha_{k}=\alpha/\sqrt{K}$ for all $k=1,\ldots,K,$ and for some $\alpha>0.$ 
	If $\alpha>0$ and $\beta>0$ are chosen such that 
	\begin{equation}\label{cor:tau-cond-2}
\textstyle		\tau^2      + \frac{\tau}{ 1-\beta}  +\frac{\beta^2}{1-\beta^2}   \le 
		\frac{ (1-\beta)^2 K  }{  2\alpha^2   \rho^2},
		\text{ and } 
		3\rho+ \frac{2(1+5\rho)\beta^2}{(1-\beta)(1-\beta^2)}  \le    \frac{(1-\beta)\sqrt{K} } {2\alpha  },
	\end{equation}
	then under Assumptions \ref{assump:unbiased} and \ref{smoothness}--\ref{assump:bound-tau}, the iterate $\bar\vx^{(K)}$ given in \eqref{eq:def-xK} 
	 satisfies 
	\begin{equation}\label{cor:converge}
		\EE\|\nabla F(\bar\vx^{(K)})\|^2  
		\le  \textstyle   \left(  \frac{\rho^2\alpha\tau}{2(1-\beta)\sqrt{K}}+  \frac{(1+5\rho)\beta^2}{2(1-\beta^2)}+\rho \right) 
		\frac{4\alpha  \sigma^2 }{(1-\beta) \sqrt{K}}   
		+ \frac{4(1-\beta)\left[ F(\vx^{(1)})-\inf_{\vx} F(\vx)\right] }{  \alpha \sqrt{K} }.
	\end{equation}
\end{corollary}
\begin{proof}
	When $\alpha_{k}\equiv\alpha/\sqrt{K}$, \eqref{eq:vanish} is trivially true, and in addition, when \eqref{cor:tau-cond-2} hold, it is not hard to verify  	
	\begin{equation}\label{eq:twohalfs}		 
\textstyle		\rho^2     \left[ \frac{\tau(\tau-1)\alpha_1\alpha_{j}}{(1-\beta)^2} +\frac{ (\tau-1) \alpha_{j}^2}{(1-\beta)^2} 
		+ \frac{\tau\alpha_j^2}{(1-\beta)^3} + \frac{ \alpha_j^2}{(1-\beta)^2(1-\beta^2)} 
		\right]\le  \frac{1}{2},
		\text{ and } 
		\left[ \frac{3\rho }{1-\beta} +  \frac{2(1+5\rho)\beta^2}{(1-\beta)^2(1-\beta^2)} \right]  \alpha_{j}   \le  \frac{1}{2},
	\end{equation}
which implies \eqref{eq:main-cond}. Finally, \eqref{eq:main-ncv} simplifies to \eqref{cor:converge}.
\end{proof}

\begin{remark}
From \eqref{eq:twohalfs}, we see that the delay can reduce the convergence speed of Alg.~\ref{alg:async-hvb-sgm} by roughly  $O(\frac{\tau}{\sqrt K})$.  When $\tau=o( \sqrt{K}    )$, the slow-down effect is negligible.  
\end{remark}

\begin{corollary}\label{cor:variant}
	Given a maximum number $K $ of iterations, let $\alpha_{k}=\alpha/\sqrt{a+k-1}$ for all $k=1,\ldots,K,$ and for some $a\ge 2\tau $ 
	 such that $a\sqrt{a+1} \ge \frac{ 1-\beta }{2\alpha}$. 
	If 
	\begin{equation}\label{cor:tau-cond-22}
	\textstyle	
		\tau^2      + \frac{\tau}{ 1-\beta}  +\frac{\beta^2}{1-\beta^2}   \le 
		\frac{ (1-\beta)^2 a  }{  2\alpha^2   \rho^2},
		\text{ and } 
		3\rho+ \frac{2(1+5\rho)\beta^2}{(1-\beta)(1-\beta^2)}   \le    \frac{(1-\beta)\sqrt{a} } {2\alpha  },
	\end{equation}
	then under Assumptions \ref{assump:unbiased} and \ref{smoothness}--\ref{assump:bound-tau}, the output of Alg.~\ref{alg:async-hvb-sgm} satisfies 
	\begin{equation}\label{cor:converge2}
	\begin{aligned}
\textstyle		\EE\|\nabla F(\bar\vx^{(K)})\|^2  		
		\le &~\textstyle \frac{2(1-\beta)\left[ F(\vx^{(1)})-\inf_{\vx} F(\vx)\right] }{   \alpha (\sqrt{a+K}-\sqrt{a}) } \\
		&~ \textstyle \hspace{-1cm}+ \Big[    \frac{   \rho^2\alpha (1+2a)\tau }{(1-\beta)a\sqrt{a}}  
		+     \frac{(1+5\rho)\beta^2}{2(1-\beta^2)}  (2+\ln\frac{a+K-2}{a})	
		+ \rho  (1+\ln\frac{a+K-1}{a}) 
		\Big]   
		\cdot\frac{2\alpha\sigma^2}{(1-\beta)  (\sqrt{a+K}-\sqrt{a})}	.
	\end{aligned}
	\end{equation}
\end{corollary}

\begin{remark}
	Note that the logarithmic terms in \eqref{cor:converge2} dominate the $\tau$-related term if $\tau\le\frac{\sqrt{a-1}}{\alpha\rho}$, which matches the condition in \eqref{cor:tau-cond-22}. When there is no delay, i.e., $\tau=0$, a convergence rate of $\widetilde{\Theta}(\frac{1}{\sqrt{K}})$ can be achieved with $a=1$; when there is a delay, i.e., $\tau>0$, \eqref{cor:converge2} with $a=\Theta(\tau^2)$ gives a rate of $\widetilde{\Theta}(\frac{1}{\sqrt{K+\tau^2}-\sqrt{\tau^2}})=\widetilde{\Theta}\Big( \frac{1}{\sqrt{K}}(\sqrt{1+\frac{\tau^2}{K}} +\sqrt{\frac{\tau^2}{K}})\Big) $. In this case, the delay will have a negligible effect on the convergence speed if $\tau=o(\sqrt{K})$.
\end{remark}

 \section{Numerical experiments}\label{sec:numerical}
In this section, we test Alg.~\ref{alg:async-hvb-sgm} by numerical experiments on three examples: phase retrieval problem,   neural network  training, and sparse bilinear logistic regression.
For each example, we test the effect of the inertial force with different $\beta_k$. Also, we demonstrate the advantage of the asynchronous implementation over the synchronous version (i.e., $\tau_k=0,\forall\,k$) of Alg.~\ref{alg:async-hvb-sgm}. In all the tests, we compare the performance of Alg.~\ref{alg:async-hvb-sgm} with different settings of $\{\alpha_k\}$ and $\{\beta_k\}$, which are fixed to constants for all iterations $k$ or decrease with respect to the number of epochs. 
 
\subsection{Phase retrieval problem}
The phase retrieval problem aims to recover a signal $\vx^*\in \RR^d$ from $m$ measuring vectors\footnote{In general, the signal $\vx$ and the measuring vectors $\{\va_i\}$ can be complex-valued. For simplicity, we focus on the real field.} 
$\{\va_i\}_{i=1}^m $ and the correspondingly  obtained magnitudes $\big\{b_i\big\}_{i=1}^m$. 
 It can be formulated into the following non-smooth minimization problem 
 \cite{eldar2014phase,duchi2019solving,davis2020nonsmooth}:

\begin{equation}\label{eq:pr-prob}
\min_{\vx\in \RR^{d}}\frac{1}{m}\sum_{i=1}^m \left| |\langle\va_i,\vx\rangle|^2-b_i^2\right|,
\end{equation}
which is in the form of \eqref{eq:stoc-prob} with $F(\vx)=\frac{1}{m}\sum_{i=1}^m \big| |\langle\va_i,\vx\rangle|^2-b_i^2\big|$ and $r(\vx)\equiv0$. 
In the test, the vector $\va_i$ followed the standard multivariate Gaussian distribution, i.e., $\va_i\sim \mathcal{N}(\vzero, \vI)$, and we let  $b_i=|\langle \va_i,\vx^*\rangle|, \forall\, i$, for a ground truth $\vx^*$.
 Hence, the optimal objective value is \emph{zero}.
%

\begin{figure}[htbp] 
\begin{center}
\includegraphics[width=0.7\columnwidth]{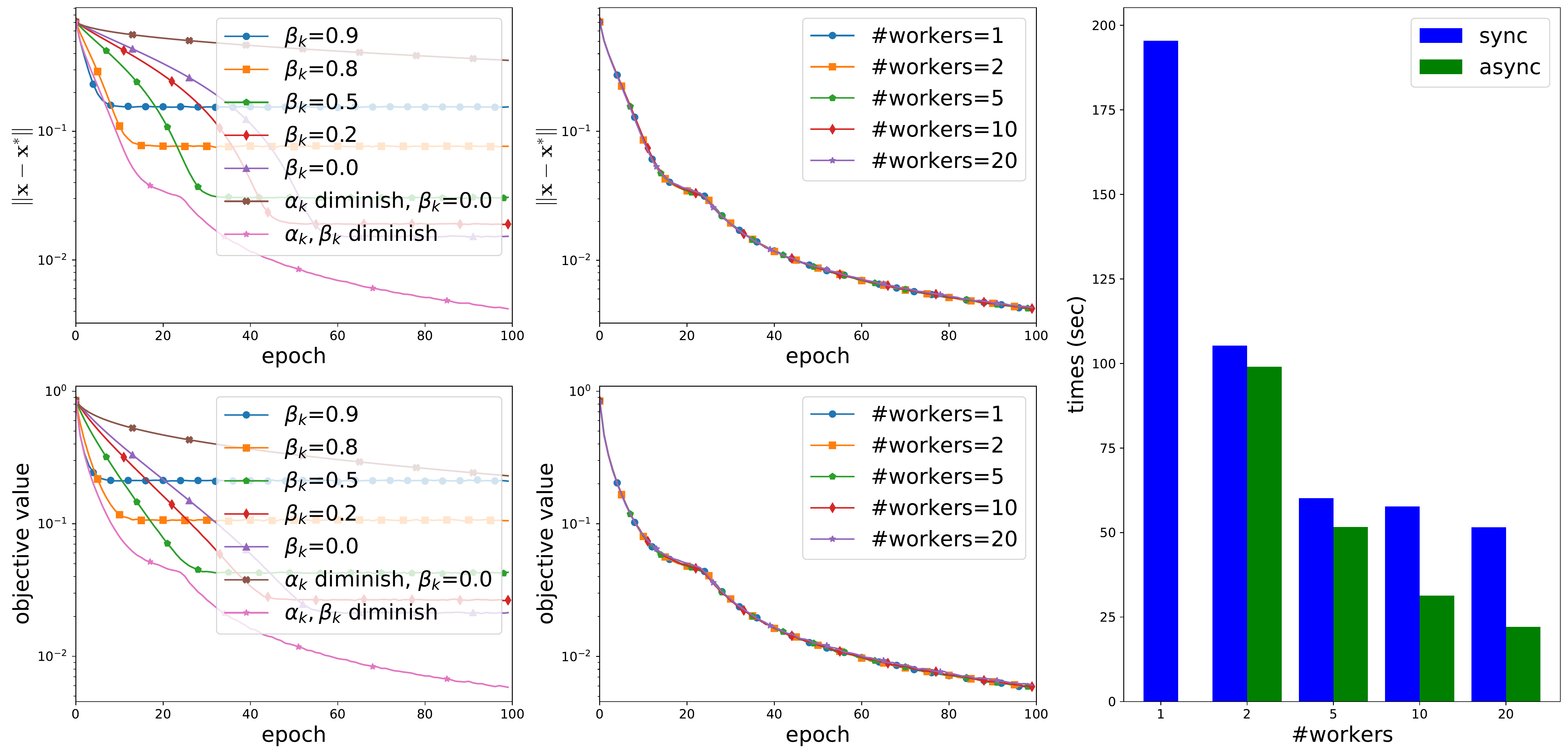} 
\caption{Results by Alg.~\ref{alg:async-hvb-sgm} on solving instances of the Phase Retrieval Problem \eqref{eq:pr-prob} with randomly generated $\vx^*$, $m=50,000$ and $d=20,000$. Left: 
non-parallel implementation of Alg.~\ref{alg:async-hvb-sgm} with different choices of $\{\alpha_k\}$ and $\{\beta_k\}$;  
Middle: async-parallel implementation of Alg.~\ref{alg:async-hvb-sgm} with diminishing $\{\alpha_k\}$ and $\{\beta_k\}$, and with different numbers of workers; Right: running time (in second) of the sync-parallel and async-parallel implementation of Alg.~\ref{alg:async-hvb-sgm} with different numbers of workers.  
}
\label{fig:PhaseRetrieval}
\end{center} 
\end{figure}

\vspace{0.1cm}

\noindent\textbf{Synthetic data.}~We first solved \eqref{eq:pr-prob} with $\vx^*$ generated from a uniform distribution on the $d$-dimensional unit sphere. Fig.~\ref{fig:PhaseRetrieval} shows the results for $m=50,000$ and $d=20,000$. We tested the algorithm for several pairs of $(m, d)$ and observed similar results. In the test, we computed a stochastic subgradient by using 100 data points, i.e., the minibatch size was set to 100. The parameters either followed a constant scheme with $\alpha_k = \alpha, \beta_k = \beta, \forall\, k$ where $\alpha = 5\times10^{-5}$ and $\beta\in \{0, 0.2, 0.5, 0.8, 0.9\}$; or diminished with $\alpha_k=\frac{5\times10^{-5}}{\sqrt{e_k+1}}$ and $\beta_k=\min\big\{0.9, \frac{2}{ (e_k+1)^{1/4}} \big\}, \forall\, k$, or $\beta_k=0, \forall\, k$. Here, $e_k$ denotes the epoch number at the $k$-th iteration. 
 During the test, we also experimented with different values of the constant $\alpha$. We found that for a smaller $\alpha$, the algorithm converged more slowly but could reach a lower objective value. 
The choice $\alpha = 5\times10^{-5}$ resulted in a good trade-off between the convergence speed and the final objective value.
%

From the left subfigure in Fig.~\ref{fig:PhaseRetrieval}, we see that the algorithm with a bigger $\beta$ converged faster but achieved a higher objective value.  
The convergence of the algorithm with a diminishing $\{\alpha_k\}$ and constant $\beta_k=0$ is the slowest. 
The best results were obtained by the choice of diminishing $\{\alpha_k\}$ and $\{\beta_k\}$.  Comparing the curve with diminishing $\{\alpha_k\}$ and $\{\beta_k\}$ to that with $\beta_k=0.9, \forall\, k$, we notice that the two curves are almost the same 
within the first 5 epochs, i.e., before the latter one becomes flat. 
However, the former can decrease the objective to a significantly smaller value. Thus both the choices of $\{\alpha_k\}$ and  $\{\beta_k\}$ contribute to the best results.
With the diminishing $\{\alpha_k\}$ and $\{\beta_k\}$ that yield the best results for the non-parallel case, we then compared the sync-parallel and async-parallel implementations of  Alg.~\ref{alg:async-hvb-sgm}. 
The middle subfigure in Fig.~\ref{fig:PhaseRetrieval} shows the results for the async-parallel version with different numbers of workers. The right subfigure shows the running time of both versions. The results show that the convergence speed (in terms of epoch number) of the async-parallel method is almost never affected by the asynchrony (or information delay). In addition, we see that the async-parallel implementation yielded significantly higher parallelization speed-up over the sync-parallel one, according to the right subfigure in Fig.~\ref{fig:PhaseRetrieval}. 

\begin{figure}[htbp] 
\begin{center} 
\includegraphics[width=0.42\columnwidth]{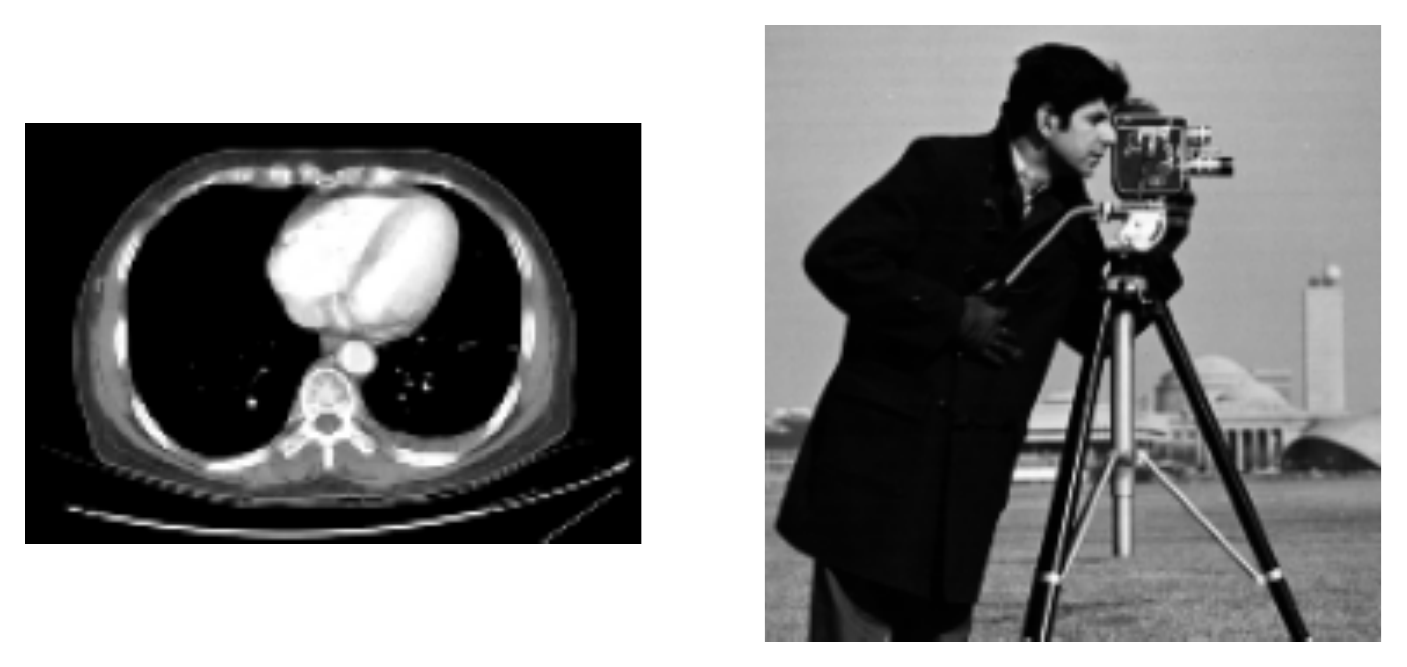}
\caption{Ground-truth images. Left: a CT scan image. Right: the cameraman image.}
\label{fig:image_exact}
\end{center} 
\end{figure}  

\vspace{0.1cm}

\noindent\textbf{Image data.}~
We also solved \eqref{eq:pr-prob} with $\vx^*$ flattened from an image. We tested with two images: a CT scan image\footnote{https://aimi.stanford.edu/radiopaedia-list-ai-imaging-datasets} of size $94 \times 138$ after downsampling and the cameraman image\footnote{https://github.com/antimatter15/cameraman} of size $196\times 196$ after cropping.  Fig.~\ref{fig:image_exact} shows the ground-truth images,  Fig.~\ref{fig:radiopaedia} and Fig.~\ref{fig:cameraman} show convergence curves and computing times, and Fig.~\ref{fig:radiopaedia_finals} and Fig.~\ref{fig:cameraman_finals} show recovered images.
In the test, for the CT scan image, 
$d=12,972$, and we selected $m=40,000$, computed each stochastic subgradient by using 100 randomly sampled data points, and set $\alpha_k=\frac{10^{-4}}{\sqrt{e_k+1}}$; for the cameraman image, 
$d=38,416$, and we selected $m=60,000$, computed each stochastic subgradient by using 60 randomly sampled data points and set $\alpha_k=\frac{5\times10^{-5}}{\sqrt{e_k+1}}$. 
We first tested the non-parallel version of Alg.~\ref{alg:async-hvb-sgm} with $\beta_k=\beta, \forall k$, where $\beta\in \{0, 0.2, 0.5, 0.8, 0.9\}$, and then tested the parallel version by different numbers of workers.

\begin{figure}[htbp] 
\begin{center} 
\includegraphics[width=0.7\columnwidth]{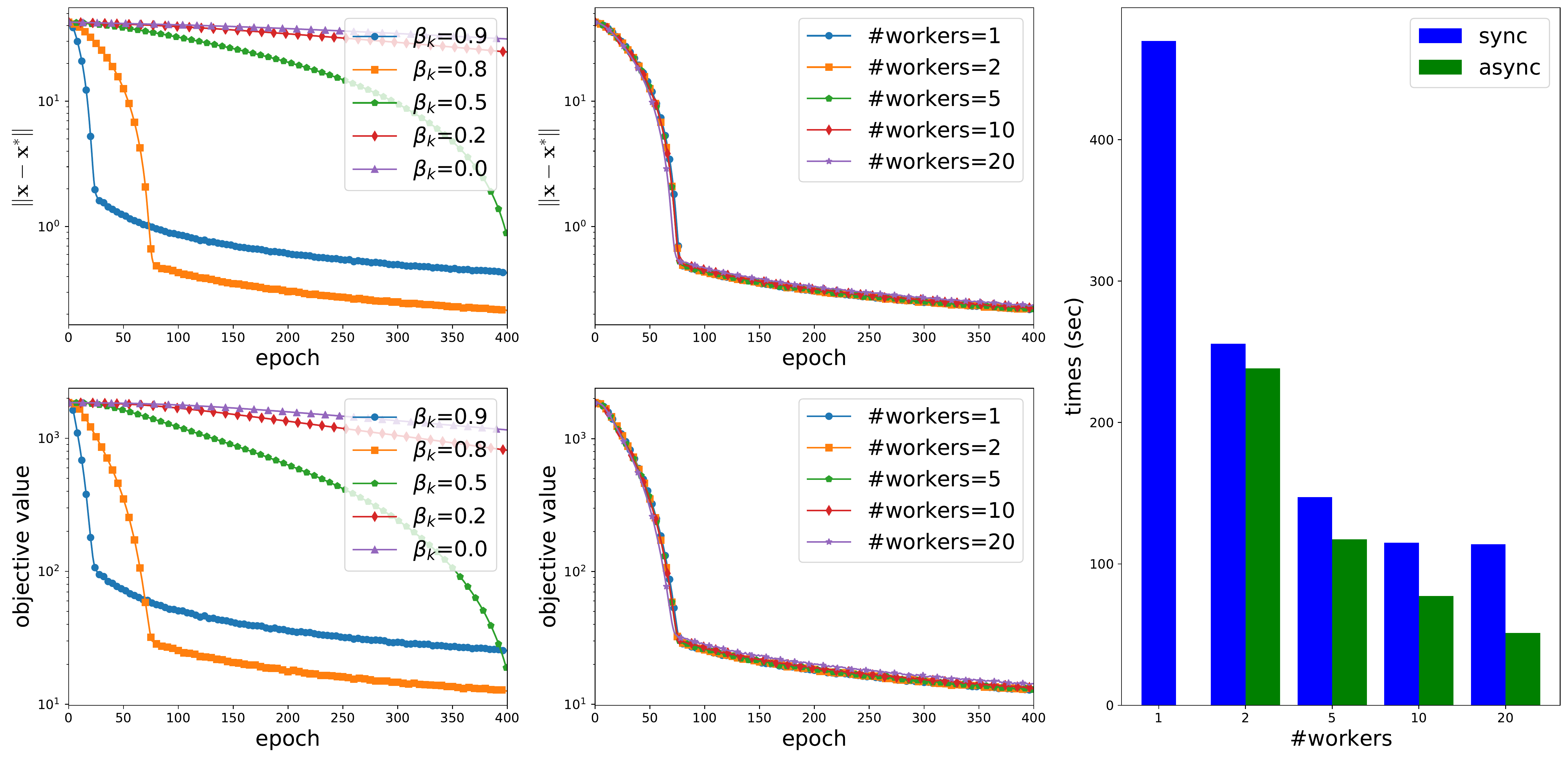}
\caption{Results by Alg.~\ref{alg:async-hvb-sgm} on solving instances of the Phase Retrieval Problem \eqref{eq:pr-prob} with a CT scan image as $\vx^*$ and $m=40,000$. 
Left: non-parallel implementation of Alg.~\ref{alg:async-hvb-sgm} with diminishing $\{\alpha_k\}$ and different choices of $\{\beta_k\}$;  
Middle: async-parallel implementation of Alg.~\ref{alg:async-hvb-sgm} with diminishing $\{\alpha_k\}$ and $\beta_k=0.8$, and with different numbers of workers; Right: running time (in second) of the sync-parallel and async-parallel implementation of Alg.~\ref{alg:async-hvb-sgm} with different numbers of workers.}
\label{fig:radiopaedia}
\end{center} 
\end{figure}  
\begin{figure}[htbp] 
\begin{center} 
\includegraphics[width=0.96\columnwidth]{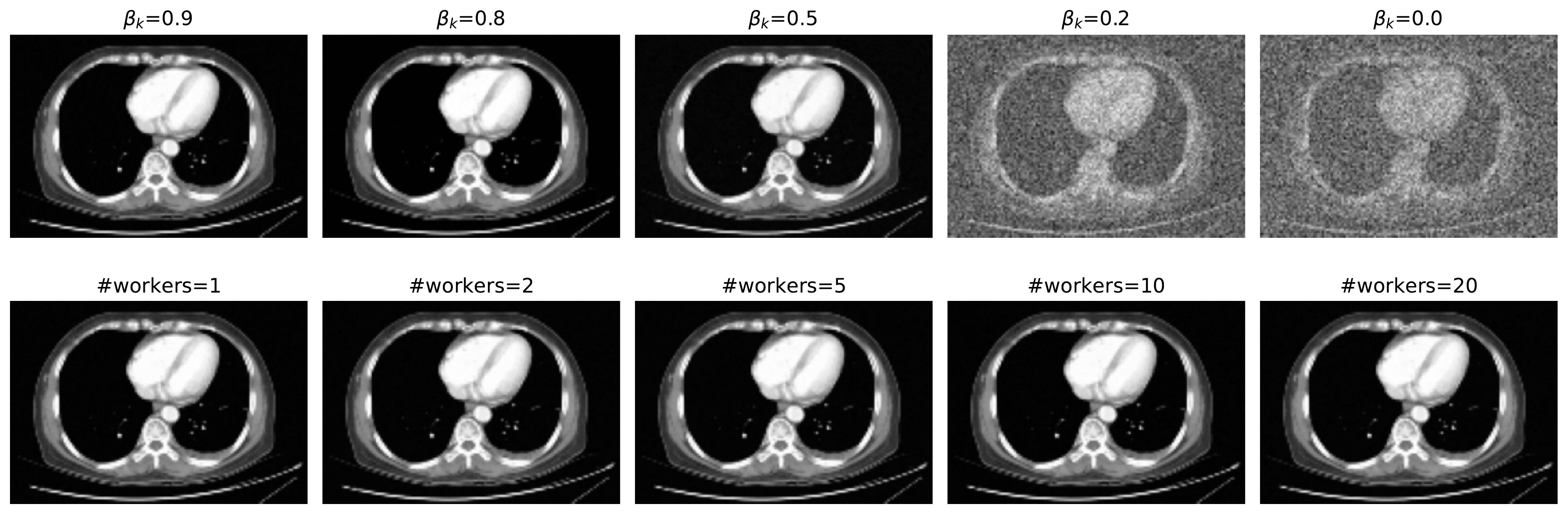}
\caption{Recovered images by Alg.~\ref{alg:async-hvb-sgm} on solving instances of the Phase Retrieval Problem \eqref{eq:pr-prob} with a CT scan image as $\vx^*$ and $m=40,000$. 
Top: non-parallel implementation of Alg.~\ref{alg:async-hvb-sgm} with diminishing $\{\alpha_k\}$ and different choices of $\{\beta_k\}$;  
Bottom: async-parallel implementation of Alg.~\ref{alg:async-hvb-sgm} with diminishing $\{\alpha_k\}$ and $\beta_k=0.8$, and with different numbers of workers.}
\label{fig:radiopaedia_finals}
\end{center} 
\end{figure}  

 \begin{figure}[htbp] 
\begin{center} 
\includegraphics[width=0.7\columnwidth]{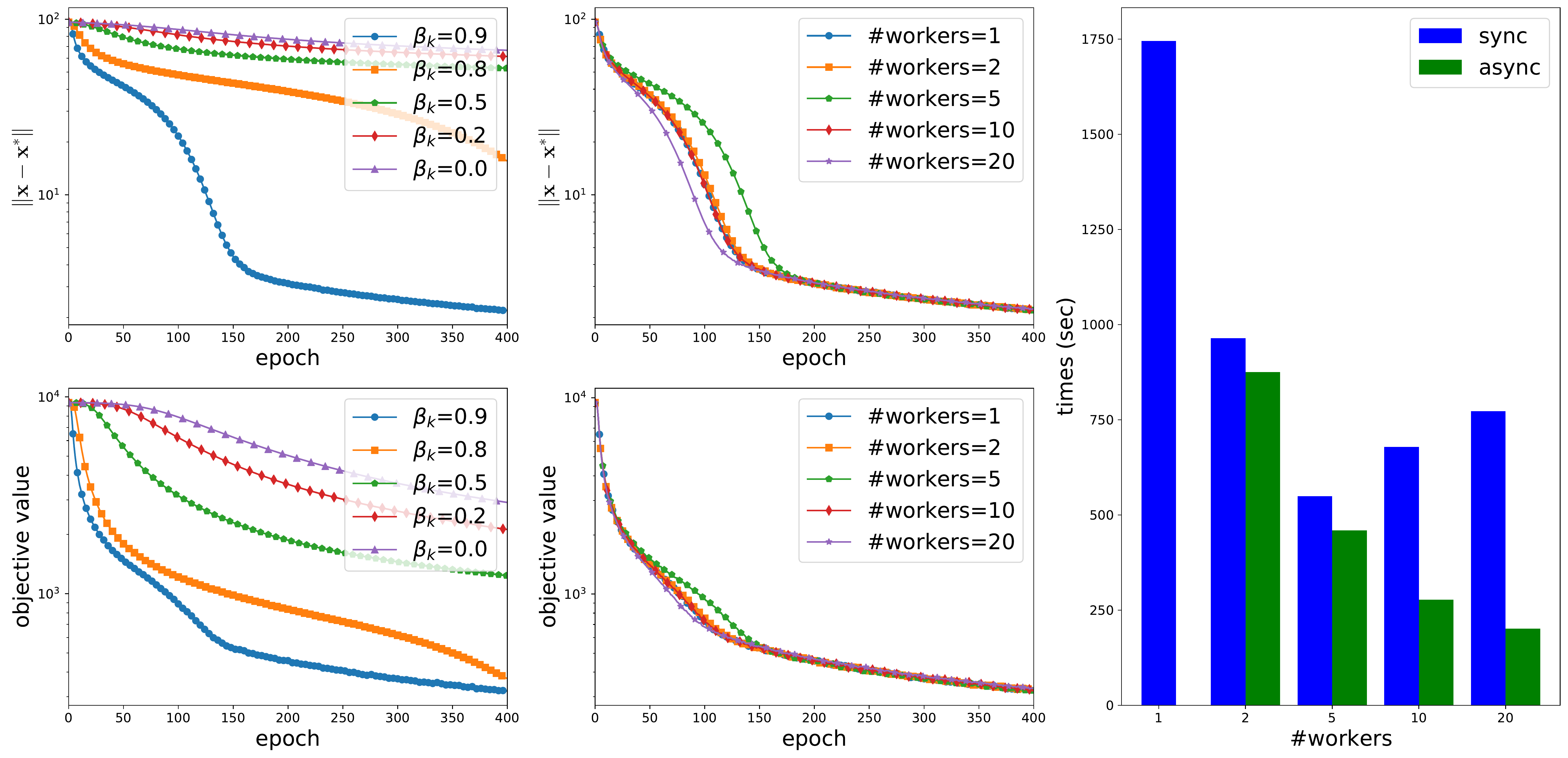}
\caption{Results by Alg.~\ref{alg:async-hvb-sgm} on solving instances of the Phase Retrieval Problem \eqref{eq:pr-prob} with  the crameman image as $\vx^*$ and $m=60,000$. 
Left: non-parallel implementation of Alg.~\ref{alg:async-hvb-sgm} with diminishing $\{\alpha_k\}$ and different choices of $\{\beta_k\}$;  
Middle: async-parallel implementation of Alg.~\ref{alg:async-hvb-sgm} with diminishing $\{\alpha_k\}$ and $\beta_k=0.9$, and with different numbers of workers; Right: running time (in second) of the sync-parallel and async-parallel implementation of Alg.~\ref{alg:async-hvb-sgm} with different numbers of workers.}

\label{fig:cameraman}
\end{center} 
\end{figure}  

 \begin{figure}[htbp] 
\begin{center} 
\includegraphics[width=0.96\columnwidth]{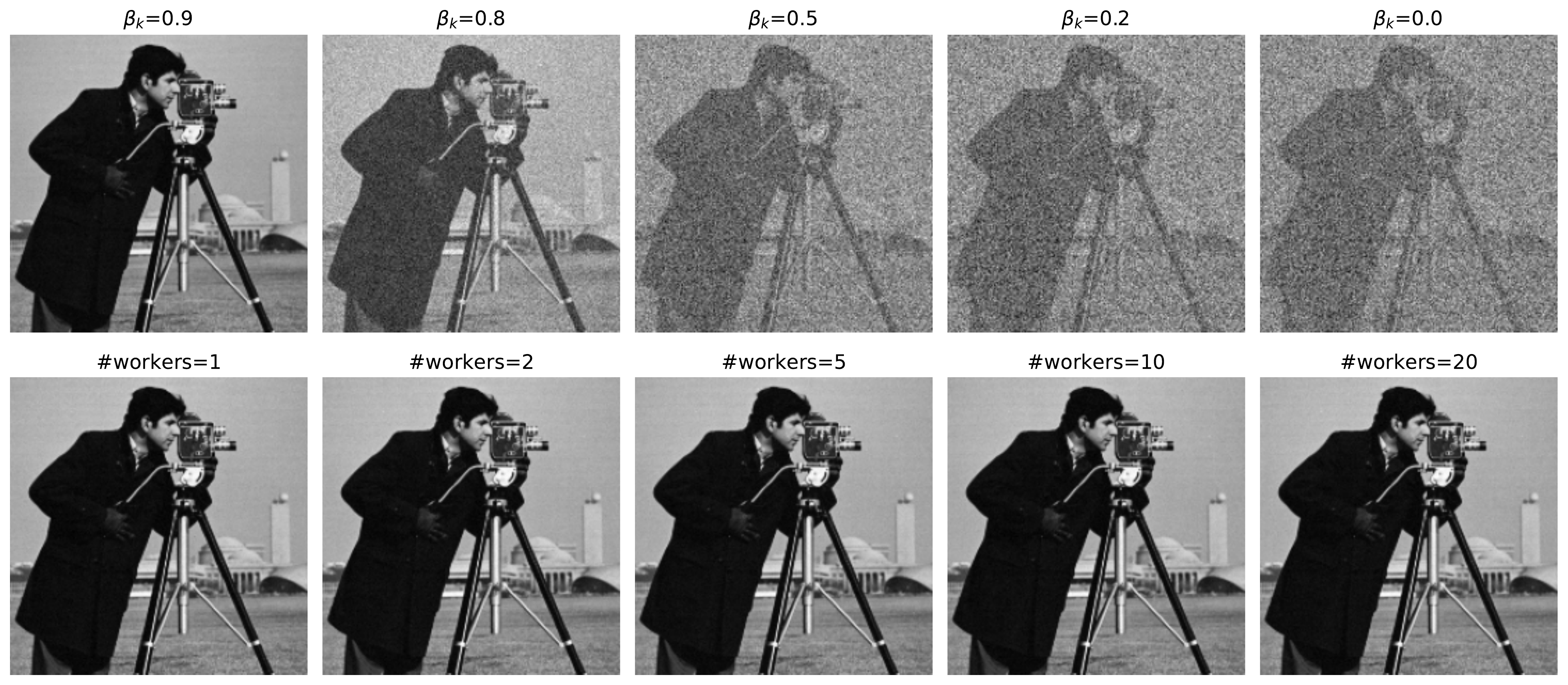}
\caption{Recovered images by Alg.~\ref{alg:async-hvb-sgm} on solving instances of the Phase Retrieval Problem \eqref{eq:pr-prob} with the crameman image as $\vx^*$ and $m=60,000$. 
Top: non-parallel implementation of Alg.~\ref{alg:async-hvb-sgm} with diminishing $\{\alpha_k\}$ and different choices of $\{\beta_k\}$;  
Bottom: async-parallel implementation of Alg.~\ref{alg:async-hvb-sgm} with diminishing $\{\alpha_k\}$ and $\beta_k=0.9$, and with different numbers of workers.}
\label{fig:cameraman_finals}
\end{center} 
\end{figure}  

From the left subfigures in Fig.~\ref{fig:radiopaedia} and Fig.~\ref{fig:cameraman}, we see that the algorithm with a bigger $\beta$ converged faster. After $400$ epochs, the algorithm achieved the lowest objective value and the smallest distance from $\vx^*$ with $\beta_k=0.8, \forall\, k$ for the CT scan image, and with $\beta_k=0.9, \forall\, k$ for the cameraman image.
Alg.~\ref{alg:async-hvb-sgm} recovered the image clearly for the CT scan image with  $\beta_k\equiv \beta\in \{0.9,0.8,0.5\}$ shown in the top subfigures in Fig.~\ref{fig:radiopaedia_finals} and for the cameraman image with $\beta_k=0.9, \forall\, k$ in the top subfigures in Fig.~\ref{fig:cameraman_finals}. The recovered images became clearer as the $\beta$ value increases. 
The middle subfigures in Fig.~\ref{fig:radiopaedia} and Fig.~\ref{fig:cameraman} show the results for the async-parallel version of Alg.~\ref{alg:async-hvb-sgm} with different numbers of workers, and the bottom subfigures in Fig.~\ref{fig:radiopaedia_finals} and Fig.~\ref{fig:cameraman_finals} show the corresponding recovered images. 
The right subfigures in Fig.~\ref{fig:radiopaedia} and Fig.~\ref{fig:cameraman} show the running time of both sync-parallel and async-parallel versions of Alg.~\ref{alg:async-hvb-sgm}. The results show that the convergence speed (in terms of epoch number) of the async-parallel method is rarely affected by the asynchrony (or information delay). In addition, we see that the async-parallel implementation yielded significantly higher parallelization speed-up over the sync-parallel one.

\subsection{Neural network models training}
In this subsection, we trained two neural network models by Alg.~\ref{alg:async-hvb-sgm}. One is LeNet5 on the MNIST dataset \cite{lecun1998gradient} and the other AllCNN \cite{springenberg2014striving} on the Cifar10 dataset \cite{krizhevsky2009learning}. LeNet5 has 2 convolutional,  2 max-pooling, and 3 fully-connected layers. AllCNN has 9 convolutional and 1 avg-pooling layers. The outputs of the two models are re-scaled as probabilities in all classes for each data sample by the softmax function. The estimated probabilities and the true class labels are fed to the negative log likelihood loss function to get the losses.  The objective is to minimize the mean loss over all data samples, which is in the form of \eqref{eq:stoc-prob} with the model weights as $\vx$, the mean loss as $F(\vx)$ and $r(\vx)\equiv0$.
For both trainings, we set 
$\alpha_k = \alpha, \forall\, k$ and selected the best $\alpha$ from $\{0.01, 0.005, 0.001, 0.0005, 0.0001\}$. 
For training LeNet5, we used $\alpha_k=0.001$, and for training Cifar10, we used  $\alpha_k=0.005, \forall\, k$.

 \begin{figure}[htbp] 
\begin{center} 
\includegraphics[width=0.96\columnwidth]{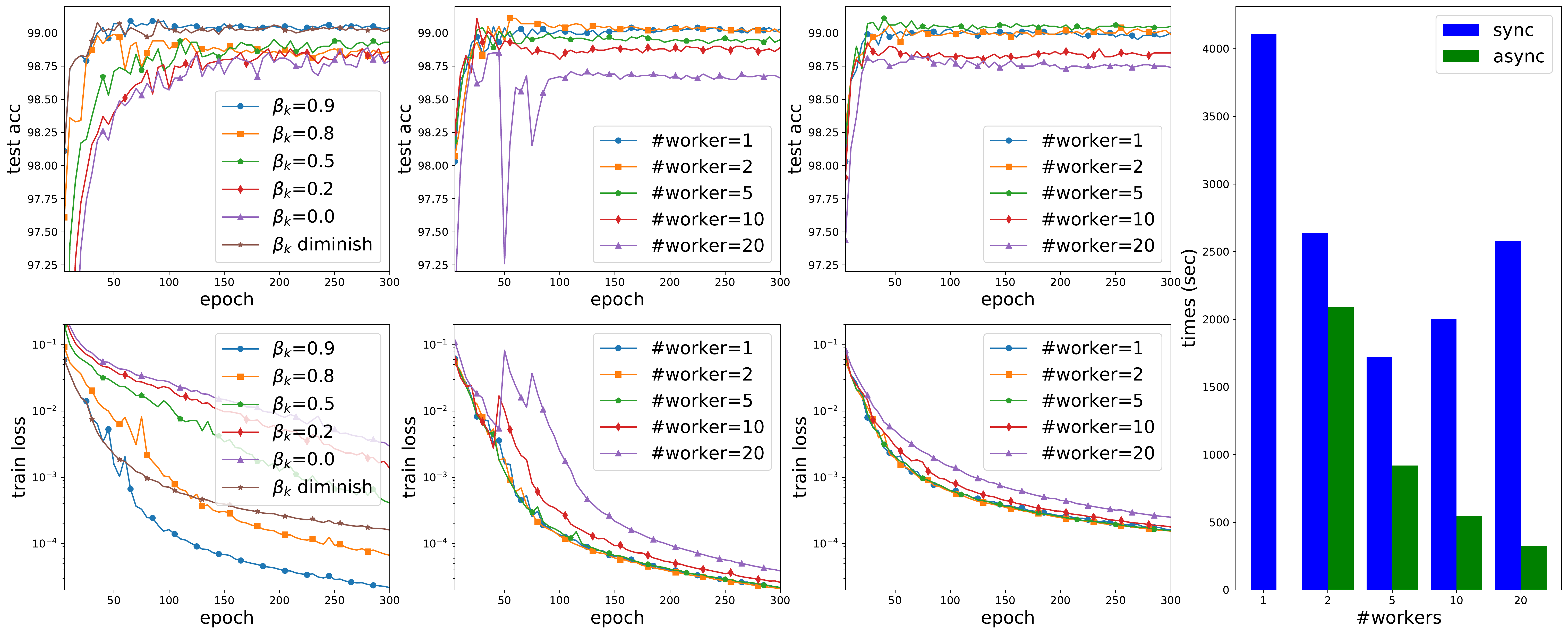}
\caption{Results by Alg.~\ref{alg:async-hvb-sgm} on training LeNet5 on the MNIST dataset. 
First column: non-parallel implementation of Alg.~\ref{alg:async-hvb-sgm} with $\alpha_k=0.001,\forall\, k$ and different choices of $\{\beta_k\}$; Second column:  async-parallel implementation of Alg.~\ref{alg:async-hvb-sgm} 
with $\alpha_k=0.001$ and $\beta_k=0.9,\forall\, k$; Third column: async-parallel implementation of Alg.~\ref{alg:async-hvb-sgm} 
with $\alpha_k=0.001$ and $\beta_k=\min\big\{0.9, \frac{2}{ (e_k+1)^{1/4}} \big\},\forall\, k$; Fourth column: running time (in second) of the sync-parallel and async-parallel implementations of Alg.~\ref{alg:async-hvb-sgm} with different numbers of workers.
}
\label{fig:mnist}
\end{center} 
\end{figure}  
The results of training LeNet5 on the MNIST dataset are shown in Fig.~\ref{fig:mnist}.  In the test, we computed a stochastic subgradient by using $40$ data samples, i.e., the minibatch size was set to $40$. 
We first tested Alg.~\ref{alg:async-hvb-sgm} with $\beta_k=\beta, \forall k$, where $\beta\in \{0, 0.2, 0.5, 0.8, 0.9\}$, or $\beta_k=\min\big\{0.9, \frac{2}{ (e_k+1)^{1/4}} \big\},\forall\, k$. The first column of  Fig.~\ref{fig:mnist} shows that  the algorithm with a bigger $\beta$ gave better results. 
Notice that the algorithm with $\beta_k=0.9$ or $\beta_k=\min\big\{0.9, \frac{2}{ (e_k+1)^{1/4}} \big\},\forall\, k$ give the highest testing accuracy. For these two choices, we ran the async-parallel version of Alg.~\ref{alg:async-hvb-sgm} with different numbers of workers. 
From the results in the second and third columns of Fig.~\ref{fig:mnist}, we see that the asynchrony had negative effect on the behavior of the algorithm, especially when more workers were used.  Nevertheless, the final training loss for all different number of workers is almost the same, and the final testing accuracy by using 10 or 20 workers is slightly lower than that produced by using fewer workers. 
The fourth column compares the running time of the sync-parallel and async-parallel implementations of Alg.~\ref{alg:async-hvb-sgm} with $\beta_k=\min\big\{0.9, \frac{2}{ (e_k+1)^{1/4}} \big\},\forall\, k$. Again, the bars show significantly higher parallelization speed-up by the async-parallel implementation over the sync-parallel one. 

\begin{figure}[htbp] 
\begin{center}
\includegraphics[width=0.7\columnwidth]{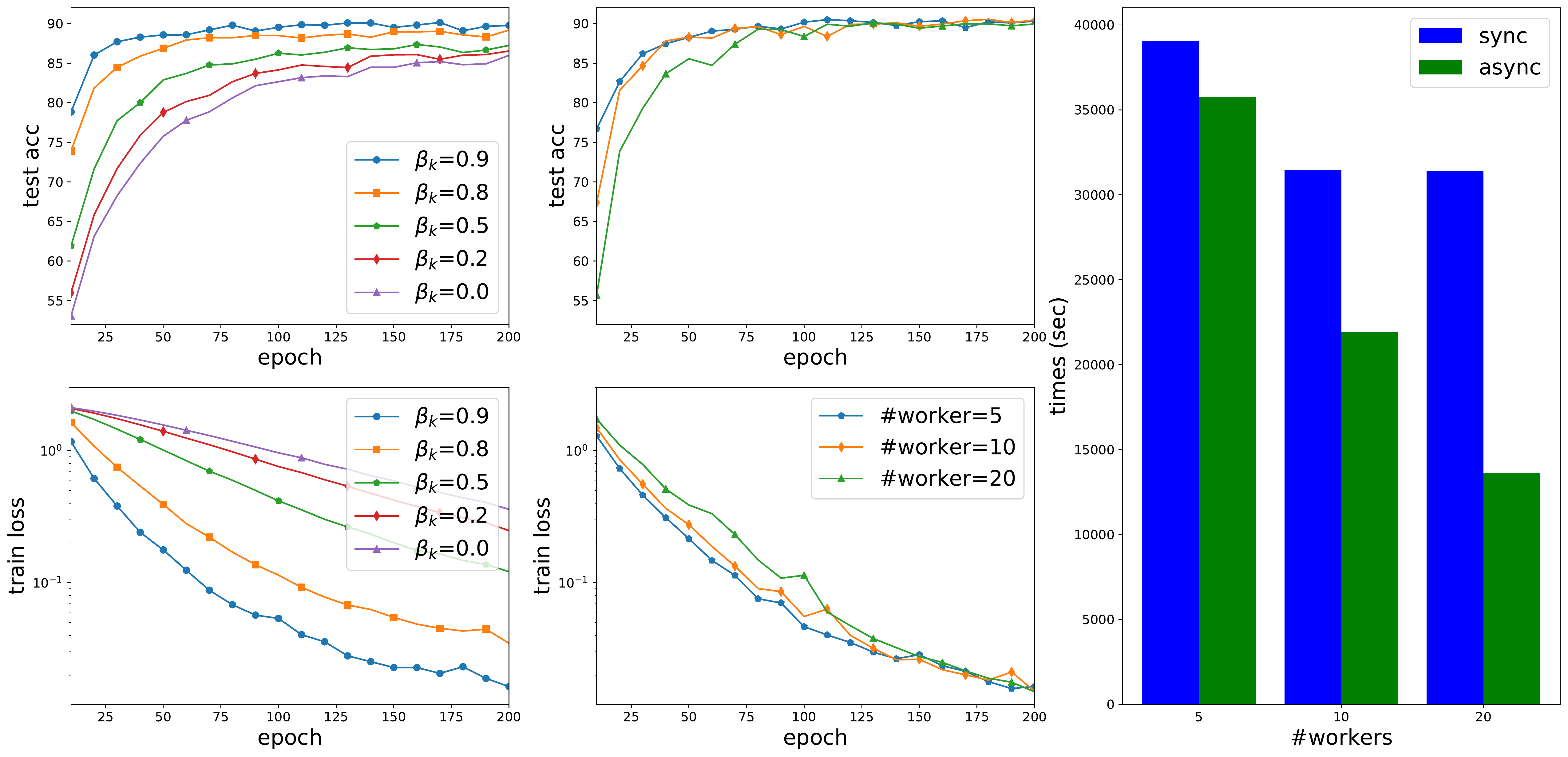}
\caption{ 
Results by Alg.~\ref{alg:async-hvb-sgm} on training AllCNN on  the Cifar10 dataset.
Left:
non-parallel implementation of Alg.~\ref{alg:async-hvb-sgm} with $\alpha_k=0.005$ and different $\{\beta_k\}$; 
Middle: 
async-parallel implementation of Alg.~\ref{alg:async-hvb-sgm} with 
$\alpha_k=0.005, \beta_k=0.9$, and with different numbers of workers; Right: running time (in second) of the sync-parallel and async-parallel implementations of Alg.~\ref{alg:async-hvb-sgm} with different numbers of workers.
}
\label{fig:cifar}
\end{center} 
\end{figure} 
The results of training AllCNN on the Cifar10 dataset are shown in Fig.~\ref{fig:cifar}. In the test, we set 
the minibatch size to $100$ and $\beta_k=\beta, \forall k$, where $\beta\in \{0, 0.2, 0.5, 0.8, 0.9\}$. 
 The left column of  Fig.~\ref{fig:cifar} shows that the algorithm with a bigger $\beta$ gave better results. The choice of $\beta_k=0.9, \forall\, k$ yielded the best results. 
With this choice, we compared the sync-parallel and async-parallel implementations of  Alg.~\ref{alg:async-hvb-sgm}. 
The middle column in Fig.~\ref{fig:cifar} shows the results for the async-parallel version with different numbers of workers. The right column shows the running time of both versions. 
From the results, we see that the convergence speed (in terms of epoch number) of the async-parallel method is almost not affected by the asynchrony. In addition, we see again that the async-parallel implementation yielded higher parallelization speed-up over the sync-parallel one.
 
\subsection{Sparse bilinear logistic regression}
In this subsection, we test Alg.~\ref{alg:async-hvb-sgm} on solving the sparse bilinear logistic regression (BLR) built in \cite{shi2014sparse}.
Let $\{(X_i, y_i)\}_{i=1}^m$  be the training data set 
with each data sample $X_i\in\RR^{s\times t}$ and label $y_i\in\{1,2,...,C\}$ for  $i=1,2,...,m$, where $C$ is the number of classes.  The sparse BLR is modeled as 
\begin{equation}\label{eq:bilinear}
\min_{\mathcal{U},\mathcal{V},\vb} -\frac{1}{m}\sum_{i=1}^m \log \left(\frac{\exp[\tr(U_{y_i}X_iV_{y_i})+b_{y_i}]}{\sum_{j=1}^C \exp[\tr(U_jX_iV_j)+b_j ]}\right) + \lambda (\|\mathcal{U}\|_1+\|\mathcal{V}\|_1+\|\vb\|_1),
\end{equation}
where $\mathcal{U}=(U_1, U_2, ..., U_C), \mathcal{V}=(V_1, V_2, ..., V_C), \vb=(b_1, b_2, ..., b_C)$ with $U_j\in\RR^{p\times s}, V_j\in\RR^{t\times p}, b_j\in\RR$ for $j=1,2,...,C$,    $\|\mathcal{U}\|_1 :=\sum_{j=1}^C\sum_{i=1}^{p}\sum_{l=1}^{s} |(U_j)_{i,l}|$, 
$\lambda\ge0$ is the weight for the sparse regularizer, and $\tr(S):=\sum_{i=1}^p S_{i,i}$ for any matrix $S\in\RR^{p\times p}$.
To solve \eqref{eq:bilinear}, we apply Alg.~\ref{alg:async-hvb-sgm}  with $\vx = (\mathcal{U},\mathcal{V},\vb)$, $F(\vx)$ being the first term in \eqref{eq:bilinear}, 
and $r(\vx)= \lambda (\|\mathcal{U}\|_1+\|\mathcal{V}\|_1+\|\vb\|_1)$. 

 In this test, we used 
the MNIST dataset \cite{lecun1998gradient} and  set the minibatch to 100 while computing a stochastic gradient of $F$. 
To obtain a relatively high accuracy and also relatively cheap computation, we chose $p=5$ and $\lambda=10^{-3}$. 
The learning rate was set to $\alpha_k=\alpha, \forall k$ with $\alpha$ tuned from $\{0.01, 0.005, 0.001, 0.0005, 0.0001\}$. 
To ensure convergence and also satisfactory final testing accuracy for both async-parallel and sync-parallel implementations of Alg.~\ref{alg:async-hvb-sgm}, we set $\alpha = 0.0005$. Note that the sync-parallel version could converge faster in the beginning with a larger $\alpha$ 
but the final testing accuracy and training loss were similar to those produced by using $\alpha=0.0005$.

\begin{figure}[htbp] 
\begin{center} 
\includegraphics[width=0.7\columnwidth]{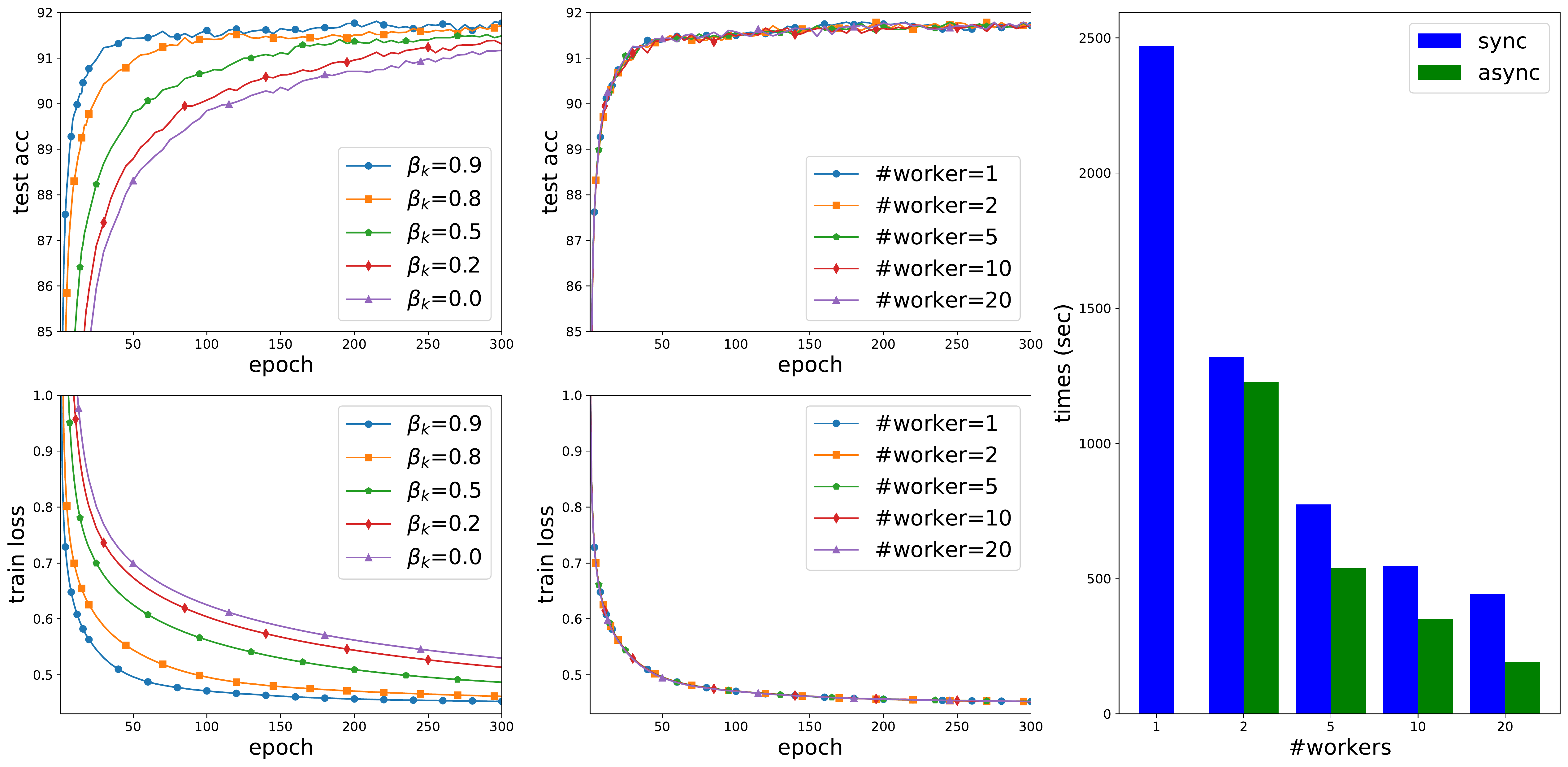} 
\caption{
Results by Alg.~\ref{alg:async-hvb-sgm} to solve the sparse bilinear logistic regression \eqref{eq:bilinear} on the MNIST dataset with $p=5$ and $\lambda=0.001$.  
Left: non-parallel implementation of Alg.~\ref{alg:async-hvb-sgm} with  $\alpha_k=0.0005$ and different $\{\beta_k\}$; Middle: async-parallel implementation of Alg.~\ref{alg:async-hvb-sgm} with 
$\alpha_k=0.0005, \beta_k=0.9,\forall\, k$, and with different numbers of workers; Right: running time (in second) of the sync-parallel and async-parallel implementations of Alg.~\ref{alg:async-hvb-sgm} with different numbers of workers.
}
\label{fig:bilinear_const}
\end{center} 
\end{figure}

The left column of Fig.~\ref{fig:bilinear_const} shows the results 
by Alg.~\ref{alg:async-hvb-sgm} with $\beta_k=\beta\in \{0, 0.2, 0.5, 0.8, 0.9\},  \forall k$.
We see that the algorithm with a bigger $\beta$ converges faster. 
The middle column in Fig.~\ref{fig:bilinear_const} shows the results by the async-parallel implementation of Alg.~\ref{alg:async-hvb-sgm} with $\beta_k=0.9, \forall k$ and with different numbers of workers. The right column shows the running time of both sync-parallel and async-parallel implementations. 
The results show that the convergence speed (in terms of epoch number) of the async-parallel method is almost not affected by the asynchrony. In addition, we see that the async-parallel implementation yielded significantly higher parallelization speed-up over the sync-parallel one.

\section{Conclusions}\label{sec:conclusion}
We have proposed an inertial-accelerated proximal stochastic subgradient method for solving non-convex stochastic optimization. An $O(1/K^{\frac{1}{2}})$ convergence rate result is established for three different problem classes, by the measure of the expected value of the gradient norm square of the objective function or its Moreau envelope, where $K$ is the number of total iterations. The same-order convergence rate can be shown even if the derivative information is outdated in an asynchronous distributed computing environment, provided that the delay (or staleness) of the derivative is in a tolerable range. Numerical experiments on phase retrieval, neural network training, and sparse bilinear logistic regression demonstrate faster convergence by using the inertial-acceleration technique and also the higher parallelization speed-up of the asynchronous computing over the synchronous counterpart.

\section*{Acknowledgements} The authors would like to thank two anonymous referees for their valuable comments and also the associate editor for the suggestion to add tests with images, which help greatly improve the paper. The work of Y. Xu and Y. Yan is partly supported by NSF grant DMS-2053493 and the Rensselaer-IBM AI Research Collaboration, part of the IBM AI Horizons Network. J. Chen is supported in part by DOE Award DE-OE0000910. 

\small 
\appendix

\section{Remaining proofs}
In this section, we provide proofs of the lemmas that are used in our analysis. 

\noindent\textbf{Proof of Lemma~\ref{lem:keybound2}.}
	For ease of notation, we denote $\delta=1-\alpha_{k}\overline{\rho}$ in this proof. We have 
	\begin{align}
		&~\|\vx^{(k+1)}-\widetilde\vx^{(k)} \|^2 \nonumber\\
		=&~ \big\| \prox_{\alpha_{k} r}\big( \vx^{(k)} - \alpha_{k} \vg^{(k)} +\beta_k (\vx^{(k)}-\vx^{(k-1)})\big) 
		- \prox_{\alpha_{k} r}\big(\alpha_{k}\overline{\rho}\vx^{(k)}-\alpha_{k}\widetilde\vv^{(k)}+(1-\alpha_{k}\overline{\rho})\widetilde\vx^{(k)} \big) \big\|^2   \nonumber\\
		\le	&~ \big\|  \delta ( \vx^{(k)}-\widetilde\vx^{(k)}   )- \alpha_{k}(\vg^{(k)} -\widetilde\vv^{(k)})  +\beta_k (\vx^{(k)}-\vx^{(k-1)})  \big\|^2   \label{eq:fix1}\\
		=&~ \big\|  \delta ( \vx^{(k)}-\widetilde\vx^{(k)}   )  +\beta_k (\vx^{(k)}-\vx^{(k-1)})  \big\|^2
		+\alpha_{k}^2  \big\|   \vg^{(k)} -\widetilde\vv^{(k)}  \big\|^2  \\
		&~- 2\alpha_{k}\big\langle \delta ( \vx^{(k)}-\widetilde\vx^{(k)}   )  +\beta_k (\vx^{(k)}-\vx^{(k-1)}),\vg^{(k)} -\widetilde\vv^{(k)} \big\rangle \label{eq:fix1-2}, \nonumber
	\end{align}
	where the first equality is from \eqref{eq:update-x-2} and  \eqref{eq:relate-xtilde}, and the inequality follows from the nonexpansiveness of the proximal mapping.		  
	Taking conditional expectation on $\xi_k$ over the equation in \eqref{eq:fix1-2} gives		  	
	\begin{align}
		&\EE_{\xi_k}  \|\vx^{(k+1)}-\widetilde\vx^{(k)} \|^2 \nonumber\\
		\le&~ \big\|  \delta ( \vx^{(k)}-\widetilde\vx^{(k)}   )  +\beta_k (\vx^{(k)}-\vx^{(k-1)})  \big\|^2
		+\alpha_{k}^2 \EE_{\xi_k}  \big\|   \vg^{(k)} -\widetilde\vv^{(k)}  \big\|^2 \nonumber \\ 
		& ~- 2\alpha_{k}\big\langle \delta ( \vx^{(k)}-\widetilde\vx^{(k)}   )  +\beta_k (\vx^{(k)}-\vx^{(k-1)}),\vv^{(k)} -\widetilde\vv^{(k)} \big\rangle    \nonumber \\
		\le&~ \big(  \delta \| \vx^{(k)}-\widetilde\vx^{(k)}   \|  +\beta_k \|\vx^{(k)}-\vx^{(k-1)}\|  \big)^2
		+4\alpha_{k}^2 M ^2 \nonumber \\
		&~ - 2\alpha_{k}\delta\big\langle    \vx^{(k)}-\widetilde\vx^{(k)}      ,\vv^{(k)} -\widetilde\vv^{(k)} \big\rangle   
		- 2\alpha_{k}\beta_k\big\langle   \vx^{(k)}-\vx^{(k-1)} ,\vv^{(k)} -\widetilde\vv^{(k)} \big\rangle \nonumber \\
		\le& ~      \delta^2 (1+c_k) \| \vx^{(k)}-\widetilde\vx^{(k)}   \|^2  +({\textstyle 1+\frac{1}{c_k}})\beta_k^2 \|\vx^{(k)}-\vx^{(k-1)}\|^2  
		+4\alpha_{k}^2 M ^2  \nonumber \\
		&- 2\alpha_{k}\delta\big\langle  \vx^{(k)}-\widetilde\vx^{(k)}, \vv^{(k)} -\widetilde\vv^{(k)} \big\rangle   
		+  \beta_k^2    \|\vx^{(k)}-\vx^{(k-1)}\|^2+    \alpha_{k}^2\|\vv^{(k)} -\widetilde\vv^{(k)}\|^2\nonumber, 
	\end{align}
	where the second inequality holds by \eqref{eq:subgrad-2ndmoment}, and the third inequality follows from the Young's inequality along with a scalar $c_k>0$.  Now we obtain the desired result by plugging \eqref{eq:bd-crs-term} into the above inequality, bounding $\|\vv^{( k)} -\widetilde\vv^{(k)}\|^2\le 4M ^2$, and noticing
	\begin{equation*}
		\delta^2 (1+c_k) +2\alpha_{k}\delta\rho
		=   1- 2\alpha_{k}(\overline{\rho}-\rho) -\alpha_{k}^2\overline{\rho}(2\rho-\overline{\rho}) +c_k\delta^2  
		\le 1- 2\alpha_{k}(\overline{\rho}-\rho)   +c_k,  
	\end{equation*}
	where the equality holds because $\delta=1-\alpha_{k}\overline{\rho}$, and the inequality follows from $\delta<1$, $c_k>0$, and $\overline{\rho}\le 2\rho$.
\endproof

\noindent\textbf{Proof of Lemma~\ref{lem:nsm-error-term}.}
	Taking conditional expectation on $\tau_k$, we have $\EE_{\tau_k}[F(\vx^{(k-\tau_k)})] = \sum_{j=0}^\tau p_j F(\vx^{(k-j)})$, where we let $\vx^{(k)}=\vx^{(1)},\forall\, k\le 0$. Hence,
	\begin{align*}
	\textstyle	\sum_{k=1}^K  \EE \big[F(\vx^{(k-\tau_k)})\big] = & \textstyle ~\sum_{k=1}^K \sum_{j=0}^\tau p_j \EE\big[F(\vx^{(k-j)})\big] \\
	= & \textstyle ~ \sum_{k=1}^K \sum_{t=k-\tau}^k p_{k-t} \EE\big[F(\vx^{(t)})\big] = \sum_{t=1-\tau}^K\sum_{k=\max\{1,t\}}^{\min\{K,t+\tau\}} p_{k-t} \EE\big[F(\vx^{(t)})\big],
	\end{align*}
	and
	\begin{align}\label{eq:sum-F-term}
		\sum_{k=1}^K  \EE \big[F(\vx^{(k)})-F(\vx^{(k-\tau_k)})\big] = &~\sum_{k=1}^K \EE \big[F(\vx^{(k)})\big] - \sum_{k=1-\tau}^K\sum_{t=\max\{1,k\}}^{\min\{K,k+\tau\}} p_{t-k} \EE\big[F(\vx^{(k)})\big]\cr
		=&~\left(1- \sum_{k=1-\tau}^1\sum_{t=1}^{k+\tau}p_{t-k}\right)F(\vx^{(1)}) + \sum_{k=K-\tau+1}^K\left(1-\sum_{t=k}^K p_{t-k}\right)\EE\big[F(\vx^{(k)})\big]\nonumber\\
		\le &~ \tau\max\big\{0, -F(\vx^{(1)})\big\} + \tau C_F.
	\end{align}
	
In addition, because $\tau_k\le \tau, \forall\, k$, it holds $\sum_{k=1}^K \|\vx^{(k-\tau_k)}-\vx^{(k)}\|^2\le \tau^2 \sum_{k=1}^K \|\vx^{(k-1)}-\vx^{(k)}\|^2$, which together with \eqref{eq:sum-xdiff-sq} gives
\begin{align}\label{eq:sum-xtauk-term}
\textstyle\sum_{k=1}^K \|\vx^{(k-\tau_k)}-\vx^{(k)}\|^2 \le \tau^2 \left(\frac{\alpha}{\gamma\sqrt K} \big(\phi(\vx^{(1)}) - \phi^* \big) +\frac{\alpha^2M ^2}{\gamma^2}\right).
\end{align}
	
	For the last term in $\cE_k$, we use \eqref{eq:subgrad-2ndmoment} and Assumption~\ref{assump:rand-delay} to bound it as follows
	\begin{align*}
	\textstyle	-\big\langle    \vx^{(k)}-\vx^{(k-\tau_k)} ,\vv^{(k)} \big\rangle \le M  \|\vx^{(k)}-\vx^{(k-\tau_k)}\| \le M \sum_{j=1}^\tau \|\vx^{(k+1-j)} -\vx^{(k-j)}\|,
	\end{align*}
	and thus
	\begin{align}\label{eq:sum-error-crs-term}
	\textstyle	\sum_{k=1}^K -\big\langle    \vx^{(k)}-\vx^{(k-\tau_k)} ,\vv^{(k)} \big\rangle \le M \tau \sum_{k=1}^K \|\vx^{(k)}-\vx^{(k-1)}\|.
	\end{align}
	By the Cauchy-Schwarz inequality and Jensen's inequality,  we have
	$$\textstyle \sum_{k=1}^K \EE\|\vx^{(k)}-\vx^{(k-1)}\| \le \sqrt{K}\sqrt{\sum_{k=1}^K \big(\EE\|\vx^{(k)}-\vx^{(k-1)}\|\big)^2}\le \sqrt{K}\sqrt{\sum_{k=1}^K \EE\|\vx^{(k)}-\vx^{(k-1)}\|^2},$$
	which together with \eqref{eq:sum-xdiff-sq} and \eqref{eq:sum-error-crs-term} gives
	\begin{equation}\label{eq:sum-error-crs-term2}
	\textstyle	\sum_{k=1}^K \EE\left[-\big\langle    \vx^{(k)}-\vx^{(k-\tau_k)} ,\vv^{(k)} \big\rangle\right] \le M \tau\sqrt K\sqrt{\frac{\alpha}{\gamma\sqrt K} \big(\phi(\vx^{(1)}) - \phi^* \big) +\frac{\alpha^2M ^2}{\gamma^2}}.
	\end{equation}
	Now we obtain the desired result from \eqref{eq:sum-F-term}, \eqref{eq:sum-xtauk-term}, and \eqref{eq:sum-error-crs-term2}.
\endproof

\noindent\textbf{Proof of Lemma~\ref{lem:neighbor}.}
	When condition~\ref{subgrad-bnd2} of Assumption~\ref{assump:subgrad-bound2} holds, the update in \eqref{eq:update-x-2} indicates that there exists a subgradient $\tilde\nabla r_2(\vx^{(k+1)})$ such that 
	\[\big\langle \vy-\vx^{(k+1)}, \alpha_{k} \tilde\nabla r_2(\vx^{(k+1)})+ \vx^{(k+1)} -\vx^{(k)} + \alpha_{k} \vg^{(k)} -\beta_k (\vx^{(k)}-\vx^{(k-1)})  \big\rangle\ge 0, \textup{ for all } \vy\in X. \]
	Letting $\vy=\vx^{(k)}$ and rearranging terms in the above inequality, we have  
	\begin{equation}\label{eq:cross-term-opt} \|\vx^{(k+1)} -\vx^{(k)}\|^2\le \big\langle \vx^{(k)}-\vx^{(k+1)},  \alpha_{k} (\tilde\nabla r_2(\vx^{(k+1)})+\vg^{(k)} ) -\beta_k (\vx^{(k)}-\vx^{(k-1)})  \big\rangle ,
	\end{equation}
	which together with the Cauchy-Schwarz inequality gives
	\[ \|\vx^{(k+1)} -\vx^{(k)}\| \le \|  \alpha_{k} (\tilde\nabla r_2(\vx^{(k+1)})+\vg^{(k)} ) -\beta_k (\vx^{(k)}-\vx^{(k-1)})\|.\]
	Hence, by the triangle inequality and the Young's inequality, we have for any $c>0$, 
	\begin{align*}
		\|\vx^{(k+1)} -\vx^{(k)} \|^2  
		&\le (  \alpha_{k}\| \tilde\nabla r_2(\vx^{(k+1)})+\vg^{(k)} \|+\beta_k \|\vx^{(k)}-\vx^{(k-1)}\|)^2 \nonumber\\
		& \le \alpha_{k}^2({\textstyle 1+\frac{1}{c}})\| \tilde\nabla r_2(\vx^{(k+1)})+\vg^{(k)} \|^2+\beta_k^2(1+c) \|\vx^{(k)}-\vx^{(k-1)}\|^2  \nonumber\\
		&\le 2\alpha_{k}^2({\textstyle 1+\frac{1}{c}})( \| \tilde\nabla r_2(\vx^{(k+1)})\|^2+\|\vg^{(k)} \|^2) +\beta_k^2(1+c) \|\vx^{(k)}-\vx^{(k-1)}\|^2.		
	\end{align*}
	Take full expectation on both sides of the above inequality and use Assumption~\ref{assump:subgrad-bound} and condition~\ref{subgrad-bnd2} of Assumption~\ref{assump:subgrad-bound2} to obtain 
	\begin{equation*} 
		\EE\|\vx^{(k+1)} -\vx^{(k)} \|^2 
		\le 2\alpha_{k}^2({\textstyle 1+\frac{1}{c}})\left( M_r  ^2+M ^2\right) +\beta_k^2(1+c) \EE\|\vx^{(k)}-\vx^{(k-1)}\|^2.
	\end{equation*}
	Let $c= \frac{1}{2}(1/\tilde\beta^2 - 1)$ and sum up the above inequality over $k=1$ to $K$. We obtain \eqref{eq:ex-grad-bound2} 
	by rearranging terms and using $\vx^{( 0)}=\vx^{( 1)}$. 
\endproof

\noindent\textbf{Proof of Lemma~\ref{lem:nsm-error-term-vary}.} By similar arguments as in the proof to obtain \eqref{eq:sum-F-term}, we have
\begin{align}\label{eq:sum-F-term-vary}
		&~\sum_{k=k_0}^K  \alpha_k \EE \big[F(\vx^{(k)})-F(\vx^{(k-\tau_k)})\big] \cr
		= & ~ \sum_{k=k_0}^K \alpha_k \EE \big[F(\vx^{(k)}) - C_F\big] - \sum_{k=k_0-\tau}^K\sum_{t=\max\{k_0,k\}}^{\min\{K,k+\tau\}} \alpha_t p_{t-k} \EE\big[F(\vx^{(k)})-C_F\big]\cr
		= & ~ -\sum_{k=k_0-\tau}^{k_0-1}\sum_{t=\max\{k_0,k\}}^{\min\{K,k+\tau\}} \alpha_t p_{t-k} \EE\big[F(\vx^{(k)})-C_F\big] + \sum_{k=k_0}^K \left(\alpha_k - \sum_{t=\max\{k_0,k\}}^{\min\{K,k+\tau\}} \alpha_t p_{t-k}\right) \EE\big[F(\vx^{(k)})-C_F\big]\cr
\le & ~ 2\alpha_{k_0} \tau C_F ,
	\end{align}
where the inequality holds by the nonincreasing monotonicity of $\{\alpha_k\}$ and the fact $|F(\vx^{(k)})|\le C_F, \forall\, k$.	

In addition, from the nonincreasing monotonicity of $\{\alpha_k\}$ and $\tau_k\le \tau,\forall\, k$, it holds 
\begin{align*}
\textstyle\sum_{k=k_0}^K \alpha_k \|\vx^{(k-\tau_k)}-\vx^{(k)}\|^2 \le & ~\textstyle\alpha_{k_0} \tau^2 \sum_{k=k_0 - \tau+1}^K \|\vx^{(k-1)}-\vx^{(k)}\|^2 \le \alpha_{k_0} \tau^2 \sum_{k=2}^K \|\vx^{(k-1)}-\vx^{(k)}\|^2.
\end{align*}
Hence, by \eqref{eq:sum-xdiff-sq-vary} and \eqref{eq:ex-grad-bound2}, and the definitions of $C_1$ and $C_2$ in \eqref{eq:def_C12}, we have from the above that
	\begin{equation}\label{eq:sum-xtauk-term-vary}
	\textstyle	\sum_{k=k_0}^K \alpha_k \|\vx^{(k-\tau_k)}-\vx^{(k)}\|^2\le\alpha_{k_0} \tau^2\big(C_1 + C_2\sum_{k=1}^K\alpha_k^2\big)
	\end{equation}
	
Finally, similar to \eqref{eq:sum-error-crs-term}, we have
	\begin{align}\label{eq:sum-error-crs-term-vary}
	\textstyle	\sum_{k=k_0}^K \left[-\alpha_k\big\langle    \vx^{(k)}-\vx^{(k-\tau_k)} ,\vv^{(k)} \big\rangle\right] \le M \tau \sum_{k=k_0}^K \alpha_k \|\vx^{(k)}-\vx^{(k-1)}\|.
	\end{align}	
By the Cauchy-Schwarz inequality and Jensen's inequality, it holds	
$$\sum_{k=k_0}^K \alpha_k \EE\big[\|\vx^{(k)}-\vx^{(k-1)}\|\big] \le \sqrt{\sum_{k=k_0}^K \alpha_k^2} \sqrt{\sum_{k=k_0}^K \EE\big[\|\vx^{(k)}-\vx^{(k-1)}\|^2\big]} \le \sqrt{\sum_{k=k_0}^K \alpha_k^2} \sqrt{C_1 + C_2\sum_{k=1}^K\alpha_k^2},$$
which together with \eqref{eq:sum-error-crs-term-vary} gives
\begin{align}\label{eq:sum-error-crs-term-vary-2}
	\textstyle	\sum_{k=k_0}^K \EE\left[-\alpha_k\big\langle    \vx^{(k)}-\vx^{(k-\tau_k)} ,\vv^{(k)} \big\rangle\right] \le M \tau \sqrt{\sum_{k=k_0}^K \alpha_k^2} \sqrt{C_1 + C_2\sum_{k=1}^K\alpha_k^2}.
	\end{align}
Now \eqref{eq:bound-nsm-error-term-vary} follows from \eqref{eq:sum-F-term-vary}, \eqref{eq:sum-xtauk-term-vary}, and  \eqref{eq:sum-error-crs-term-vary-2}, and also $1-\alpha_k\overline\rho \le 1$ and $\sum_{k=1}^K\alpha_k^2\le \alpha^2(1+\ln K)$.
\endproof

\noindent\textbf{Proof of Lemma~\ref{lem:keybound}.}
	As in the proof of Lemma~\ref{lem:keybound2}, we denote $\delta=1-\alpha_{k}\overline{\rho}$ and take conditional expectation about $\xi_k$ over both sides of \eqref{eq:fix1} to have
	\begin{align}
		\EE_{\xi_k}  \|\vx^{(k+1)}-\widetilde\vx^{(k)} \|^2 
		\le	& ~\EE_{\xi_k} \big\|  \delta ( \vx^{(k)}-\widetilde\vx^{(k)}   )- \alpha_{k}(\vg^{(k)} -\widetilde\vv^{(k)})  +\beta_k (\vx^{(k)}-\vx^{(k-1)})  \big\|^2  \nonumber\\ 
		=& ~\EE_{\xi_k} \big\|  \delta ( \vx^{(k)}-\widetilde\vx^{(k)}   )- \alpha_{k}(\nabla F(\vx^{(k)}) -\nabla F(\widetilde\vx^{(k)}))  +\beta_k (\vx^{(k)}-\vx^{(k-1)}) -\alpha_{k} \vw^{(k)} \big\|^2  \nonumber \\
		=& ~   \big\|  \delta ( \vx^{(k)}-\widetilde\vx^{(k)}   )- \alpha_{k}(\nabla F(\vx^{(k)}) -\nabla F(\widetilde\vx^{(k)}))  +\beta_k (\vx^{(k)}-\vx^{(k-1)})\big\|^2 + \alpha_{k}^2\EE_{\xi_k} \|\vw^{(k)}\|^2   \nonumber \\
		&~-2\alpha_{k} \EE_{\xi_k} \big\langle   \delta ( \vx^{(k)}-\widetilde\vx^{(k)}   )- \alpha_{k}(\nabla F(\vx^{(k)}) -\nabla F(\widetilde\vx^{(k)}))  +\beta_k (\vx^{(k)}-\vx^{(k-1)}), \vw^{(k)}  \big\rangle   \nonumber \\
		\le& ~   \big(  (\delta +\alpha_{k}\rho)\| \vx^{(k)}-\widetilde\vx^{(k)}  \|+  \beta_k \|\vx^{(k)}-\vx^{(k-1)}\|\big)^2 + \alpha_{k}^2 \big(\sigma^2+\rho^2\|\vx^{(k-\tau_k)} - \vx^{(k)}\|^2\big) +\cE , \label{eq:caliE}
	\end{align}
	where we have used Assumption~\ref{smoothness} and Lemma~\ref{lem:g-deviate} to obtain the last inequality, and we denote  $\vw^{(k)}=\vg^{(k)} - \nabla F(\vx^{(k)})$ and
	\[\cE:=-2\alpha_{k}   \big\langle   \delta ( \vx^{(k)}-\widetilde\vx^{(k)}   )- \alpha_{k}(\nabla F(\vx^{(k)}) -\nabla F(\widetilde\vx^{(k)}))  +\beta_k (\vx^{(k)}-\vx^{(k-1)}), \nabla F(\vx^{(k-\tau_k)}) -\nabla F(\vx^{(k)}) \big\rangle.\] 
	Now we apply the Young's inequality to bound the first square term in \eqref{eq:caliE} to obtain 
	\begin{equation}\label{eq:caliE2-0} 
	\begin{aligned}
			\EE_{\xi_k}  \|\vx^{(k+1)}-\widetilde\vx^{(k)}  \|^2 
			\le &~ (1+c_k)  (\delta +\alpha_{k}\rho)^2\| \vx^{(k)}-\widetilde\vx^{(k)}  \|^2 \\
			& ~+ ({\textstyle 1+\frac{1}{c_k} }) \beta_k^2 \|\vx^{(k)}-\vx^{(k-1)}\|^2 
			+   \alpha_{k}^2\sigma^2+\alpha_{k}^2\rho^2\|\vx^{(k-\tau_k)} - \vx^{(k)}\|^2 + \cE,  
	\end{aligned}		 
	\end{equation}		
	where $c_k$ is any positive number. Recall $\delta=1-\alpha_{k}\overline{\rho}$, and thus  
	$$
		(1+c_k)  (\delta +\alpha_{k}\rho)^2=(1+c_k)  \left( 1- \alpha_{k}(\overline{\rho}-\rho)(2-\alpha_{k}(\overline{\rho}-\rho)) \right) 
		\le(1+c_k)  \left( 1- \alpha_{k}(\overline{\rho}-\rho)\right)  
		\le1+c_k- \alpha_{k}(\overline{\rho}-\rho), 
$$
	where the two inequalities follow from $0<\alpha_{k}(\overline{\rho}-\rho)<1$ and $c_k>0$. Hence, \eqref{eq:caliE2-0} implies
	\begin{equation} \label{eq:caliE2}	
	\begin{aligned}
	\textstyle		\EE_{\xi_k}  \|\vx^{(k+1)}-\widetilde\vx^{(k)}  \|^2 
	\le  & \textstyle ~\left(1+c_k- \alpha_{k}(\overline{\rho}-\rho)\right)   \| \vx^{(k)}-\widetilde\vx^{(k)}  \|^2 \\
	& \textstyle ~+(  1+\frac{1}{c_k}  ) \beta_k^2 \|\vx^{(k)}-\vx^{(k-1)}\|^2 
			+   \alpha_{k}^2\sigma^2   +\alpha_{k}^2\rho^2\|\vx^{(k-\tau_k)} - \vx^{(k)}\|^2 +\cE.
			\end{aligned}
	\end{equation}	
	
	Below we bound $\cE$. We have by the triangle inequality and the $\rho$-smoothness of $F$ that
	\begin{align}
		\cE\le&~ 2\alpha_{k}\rho       \big(  (\delta +\alpha_{k}\rho)\| \vx^{(k)}-\widetilde\vx^{(k)}  \|+  \beta_k \|\vx^{(k)}-\vx^{(k-1)}\|\big) \|\vx^{(k-\tau_k)} - \vx^{(k)}\| \nonumber\\
		\le&~  \textstyle \frac{1}{2}\alpha_{k}(\overline{\rho}-\rho)  \| \vx^{(k)}-\widetilde\vx^{(k)}  \|^2 + \frac{2\alpha_{k}\rho^2}{\overline{\rho}-\rho}\|\vx^{(k-\tau_k)} - \vx^{(k)}\|^2   \label{eq:boundcaliE} \\
		& ~+  \beta_k^2 \|\vx^{(k)}-\vx^{(k-1)}\|^2+  \alpha_{k}^2\rho^2 \|\vx^{(k-\tau_k)} - \vx^{(k)}\|^2, \nonumber
	\end{align} 
	where we have used $\delta +\alpha_{k}\rho=1-\alpha_{k}(\overline{\rho}-\rho)<1$ and the Young's inequality twice to obtain the second inequality. Plug \eqref{eq:boundcaliE}  into \eqref{eq:caliE2} and rearrange terms. We obtain  \eqref{eq:keybound} and complete the proof.
\endproof

\begin{lemma}
	\label{lem:vm-2mom}
	Let $\{\vx^{(k)}\}_{k\ge 1} $ and $\{\vg^{(k)}\}_{k\ge 1} $ be generated from Algorithm~\ref{alg:async-hvb-sgm}, and let $\{q_k\}_{k\ge 1} $ be a sequence of constants. Under Assumptions~\ref{assump:unbiased} and \ref{assump:variance}, we have 
	\begin{equation}\label{g-2mom}
		\textstyle	\EE \big\|\sum_{j=1}^k q_j\vg^{(j)}\big\|^2\le  \sum_{l=1}^k q_l\sum_{j=1}^k q_j  u_j  
		+ \sum_{j=1}^k q_j^{2}   u_j  + \sigma^2\sum_{j=1}^k q_j^{2}.
	\end{equation}
\end{lemma}

\noindent\textbf{Proof of Lemma~\ref{lem:vm-2mom}.}
	From \eqref{eq:update-m} and Assumption~\ref{assump:unbiased}, we have 
	\begin{equation}\label{eq:ex-m}
		\textstyle  \sum_{j=1}^k q_j\vg^{(j)}
		=   \sum_{j=1}^k q_j  \nabla f(\vx^{(j-\tau_j)}; \xi_j). 
	\end{equation}
	Taking a total expectation results in	
	\begin{equation}\label{eq:ex-m-mean}
		\textstyle\EE\big[ \sum_{j=1}^k q_j\vg^{(j)}\big]  =    \sum_{j=1}^k q_j \EE [\nabla F(\vx^{(j-\tau_j)} ) 
		]=  \sum_{j=1}^k q_j \EE\vu^{(j)},
	\end{equation}
	which further implies that
	\begin{equation}\label{vm-mean2}
		\textstyle\big\|\EE\big[ \sum_{j=1}^k q_j\vg^{(j)}\big]\big\|^2\le     \sum_{l=1}^k q_l\sum_{j=1}^k q_j \EE\|\vu^{(j)}\|^2 =\sum_{l=1}^k q_l\sum_{j=1}^k q_j  u_j.
	\end{equation}
In \eqref{vm-mean2}, the inequality is obtained by using the triangle inequality, Cauchy-Schwarz inequality, and then Jensen's inequality. 
We further bound the variance as follows:
	\begin{align}
		\textstyle\EE\big\|  \sum_{j=1}^k q_j\vg^{(j)} - \EE\big[ \sum_{j=1}^k q_j\vg^{(j)}\big]\big\|^2 
		=&\textstyle\EE\big\|   \sum_{j=1}^k q_j  \big(  \nabla f(\vx^{(j-\tau_j)}; \xi_j)-\EE\vu^{(j)}\big)  
		\big\|^2 \nonumber\\
		=&\textstyle \sum_{j=1}^k q_j^{2} \EE\big\| \nabla f(\vx^{(j-\tau_j)}; \xi_j)-\EE\vu^{(j)}   
		\big\|^2 \nonumber\\
		=&\textstyle  \sum_{j=1}^k q_j^{2} \big(  \EE\| \nabla f(\vx^{(j-\tau_j)}; \xi_j)-\vu^{(j)}\|^2
		+\EE\| \vu^{(j)}-\EE\vu^{(j)}\|^2 \big)   \nonumber\\   
		\le&\textstyle\sum_{j=1}^k q_j^{2} \big(  \sigma^2 
		+u_j \big). \label{vm-var}
	\end{align}
	Here, the second equality is because the expectations are null for all cross terms $\EE(\vg^{(j)}-\EE\vg^{(j)})^\top(\vg^{(j^\prime)}-\EE\vg^{(j^\prime)})$ with $j>j^\prime$, since each $\xi_j$ is independent from $\{\vx^{(j)},\ldots,\vx^{(1)}\}$ and $\xi_{j^\prime}$; 
	the third equality is because of Assumption~\ref{assump:unbiased};
	the inequality 
	is by Assumption~\ref{assump:variance} and 
	that the variance is upper-bounded by the second moment. Combine \eqref{vm-mean2} and \eqref{vm-var} gives \eqref{g-2mom}.
\endproof

\noindent\textbf{Proof of Lemma~\ref{lem:pi-bnd}.}
	By definition \eqref{def:theta}, we obtain the equality in \eqref{eq:pi-sum1}. 
	Then the first inequality in \eqref{eq:pi-sum1} follows from
	\begin{equation*}
		\textstyle	\sum_{j=1}^{k-1} \pi_{k,j}(t)=  
		\sum_{j=k-\tau_k+1}^{k-1}\frac{1-t^{k-j}}{1-t}
		+  \frac{1-t^{\tau_k}}{1-t}\sum_{j=1}^{k-\tau_k}t^{k-\tau_k-j}
		= \frac{ \tau_k (1-t) -t^{k-\tau_k}(1-t^{\tau_k})}{(1-t)^2}\le \frac{ \tau}{ 1-t },		
	\end{equation*}
	and the second inequality follows from 
	\begin{align*}
		\textstyle\sum_{j=1}^{k-1} \pi_{k,j}^2(t)=&  \textstyle\sum_{j=k-\tau_k+1}^{k-1}\frac{1-2t^{k-j}+t^{2(k-j)}}{(1-t)^2}
		+  \frac{(1-t^{\tau_k})^2}{(1-t)^2}\sum_{j=1}^{k-\tau_k}t^{2(k-\tau_k-j)}\nonumber\\
		=&  \textstyle\frac{ \tau_k (1-t^2) -2(1-t^{\tau_k})(1+t)+(1-t^{2\tau_k})+ (1-t^{\tau_k})^2(1- t^{2(k-\tau_k)})}{(1-t)^2(1-t^2)}
		\le\textstyle\frac{ \tau}{ (1-t)^2}.   
	\end{align*}

\noindent\textbf{Proof of Lemma~\ref{lem:hatgap-dist}.}
	From \eqref{eq:update-m} and Assumption~\ref{assump:unbiased}, we have 
	$
	\vm^{(k)} = \sum_{j=1}^k\beta^{k-j} (1-\beta)\vg^{(j)}
	;
	$	
	apply Lemma~\ref{lem:vm-2mom} with the choice of $q_j=\beta^{k-j} (1-\beta)$ for all $j\in[k]$ to obtain \eqref{vm-2mom} from \eqref{g-2mom}.
	Meanwhile, 
	\begin{align*}
		\vx^{(k-\tau_k)}- \vx^{(k)}&\textstyle=-\sum_{l=0}^{\tau_k-1} \vx^{(k-l)}-\vx^{(k-l-1)}= \sum_{l=0}^{\tau_k-1}\frac{\alpha_{k-l-1}}{1-\beta}\vm^{(k-l-1)}\nonumber\\
		&\textstyle=  \sum_{l=0}^{\tau_k-1}\alpha_{k-l-1} \sum_{j=1}^{k-l-1}\beta^{k-l-j-1}  \vg^{(j)} =   \sum_{j=1}^{k-1} \theta_{k,j}  \vg^{(j)} 
	\end{align*}
	by \eqref{eq:update-x} and \eqref{eq:ex-m}. Apply Lemma~\ref{lem:vm-2mom} with the choice of $q_j= \theta_{k,j}$ for all $j\in[k]$. We have \eqref{hatgap-2mom} from \eqref{g-2mom}.
\endproof

\noindent\textbf{Proof of Lemma~\ref{lem:zz}.}
	From \eqref{defzk}, we have that for $k\geq 1$,
	\begin{align*}
		\vz^{(k+1)}-\vz^{(k)}
		&=\textstyle \frac{1}{1-\beta}(\vx^{(k+1)}-\vx^{(k)})-\frac{\beta}{1-\beta}(\vx^{(k)}-\vx^{(k-1)}) \\
		&=\textstyle -\frac{1}{(1-\beta)^2}\alpha_{k} \vm^{(k)} +\frac{\beta}{(1-\beta)^2} \alpha_{k-1} \vm^{(k-1)}\\
		&=\textstyle \frac{-1}{(1-\beta)^2}\alpha_{k} (\beta\vm^{(k-1)}+(1-\beta)\vg^{(k)}) +\frac{\beta}{(1-\beta)^2} \alpha_{k-1} \vm^{(k-1)}  \\
		&=\textstyle\frac{\beta}{(1-\beta)^2}(\alpha_{k-1} -\alpha_{k} )\vm^{(k-1)}-\frac{\alpha_{k}}{1-\beta} \vg^{(k)} \\
		&=\textstyle\frac{\beta}{1-\beta}(1 - \alpha_{k}/  \alpha_{k-1})\frac{\alpha_{k-1}}{1-\beta}\vm^{(k-1)}-\frac{\alpha_{k}}{1-\beta} \vg^{(k)}.
	\end{align*}
	The second equality is by \eqref{eq:update-x}; the third equality is by \eqref{eq:update-m}.
	The above equality together with \eqref{eq:update-x} gives \eqref{zzandxx}, and \eqref{eq:zxlip} trivially holds by the smoothness of $F$ and \eqref{defzk}.
\endproof

The inequalities in the lemma below are easy to show.

\begin{lemma}\label{lem:sum-alpha-bnd}
	Let $a$ be a positive integer. Then 
		$$\textstyle \sum_{k=1}^{K}\frac{1}{\sqrt{a+k-1}} \ge\int_{a}^{a+K} \frac{1}{\sqrt{x}} dx = 2 (\sqrt{a+K}-\sqrt{a})
	,$$
	$$\textstyle \sum_{k=1}^{K}\frac{1}{a+k-1}\le 1+ \int_{a}^{a+K-1} \frac{1}{x}dx = 1+\ln\frac{a+K-1}{a}
	.$$
\end{lemma}

\noindent\textbf{Proof of Corollary~\ref{cor:variant}.}
	With $\alpha_{k}=\alpha/\sqrt{a+k-1}, \forall\, k\ge1$, \eqref{eq:vanish} holds if and only if 
	\begin{equation}\label{eq:equiv-alpha-cond}
	\textstyle	\frac{\alpha }{2(1-\beta)\sqrt{a+k-1}}\ge (1-\sqrt{a+k-2}/\sqrt{a+k-1})^2=\frac{1}{((\sqrt{a+k-2}+\sqrt{a+k-1})\sqrt{a+k-1})^2}. 
	\end{equation}
	Notice $\frac{1}{ (\sqrt{a+1}+\sqrt{a})^2}\le\frac{1}{4a}$, and thus $a\sqrt{a+1} \ge \frac{ 1-\beta }{2\alpha}$ indicates $ \alpha\ge \frac{2(1-\beta)}{\sqrt{a+1}(\sqrt{a+1}+\sqrt{a})^2} $, which further implies the inequality in \eqref{eq:equiv-alpha-cond} for all $k\ge2$. 
	%
	Moreover,	when \eqref{cor:tau-cond-22} holds, it is not difficult to verify that the two inequalities in \eqref{eq:twohalfs} are true, so we have \eqref{eq:main-cond} and thus \eqref{eq:main-ncv} from Theorem~\ref{thm:ncv}. 
	
	Below we simplify the inequality in \eqref{eq:main-ncv} for the setting of $\alpha_k$.	First, 
	\begin{align*} 
		\textstyle\sum_{k=1}^K \alpha_{k} \alpha_{\max\{k-\tau_k,1\}}^2 \le &~ \textstyle \sum_{k=1}^K \frac{\alpha^3}{\sqrt{a+k-1}(a+k-1-\tau)}
		\le  \sum_{k=1}^K \frac{2\alpha^3}{\sqrt{a+k-1}(a+k-1)}  \\ 
		 \le & ~  \textstyle 2\alpha^3\left( \frac{1}{a\sqrt{a}}  + \int_{a}^{a+K-1} \frac{1}{x\sqrt{x}}dx  \right) 
		\le      \frac{2\alpha^3(1+2a)}{ a\sqrt{a}} ;
\end{align*} 
	second, by Lemma~\ref{lem:sum-alpha-bnd},
	$${\textstyle \sum_{k=1}^{K}\alpha_{k}= \sum_{k=1}^{K}\frac{\alpha}{\sqrt{a+k-1}} \ge 2\alpha (\sqrt{a+K}-\sqrt{a}),	
	\text{ and }
	\sum_{k=1}^{K}\alpha_k^2 = \sum_{k=1}^{K}\frac{\alpha^2}{a+k-1}\le \alpha^2 (1+\ln\frac{a+K-1}{a}).}
	$$  
	Substituting the above three inequalities into \eqref{eq:main-ncv} gives \eqref{cor:converge2}.
\endproof


\section{Proof of Theorem~\ref{thm:rate-prox-delay-vary}}\label{varying2}
The key of the proof is to bound $\sum_k\EE[\|\vx^{(k)} - \vx^{(k+1)}\|^2]$ while using Theorem~\ref{thm:prox-async}.
First, similar to \eqref{eq:cross-term-opt}, we have
\begin{equation}\label{eq:cross-term-opt-comp} \|\vx^{(k+1)} -\vx^{(k)}\|^2\le \big\langle \vx^{(k)}-\vx^{(k+1)},  \alpha_{k} (\tilde\nabla r(\vx^{(k+1)})+\vg^{(k)} ) -\beta_k (\vx^{(k)}-\vx^{(k-1)})  \big\rangle ,
\end{equation}
where $\tilde\nabla r(\vx^{(k+1)})$ is a subgradient of $r$ at $\vx^{(k+1)}$. By the convexity of $r$, it holds
\begin{equation}\label{eq:cross-term-opt-comp-r}
\big\langle \vx^{(k)}-\vx^{(k+1)}, \tilde\nabla r(\vx^{(k+1)})\big\rangle \le r(\vx^{(k)}) - r(\vx^{(k+1)}).
\end{equation}
In addition, from the $\rho$-smoothness of $F$ and the Young's inequality, we have
\begin{align}\label{eq:cross-term-opt-comp-g}
\big\langle \vx^{(k)}-\vx^{(k+1)}, \vg^{(k)} \big\rangle = & ~\big\langle \vx^{(k)}-\vx^{(k+1)}, \nabla F(\vx^{(k)}) +  \vg^{(k)} - \nabla F(\vx^{(k)})\big\rangle\nonumber\\
& ~ \hspace{-1.5cm}\le  \textstyle F(\vx^{(k)}) - F(\vx^{(k+1)}) + \frac{\rho}{2}\|\vx^{(k)}-\vx^{(k+1)}\|^2 + \frac{1}{4\alpha_k}\|\vx^{(k)}-\vx^{(k+1)}\|^2 + \alpha_k \|\vg^{(k)} - \nabla F(\vx^{(k)})\|^2,
\end{align}
and
\begin{equation}\label{eq:cross-term-opt-comp-beta}
\textstyle \big\langle \vx^{(k)}-\vx^{(k+1)}, -\beta_k (\vx^{(k)}-\vx^{(k-1)})  \big\rangle\le  \frac{1}{4}\|\vx^{(k)}-\vx^{(k+1)}\|^2 + \beta_k^2 \|\vx^{(k)}-\vx^{(k-1)}\|^2
\end{equation}
Plugging \eqref{eq:cross-term-opt-comp-r}, \eqref{eq:cross-term-opt-comp-g} and \eqref{eq:cross-term-opt-comp-beta} into \eqref{eq:cross-term-opt-comp} and rearranging terms  yield
\begin{equation}\label{eq:cross-term-opt-comp-rg}
\textstyle \frac{1}{2}(1-\alpha_k \rho) \|\vx^{(k+1)} -\vx^{(k)}\|^2\le \alpha_k\big(\phi(\vx^{(k)}) - \phi(\vx^{(k+1)})\big) + \alpha_k^2 \|\vg^{(k)} - \nabla F(\vx^{(k)})\|^2 + \beta_k^2 \|\vx^{(k)}-\vx^{(k-1)}\|^2.
\end{equation}
Moreover, by Assumptions~\ref{assump:variance} and \ref{assump:bound-tau} and the $\rho$-smoothness of $F$, we have
\begin{align}\label{eq:variance-term-comp}
\EE\big[\|\vg^{(k)} - \nabla F(\vx^{(k)})\|^2\big] \le & ~2\EE\big[\|\vg^{(k)} - \nabla F(\vx^{(k-\tau_k)})\|^2\big] + 2\EE\big[\|\nabla F(\vx^{(k-\tau_k)})- \nabla F(\vx^{(k)})\|^2\big]\cr
&\hspace{-1cm}\le  \textstyle 2\sigma^2 + 2\rho^2\EE\big[\|\vx^{(k-\tau_k)}- \vx^{(k)}\|^2\big] \le 2\sigma^2 + 2\tau\rho^2\sum_{j=1}^\tau\EE\big[\|\vx^{(k-j)}- \vx^{(k-j+1)}\|^2\big].
\end{align}

Now taking full expectation on \eqref{eq:cross-term-opt-comp-rg}, substituting \eqref{eq:variance-term-comp} there, and summing over $k=1$ to $K$, we obtain by rearranging terms that
\begin{equation}\label{eq:cross-term-opt-comp-sum}
\begin{aligned}
& ~\textstyle \sum_{k=1}^K\frac{1}{2}(1-\alpha_k \rho - \beta_{k+1}^2) \EE\big[ \|\vx^{(k+1)} -\vx^{(k)}\|^2\big] \\
\le & ~ \textstyle \sum_{k=1}^K \alpha_k\EE\big(\phi(\vx^{(k)}) - \phi(\vx^{(k+1)})\big) + 2\sigma^2 \sum_{k=1}^K \alpha_k^2 + 2\tau\rho^2 \sum_{k=1}^K \alpha_k^2 \sum_{j=1}^\tau\EE\big[\|\vx^{(k-j)}- \vx^{(k-j+1)}\|^2\big],
\end{aligned}
\end{equation}
where we have used $\vx^{(0)}=\vx^{(1)}$. Since $\alpha_k$ is nonincreasing, we have
\begin{align*}
\textstyle \sum_{k=1}^K \alpha_k^2 \sum_{j=1}^\tau\EE\big[\|\vx^{(k-j)}- \vx^{(k-j+1)}\|^2\big] \le \tau  \sum_{k=1}^K \alpha_k^2 \EE\big[ \|\vx^{(k)} -\vx^{(k-1)}\|^2\big],
\end{align*}
which substituted into \eqref{eq:cross-term-opt-comp-sum} and together with \eqref{eq:bd-sum-phi} gives
\begin{equation}\label{eq:cross-term-opt-comp-sum-2}
 \textstyle \sum_{k=1}^K\frac{1}{2}(1-\alpha_k \rho - \beta_{k+1}^2 - 2\tau^2\rho^2\alpha_{k+1}^2) \EE\big[ \|\vx^{(k+1)} -\vx^{(k)}\|^2\big] 
\le \textstyle 2 \alpha_1C_\phi + 2\sigma^2 \sum_{k=1}^K \alpha_k^2.
\end{equation}
By the choice of parameters and the definition of $\tilde\gamma$ in \eqref{eq:require-para-comp}, we have from \eqref{eq:cross-term-opt-comp-sum-2} and Lemma~\ref{lem:sum-alpha-bnd} that
\begin{equation}\label{eq:cross-term-opt-comp-sum-3}
 \textstyle \sum_{k=1}^K \EE\big[ \|\vx^{(k+1)} -\vx^{(k)}\|^2\big] 
\le \textstyle \frac{2}{\tilde\gamma}\left( \alpha_1C_\phi + \sigma^2 \alpha^2 (1+\ln\frac{a+K-1}{a}) \right).
\end{equation}
Notice $\big( 2+\frac{4}{\alpha_{k}(\overline{\rho}-\rho)} \big) \beta_k^2\le 2\tilde\beta^2 + \frac{4\beta^2}{\alpha(\overline{\rho}-\rho)}$ and $\alpha_{k}^2+\frac{ \alpha_{k} }{\overline{\rho}-\rho}\le \frac{\alpha^2}{a}+ \frac{ \alpha }{\sqrt{a}(\overline{\rho}-\rho)}, \forall\, k\ge 1$. Therefore, 
\begin{align}\label{eq:bd-x-diff0-prox-sum3-vary}
&~\textstyle\overline{\rho}\sum_{k=1}^{K}\big(  2+\frac{4}{\alpha_{k}(\overline{\rho}-\rho)}  \big) \beta_k^2 \EE\|\vx^{(k)}-\vx^{(k-1)}\|^2 +\frac{\overline{\rho}\rho^2}{2}\sum_{k=1}^{K}\big(   \alpha_{k}^2+\frac{ \alpha_{k} }{\overline{\rho}-\rho} \big) \EE\|\vx^{(k-\tau_k)} - \vx^{(k)}\|^2\cr
\le &~\textstyle \Big( \overline{\rho}\big( 2\tilde\beta^2 + \frac{4\beta^2}{\alpha(\overline{\rho}-\rho)} \big) + \frac{\tau^2\overline{\rho}\rho^2}{2} \big(\frac{\alpha^2}{a}+ \frac{ \alpha }{\sqrt{a}(\overline{\rho}-\rho)}\big)  \Big) \sum_{k=1}^K\EE\|\vx^{(k)}-\vx^{(k-1)}\|^2\cr
\le &~\textstyle \Big( \overline{\rho}\big( 2\tilde\beta^2 + \frac{4\beta^2}{\alpha(\overline{\rho}-\rho)} \big) + \frac{\tau^2\overline{\rho}\rho^2}{2} \big(\frac{\alpha^2}{a}+ \frac{ \alpha }{\sqrt{a}(\overline{\rho}-\rho)}\big)  \Big) \frac{2}{\tilde\gamma}\left( \alpha_1C_\phi + \sigma^2 \alpha^2  (1+\ln\frac{a+K-1}{a}) \right).
\end{align}
Now plug \eqref{eq:bd-x-diff0-prox-sum3-vary} and the choice of $\{\alpha_k\}$ into \eqref{eq:prox-async} to obtain the desired result.
\endproof

\bibliographystyle{abbrv}
\bibliography{optim}

\end{document}